\documentclass{amsart}
\usepackage{amsmath,amsthm}
\usepackage{amssymb}
\usepackage{enumerate}
\usepackage{todonotes}
\usepackage{graphicx}
\usepackage{cancel}
\usepackage{xcolor}
\usepackage[arrow, matrix, curve, cmtip]{xy}
\usepackage{tikz}
\usetikzlibrary{shapes,snakes}
\usepackage{geometry}
\usepackage[shortlabels]{enumitem}
\usepackage{bibentry}
\usepackage[style=alphabetic,
backend=bibtex,doi=false,isbn=false,url=false,eprint=false]{biblatex}

\newtheorem{thm}{Theorem}[section]
\newtheorem{cor}[thm]{Corollary}
\newtheorem{lem}[thm]{Lemma}
\newtheorem{prop}[thm]{Proposition}

\newtheorem{clm}[thm]{Claim}
\newtheorem{quest}[thm]{Question}

\newtheorem*{theorem}{Theorem}

\theoremstyle{definition}
\newtheorem{defin}[thm]{Definition}

\newtheorem{exa}[thm]{Example}
\newtheorem{obs}[thm]{Observation}
\newtheorem{fac}[thm]{Fact}

\newtheorem*{unthm}{Theorem}

\tikzset{
  double arrow/.style args={#1 colored by #2 and #3}{
    -stealth,line width=#1,#2, 
    postaction={draw,-stealth,#3,line width=(#1)/3,
                shorten <=(#1)/3,shorten >=2*(#1)/3}, 
  }
}

\tikzset{
	MyPersp/.style={scale=1.1,x={(1.1cm,0cm)},y={(0.8cm,0.47cm)},
    z={(0cm,1cm)}},
	MyPoints/.style={fill=white,draw=black,thick}
		}
		
\bibliography{omegaalephs}

\title[The first omega alephs]{The first omega alephs: \\ from simplices to trees of trees to higher walks}
\author{Jeffrey Bergfalk}
\address{Centro de Ciencas Matem\'{a}ticas\\
UNAM\\
Xangari, Morelia, Michoac\'{a}n\\
58089, M\'{e}xico}
\email{jeffrey@matmor.unam.mx}
\begin{document}

\begin{abstract} The point of departure for the present work is Barry Mitchell's 1972 theorem that the cohomological dimension of $\aleph_n$ is $n+1$. We record a new proof and mild strengthening of this theorem; our more fundamental aim, though, is some clarification of the higher-dimensional infinitary combinatorics lying at its core. In the course of this work, we describe simplicial characterizations of the ordinals $\omega_n$, higher-dimensional generalizations of coherent Aronszajn trees, bases for critical inverse systems over large index sets, nontrivial $n$-coherent families of functions, and higher-dimensional generalizations of portions of Todorcevic's walks technique. These constructions and arguments are undertaken entirely within a $\mathsf{ZFC}$ framework; at their heart is a simple, finitely iterable technique of compounding $C$-sequences.
\end{abstract}

\subjclass[2020]{03E05, 18G20}

\keywords{cohomological dimension, simplicial complex, characterizing cardinals, projective dimension, coherent Aronszajn tree, walks on ordinals, higher walks}

\thanks{This research was partly supported by the NSF grant DMS-1600635.}

\maketitle

\tableofcontents

\section{Introduction}
In 1972, building on the work of Barbara Osofsky and others, Barry Mitchell showed that for each nonnegative integer $n$ the cohomological dimension of the cardinal $\aleph_n$ is $n+1$ \cite{rings}.
This is data that the set-theoretic community --- i.e., the mathematical community most foundationally interested in the transfinite --- has yet to fully exploit or absorb. This fact, in turn, is all the more striking when we consider that Mitchell's results, properly understood, fall squarely within the long-active area of set-theoretic research into \emph{incompactness phenomena}, that is, into behaviors of size-$\aleph_\alpha$ structures which sharply differ from those of their smaller substructures. Mitchell's results attest the existence of a family of $(n+1)$-dimensional incompactness principles, each of which first holds at the cardinal $\aleph_n$. 

There have been several reasons for this neglect. First among these is the abstract and nonconstructive nature of Mitchell's original proof: in essence, Mitchell adapted to the setting of functor categories an inductive sequence of arguments by contradiction which had been pioneered in ring-theoretic contexts, where the homological significance of $\aleph_n$ had first been perceived (see \cite{osofsubscript}). Close reading of this proof, in other words, has probably tended to reinforce a perception of Mitchell's result as fundamentally algebraic in import, difficult and possibly pointless to disentangle from its original framework. (For the reader's convenience we sketch both Mitchell's original argument and its broader historical contexts in an appendix.)

Secondly, it would be another decade before the critical template for what Mitchell's result \emph{could in fact be about} would emerge within the field of set theory. This would be Todorcevic's \emph{method of minimal walks} \cite{todpairs}. As we will argue below, these comprise the essential content of the $n=1$ case of Mitchell's theorem. This recognition, in turn, situates the now-classical walks apparatus as only the first in a family of higher-dimensional analogues. Extracting these higher-dimensional analogues from Mitchell's theorem is one main object of our work below, and we conclude by describing their basic form. These \emph{higher walks} are almost certain to be of independent interest, and indeed, their fuller analysis, quite apart from this their original setting, is a focus of ongoing work.

Our argument in fact amounts to a new proof and slight strengthening of Mitchell's theorem: it upgrades the characterization of the key inverse system in its original proof from \emph{projective} to \emph{free}, and derives from each (suitably conditioned) \emph{$C$-sequence on $\omega_n$} a canonical witness to the $n^{\text{th}}$ instance of the theorem. As hinted above, when $n=1$ this witness is little other than the associated system of walks on countable ordinals, algebraically recast; more generally, all these witnesses are, in a suitable sense, recursive on the input of a $C$-sequence, in strong contrast to the nonconstructive arguments alluded to above. The core of our strengthened version of the  theorem also admits the following pithy reformulation:
\begin{theorem}
Let $n$ be a positive integer. Then $\omega_n$ is the least ordinal supporting no $n$-dimensional tail-acyclic simplicial complex.
\end{theorem}
\noindent A simplicial complex $B$ whose vertices are the elements of an ordinal $\gamma$ is \emph{tail-acyclic} if its restrictions to tails $[\beta,\gamma)$ of $\gamma$ are each acyclic, i.e., if the reduced simplicial homology groups $\tilde{\mathrm{H}}^{\Delta}_n$ of those restrictions all vanish, as we will explain in greater detail below. So concise a characterization of the ordinals $\omega_n$ in terms of dimensional constraints on objects of a geometric flavor is suggestive in its own right; indeed, below, this theorem will initiate a series of related recognitions in which \emph{dimension}, in cohomological and even strikingly spatial senses, emerges as a fundamental motif and structuring principle of the $\mathsf{ZFC}$ combinatorics of $\omega_n$ $(n\in\omega)$. The spatial flavor of these recognitions is reflected in the number of  accompanying figures, later sections even assuming something of the character of a ``picture book.'' These pictures are meant to bolster intuitions in the face of the accruing coordinates and abstractions that higher-dimensional combinatorics tend to entail; they foreground frameworks within which not only walks but also related classical objects like coherent Aronszajn trees appear as only the first, initiating instance in a family of multidimensional generalizations.

We now describe this paper's organization. As the above account might suggest, set theorists do form this work's primary intended audience, but only by a gentle margin. We naturally hope for a more diverse readership and accordingly adopt fairly minimal assumptions about our readers' backgrounds in either set theory or homological algebra; this is feasible because our arguments are so generally and deliberately elementary in nature. Hence in Section \ref{background} we review what basics we will need about walks and $C$-sequences, simplicial complexes, and free and projective inverse systems, relating the latter with Aronszajn trees; we also introduce a notion of \emph{internal walks} and conclude the section with a statement of Mitchell's theorem. In Section \ref{goodsection} we define \emph{tail-acyclic simplicial complexes} and state our strengthened variant of Mitchell's theorem, as well as its translation to a statement about the existence of bases of inverse systems of abelian groups. Section \ref{bases} introduces the idea at the heart of all our constructions, namely a simple, finitely iterable technique of \emph{compounding $C$-sequences}, and applies this technique to describe bases of inverse systems with $n$-dimensional generating sets. Section \ref{afterbases} derives a family of $n$-coherent functions $\mathtt{f}_n:[\omega_n]^{n+2}\to\mathbb{Z}$ from these bases and reduces the proof of Mitchell's theorem (and its strengthening) to showing these functions to be nontrivial; this nontriviality is then the argument of Section \ref{deromega1}. We endeavor in the remaining sections to reconnect these functions with more familiar mathematical objects. We begin Section \ref{treessection}, for example, by deducing from the functions $\mathtt{f}_n$ that $\omega_n$ is the least ordinal with a nonzero constant-sheaf $n^{\mathrm{th}}$ \v{C}ech cohomology group. We describe as well how these functions may be viewed as $n$-dimensional generalizations of coherent Aronszajn trees, the so-called ``trees of trees'' of the present work's title. We turn in Section \ref{highertraces} to the $n$-dimensional walks lying at the heart of the functions $\mathtt{f}_n$, treating the representative $n=1$ and $n=2$ cases in some detail. In Section \ref{conclusion}, we conclude with several open questions alongside a brief survey of other $n$-dimensional phenomena arising among the cardinals $\aleph_n$. In part for the further light it sheds on these phenomena, a sketch of Mitchell's original argument of his theorem is included in an appendix, along with some discussion of its antecedents. The proof of one theorem from the main text, being laborious, is deferred to the appendix as well, along with supplementary details of the proof of another.

One final word before beginning: Sections \ref{deromega1} and \ref{treessection} (and to a lesser degree Sections \ref{afterbases} and \ref{bases}) demand markedly more of the reader than the others; there simply is no ready language for the structures $\mathtt{f}_n$ lying at their core. First-time readers, accordingly, should feel free to skim these sections, reading them even mainly through their figures and captions, noting results like Theorem \ref{cechtheorem} and Corollary \ref{maybemylastlabel} along the way, and more conventionally re-engaging the text in Section \ref{highertraces}. Most essential more generally for any number of non-linear approaches to the text will be Definition \ref{cmpdlddrs} and the notational conventions listed at the end of Section \ref{3210}.

\section{Background material and conventions}\label{background}

\subsection{Walks and $C$-sequences}\label{walkssection}
Among the most consequential developments in the study of infinitary combinatorics over the past forty years have been the arrival and elaboration of Todorcevic's method of minimal walks (\cite{todpairs, todwalks, todcoh}). As indicated above, this method ---  or, more precisely, the question of its ultimate scope and meaning --- forms the shaping inspiration of the present work. In this section we describe its basic features. 

We first establish some conventions. Below, the early Greek letters will always denote ordinals, with only $\varepsilon$ in a few clearly marked exceptions denoting a more general partial order. In these contexts, topological references are always to the order topology on a given $\alpha$, i.e., to that generated by the initial and terminal segments of $\alpha$.
\begin{defin} A subset $Q$ of a partial order $P$ is \emph{cofinal} in $P$ if for all $p\in P$ there exists a $q\in Q$ with $p\leq q$. The \emph{cofinality} $\mathrm{cf}(P)$ of a partial order $P$ is the least cardinality of a cofinal $Q\subseteq P$. It will occasionally be convenient to adopt the convention that the cofinality of a partial order possessing a maximum element is $\aleph_{-1}$.

We write $S^n_k$ for the set $\{\alpha<\omega_n\mid\mathrm{cf}(\alpha)=\aleph_k\}$. We write $\mathrm{Lim}$ and $\mathrm{Succ}$ for the classes of limit and successor ordinals, respectively.

We write $[\gamma]^n$ for the collection of size-$n$ subsets of an ordinal $\gamma$, frequently identifying this collection with that of the strictly increasing $n$-tuples of elements of $\gamma$ without further comment.

A \emph{$C$-sequence on $\gamma$} is a family $\mathcal{C}=\{ C_\beta\mid\beta<\gamma\}$ in which each $C_\beta$ is a closed cofinal subset of $\beta$. For concision, we will often term closed cofinal subsets of ordinals \emph{clubs}.
\end{defin}

We now describe the fundamentals of minimal walks; unless otherwise indicated, this material is standard and drawn from the references cited just above. With respect to some fixed C-sequence $\mathcal{C}=\{C_\beta\mid\beta\in\gamma\}$, for any $\alpha\leq\beta<\gamma$ the \emph{upper trace of the walk from $\beta$ down to $\alpha$} is recursively defined as follows:
$$ \mathrm{Tr}(\alpha,\beta)=\{\beta\}\,\cup\,\mathrm{Tr}(\alpha,\min (C_\beta\backslash\alpha)),$$
with the \emph{boundary condition} that $\mathrm{Tr}(\alpha,\alpha)=\{\alpha\}$ for all $\alpha<\gamma$. The \emph{walk} from $\beta$ to $\alpha$ is loosely identified with its upper trace, or with the collection of \emph{steps} between successive elements thereof, which is typically pictured as a series of arcs cascading in a downwards left direction, as on the left-hand side of Figure \ref{thepifigure} below. The \emph{number of steps function} $\rho_2$ sends any $\alpha$ and $\beta$ as above to $|\mathrm{Tr}(\alpha,\beta)|-1$. When $\gamma=\omega_1$ and each $C_\beta$ in $\mathcal{C}$ is of minimal possible order-type, the $\rho_2$ \emph{fiber maps} $\varphi_\beta(\,\cdot\,):=\rho_2(\,\cdot\,,\beta):\beta\to\mathbb{Z}$ form a \emph{nontrivial coherent family of functions}. More precisely, under these assumptions \begin{align}\label{c}\varphi_\beta\big|_\alpha - \varphi_\alpha=0\hspace{.8 cm}\textnormal{modulo locally constant functions}\end{align}
for all $\alpha\leq\beta<\gamma$, but there exists no $\varphi:\gamma\to\mathbb{Z}$ such that
\begin{align}\label{nt}\varphi\big|_\alpha - \varphi_\alpha=0\hspace{.8 cm}\textnormal{modulo locally constant functions}\end{align}
for all $\alpha<\gamma$ (see \cite[Corollary 2.7]{CoOI}). \emph{Coherence} broadly refers to relations like the first above; \emph{nontriviality} refers to relations like the second. The $\rho_2$ fiber maps, for example, also exhibit the following nontrivial coherence relations, closely related but not identical to (\ref{c}) and (\ref{nt}):
\begin{align}\label{cc}\varphi_\beta\big|_\alpha - \varphi_\alpha=0\hspace{.8 cm}\textnormal{modulo bounded functions}\end{align}
for all $\alpha\leq\beta<\gamma$, but there exists no $\varphi:\gamma\to\mathbb{Z}$ such that
\begin{align}\label{ohtexmaker}\varphi\big|_\alpha - \varphi_\alpha=0\hspace{.8 cm}\textnormal{modulo bounded functions}\end{align}
for all $\alpha<\gamma$.

Complementary to the upper trace is the \emph{lower trace of the walk from $\beta$ to $\alpha$}, which is defined as follows. Enumerate $\mathrm{Tr}(\alpha,\beta)$ in descending order as $(\beta_0,\beta_1,\dots,\beta_{\rho_2(\alpha,\beta)}=\alpha)$ and let $\max\,\varnothing=0$. Under the above assumptions, the sequence
$$\mathrm{L}(\alpha,\beta)=\langle\, \max(\bigcup_{i\leq j} \alpha\cap C_{\beta_i})\mid j<\rho_2(\alpha,\beta)\,\rangle$$
is well-defined, ascending towards $\alpha$ from below in tandem with $\mathrm{Tr}(\alpha,\beta)$'s descent to $\alpha$ from above. Related considerations determine the \emph{maximal weight function} $\rho_1$, defined again under the above assumptions as
$$\rho_1(\alpha,\beta)=\max_{i<\rho_2(\alpha,\beta)} |\,\alpha\cap C_{\beta_i}|.$$
Just as for $\rho_2$, the fiber maps $\varphi_\beta(\,\cdot\,):=\rho_1(\,\cdot\,,\beta)$ form a family of functions which is nontrivially coherent --- i.e., which follows the pattern of relations in (\ref{c}) and (\ref{nt}) above --- but in this case it is with respect to the modulus of \emph{finitely supported functions}.

As the above might suggest, walks and nontrivial coherence relations exhibit a particular affinity for the ordinal $\omega_1$, one, indeed, which it would be difficult to overstate.\footnote{``Despite its simplicity, [the method of minimal walks] can be used to derive virtually all known other structures that have been defined so far on $\omega_1$'' \cite[19]{todwalks}; see also the remarks at page 7 therein.} Observe most immediately, for example, that by virtue of their nontrivial coherence the families of fiber maps associated to either $\rho_1$ or $\rho_2$ above determine that most characteristic of combinatorial structures on $\omega_1$: in either case, $(\{\varphi_{\beta}\big|_\alpha\mid\alpha<\beta<\omega_1\}, \subseteq)$ forms an $\omega_1$-Aronszajn tree.
\begin{defin} A \emph{tree} $\mathcal{T}=(T,\triangleleft)$ consists of a set $T$ of \emph{nodes} and a partial order $\triangleleft$ thereon with the property that $\triangleleft$ well-orders the set $\{s\in T\mid s\,\triangleleft\, t\}$ for each $t\in T$. The \emph{height} $\mathrm{ht}(t)$ of a node $t$ is the order-type of $\{s\in T\mid s\,\triangleleft\, t\}$; the height $\mathrm{ht}(\mathcal{T})$ of $\mathcal{T}$ is $\min\{\alpha\mid\{t\in T\mid\mathrm{ht}(t)=\alpha\}=\varnothing\}$. A \emph{branch} is a maximal $\triangleleft$-linearly ordered subset of $T$; a branch $b$ is \emph{cofinal} if $\min\{\alpha\mid\{t\in b\mid\mathrm{ht}(t)=\alpha\}=\varnothing\}=\mathrm{ht}(\mathcal{T})$. A \emph{$\kappa$-tree} $\mathcal{T}$ is a tree of height $\kappa$ satisfying $|\{t\in T\mid\mathrm{ht}(t)=\alpha\}|<\kappa$ for all $\alpha<\kappa$. A $\kappa$-tree is \emph{Aronszajn} if it possesses no cofinal branch, and \emph{coherent} if it admits representation as the set of initial segments of a coherent family (modulo finite differences, typically) of functions, ordered by inclusion.
\end{defin}
We describe higher-order variants of these structures in Section \ref{treessubsection} below. We should perhaps emphasize, though, that we do not pursue higher-dimensional combinatorics quite for their own sake. The heart of the matter, rather, is this: walks techniques' extreme successes in capturing or consolidating the $\mathsf{ZFC}$ combinatorics of $\omega_1$ are unmatched at any higher $\omega_n$, and it is natural to wonder why. Broadly speaking, there have evolved two main strategies for extending these techniques' reach to higher cardinals $\kappa$. In the first, assumptions along the lines of the combinatorial principle $\square(\kappa)$ on the underlying $C$-sequence \emph{do} translate much of what's so  productive in the $\omega_1$ setting to higher $\kappa$ (see \cite{jensenfine, todpairs, todwalks}); this has the simple disadvantage of involving assumptions supplementary to the $\mathsf{ZFC}$ axioms. The second, on the other hand, involves \emph{relaxations} of requirements, on the modulus of coherence, for example, from \emph{mod finite} to \emph{mod countable} or \emph{mod $\kappa$}. Although initially avoiding assumptions beyond the $\mathsf{ZFC}$ axioms, this approach tends subsequently to need them, in the form of cardinal arithmetic conditions, for results of real force.

Against this background, Mitchell's theorem is provocative, particularly once one recognizes its $n=1$ case as fundamentally that of minimal walks on the ordinals of $\omega_1$ (see Sections \ref{thecasen1} and especially \ref{thecasen=1again} below). While remaining well within the framework of the $\mathsf{ZFC}$ axioms, it records substantial generalizations of that case to each of the higher ordinals $\omega_n$ $(n\in\omega)$. It thereby suggests a third way of extending walks techniques, one which bypasses expansions of \emph{assumptions} via expansions of dimension. 
\subsection{Internal walks}\label{internalsection}
As the above account may have suggested, walks tend to engender \emph{multiple forms} of nontrivial coherence. This observation formalizes as follows:
\begin{enumerate}
\item Nontrivial coherent families with respect to various moduli witness the nonvanishing of the cohomology groups $\check{\mathrm{H}}^1(\gamma;\mathcal{P})$ with respect to various corresponding presheaves $\mathcal{P}$.\footnote{Readers unfamiliar with these frameworks are both referred to \cite{CoOI} and reassured that they will play no essential role in our main argument; they function mainly as a convenient shorthand below.}
\item For any limit ordinal $\gamma$ and $n\geq 1$, the groups $\check{\mathrm{H}}^n(\gamma;\mathcal{P})$ with respect to the presheaves $\mathcal{P}$ corresponding to the most combinatorially prominent moduli --- namely, those of the \emph{finitely supported functions} and the \emph{locally constant functions} to an abelian group $A$, as above --- are isomorphic. In other words, nontrivial coherent families of the first sort exist on $\gamma$ if and only if nontrivial coherent families of the second sort exist on $\gamma$ as well. Both of the aforementioned moduli will feature in more general settings below.
\end{enumerate}
For example, by way of its fiber maps the function $\rho_2$ above may be viewed as representing a nontrivial element of $\check{\mathrm{H}}^1(\omega_1;\underline{\mathbb{Z}})$, where $\underline{\mathbb{Z}}$ denotes the sheaf of locally constant functions to the integers (see \cite{CoOI}; see Section \ref{71cohomology} below for further discussion of higher cohomology groups); by way of the above isomorphism, we may identify $\rho_1$ with such an element as well.

We now describe a simple mechanism for extending the observation that $\check{\mathrm{H}}^1(\omega_1;\underline{\mathbb{Z}})\neq 0$ to the result that $\check{\mathrm{H}}^1(\delta;\underline{\mathbb{Z}})\neq 0$ all ordinals $\delta$ of cofinality $\aleph_1$. This mechanism is both a critical component of higher walks and a useful heuristic for the relativizations appearing more generally below.

Here and throughout, terms like $C_\delta$ will always denote closed cofinal subsets of $\delta$ of minimal possible order-type. We will refer to their elements via increasing enumerations $\langle\eta^\delta_i\mid i\in\mathrm{cf}(\delta)\rangle$ (or $\langle\eta_i\mid i\in\mathrm{cf}(\delta)\rangle$ when $\delta$ is clear) and require moreover that $\textnormal{cf}(\eta^\delta_i)=\mathrm{cf}(i)$ for any limit ordinal $\delta$ and $i\in\mathrm{cf}(\delta)$. We will generally assume some fixed $C$-sequence to be defined over whatever ordinals we are working with.

Fix now an ordinal $\delta$ of cofinality $\aleph_1$. For $\alpha<\beta$ in $\delta$ define the \emph{upper trace $\mathit{Tr}^\delta$ of the $C_\delta$-internal walk from $\beta$ to $\alpha$} as follows: let $\eta_i=\min C_\delta\backslash\alpha$ and $\eta_k=\min C_\delta\backslash\beta$ and let \begin{align*} \mathrm{Tr}^\delta(\alpha,\beta)=\{\beta\}\,\cup\,\{\eta_j\mid j\in\mathrm{Tr}(i,k)\}.\end{align*}
In particular, $\mathrm{Tr}^\delta(\eta_i,\eta_k)$ is the image of the walk $\mathrm{Tr}(i,k)$ under the order-isomorphism $\pi:\omega_1\to C_\delta$ mapping each $i$ to $\eta_i$. Let $\rho_2[\delta](\alpha,\beta)=|\mathrm{Tr}^\delta(\alpha,\beta)|-1$. Observe then that the fiber maps $$\{\rho_2[\delta](\,\cdot\,,\beta):\beta\to\mathbb{Z}\mid\beta\in\delta\}$$
define a nontrivial coherent family of functions (modulo locally constant functions), i.e., they witness that $\check{\mathrm{H}}^1(\delta;\mathbb{Z})\neq 0$. Observe also that if $\delta=\omega_1=C_{\omega_1}$ then $\mathrm{Tr}^\delta=\mathrm{Tr}$ and $\rho_2[\delta]=\rho_2$.

\begin{figure}
\centering
\begin{tikzpicture}[MyPersp,font=\large]
	\coordinate (A) at (-.3,0,1.5);
	\coordinate (B) at (3.2,0,1.5);
	\coordinate (C) at (3.1,0,5.5);
	\coordinate (D) at (3.1,0,0);
	
	\coordinate (E) at (5.6,0,2);
	\coordinate (F) at (10.6,0,2);
	\coordinate (G) at (10.1,0,7.33);
	\coordinate (H) at (10.1,0,0.2);	
	
	\draw (A)--(B);
	\draw[thick] (C)--(D);
	\draw (E)--(F);
	\draw[line cap=round, very thick, gray, opacity=.9] (G)--(H);	
	\draw[very thick, gray, opacity=.9] (10,0,0.2)--(10.2,0,0.2);
	\draw[thick] (10.6,0,7.33) -- (10.6,0,0);
	\draw[thick] (3,0,4.8)--(3.2,0,4.8);
	\draw[thick] (3,0,4.2)--(3.2,0,4.2);
	\draw[thick] (3,0,5.5)--(3.2,0,5.5);
	\draw[thick] (3,0,1.5)--(3.2,0,1.5);
	\draw (10.1,0,.2) node[below] {$C_\delta$};
	\draw (2.7,0,4.8)[very thick] to[out=-180,in=85] (2.1,0,4.2);	
	\draw[very thick] (2.1,0,4.2) to[out=-180,in=85] (1.2,0,2.8);
		\draw[very thick] (1.2,0,2.8) to[out=-180,in=85] (.6,0,2.2);
		\draw[very thick] (.6,0,2.2) to[out=-180,in=85] (.3,0,1.9);
		\draw[very thick] (.3,0,1.9) to[out=-180,in=85] (0,0,1.5);
		\draw[very thick, gray, opacity=.5] (2.7,0,4.8)--(2.7,0,.5);
		\draw (2.7,0,.5) node[below] {$C_k$};
		\draw[very thick, gray, opacity=.5] (2.1,0,4.2)--(2.1,0,.9);
		\draw (2.1,0,.9) node[below] {$C_j$};
		\draw[very thick, gray, opacity=.5] (1.2,0,2.8)--(1.2,0,1.2);
		\draw[very thick, gray, opacity=.5] (.6,0,2.2)--(.6,0,1.9);
		\draw[very thick, gray, opacity=.5] (.3,0,1.9)--(.3,0,1);
		\draw (3.2,0,5.5) node[right] {$\omega_1$};
		\draw (3.2,0,4.8) node[right] {$k$};
		\draw (3.2,0,4.2) node[right] {$j$};
		\draw (3.2,0,1.5) node[right] {$i$};
		\draw (2.7,0,4.8) node[circle,draw,fill=black, scale=.25]{};
		\draw (2.1,0,4.2) node[circle,draw,fill=black, scale=.25]{};
		\draw (1.2,0,2.8) node[circle,draw,fill=black, scale=.25]{};
		\draw (.6,0,2.2) node[circle,draw,fill=black, scale=.25]{};
		\draw (.3,0,1.9) node[circle,draw,fill=black, scale=.25]{};
		\draw (0,0,1.5) node[circle,draw,fill=black, scale=.25]{};
		\draw[very thick, gray, opacity=.5] (2.6,0,4.2)--(2.8,0,4.2);
		\draw[very thick, gray, opacity=.5] (2.6,0,1.3)--(2.8,0,1.3);
		\draw[very thick, gray, opacity=.5] (2.6,0,1.1)--(2.8,0,1.1);
		\draw[very thick, gray, opacity=.5] (2.6,0,.5)--(2.8,0,.5);
		\draw[very thick, gray, opacity=.5] (2.6,0,4.45)--(2.8,0,4.45);
		\draw[very thick, gray, opacity=.5] (2,0,2.8)--(2.2,0,2.8);
		\draw[very thick, gray, opacity=.5] (2,0,3.8)--(2.2,0,3.8);
		\draw[very thick, gray, opacity=.5] (2,0,3.3)--(2.2,0,3.3);
		\draw[very thick, gray, opacity=.5] (2,0,3.1)--(2.2,0,3.1);
		\draw[very thick, gray, opacity=.5] (2,0,1.1)--(2.2,0,1.1);
		\draw[very thick, gray, opacity=.5] (2,0,.9)--(2.2,0,.9);
		\draw[very thick, gray, opacity=.5] (1.1,0,2.2)--(1.3,0,2.2);
		\draw[very thick, gray, opacity=.5] (1.1,0,2.5)--(1.3,0,2.5);
		\draw[very thick, gray, opacity=.5] (1.1,0,1.2)--(1.3,0,1.2);
		\draw[very thick, gray, opacity=.5] (.5,0,1.9)--(.7,0,1.9);
		\draw[very thick, gray, opacity=1] (.2,0,1.5)--(.4,0,1.5);
		\draw[very thick, gray, opacity=.5] (.2,0,1)--(.4,0,1);
		\draw[very thick, gray, opacity=.5] (.2,0,1.2)--(.4,0,1.2);
	\draw[very thick, gray, opacity=.8] (10,0,6.8)--(10.2,0,6.8);
	\draw[very thick, gray, opacity=.8] (10,0,7)--(10.2,0,7);
	\draw[very thick, gray, opacity=.5] (10,0,6.4)--(10.2,0,6.4);
	\draw[very thick, gray, opacity=.5] (10,0,5.6)--(10.2,0,5.6);
	\draw[very thick, gray, opacity=.5] (10,0,2)--(10.2,0,2);
	\draw[thick] (10.5,0,7.33)--(10.7,0,7.33);
	\draw[thick] (10.5,0,6.4)--(10.7,0,6.4);
	\draw[thick] (10.5,0,5.6)--(10.7,0,5.6);
	\draw[thick] (10.5,0,2)--(10.7,0,2);
	\draw[thick] (10.5,0,6.08)--(10.7,0,6.08);
	\draw (10.7,0,7.33) node[right] {$\delta$};
		\draw (10.7,0,6.4) node[right] {$\eta_k$};
		\draw (10.7,0,5.6) node[right] {$\eta_j$};
		\draw (10.7,0,2) node[right] {$\eta_i$};
		\draw (10.7,0,6.07) node[right] {$\beta$};
	\draw (9.95,0,6.09)[densely dotted, thick] to[out=100,in=-2] (9.6,0,6.4);
	\draw (9.6,0,6.4)[very thick] to[out=-180,in=85] (8.8,0,5.6);	
	\draw[very thick] (8.8,0,5.6) to[out=-180,in=85] (7.6,0,3.73);
		\draw[very thick] (7.6,0,3.73) to[out=-180,in=85] (6.8,0,2.93);
		\draw[very thick] (6.8,0,2.93) to[out=-180,in=85] (6.4,0,2.57);
		\draw[very thick] (6.4,0,2.57) to[out=-180,in=85] (6,0,2);
				\draw (9.6,0,6.4) node[circle,draw,fill=black, scale=.25]{};
		\draw (8.8,0,5.6) node[circle,draw,fill=black, scale=.25]{};
		\draw (7.6,0,3.73) node[circle,draw,fill=black, scale=.25]{};
		\draw (6.8,0,2.93) node[circle,draw,fill=black, scale=.25]{};
		\draw (6.4,0,2.57) node[circle,draw,fill=black, scale=.25]{};
		\draw (6,0,2) node[circle,draw,fill=black, scale=.25]{};
		\draw (9.95,0,6.09) node[circle,draw,fill=black, scale=.25]{};
		\draw[very thick, gray, opacity=.5] (9.5,0,5.6)--(9.7,0,5.6);
		\draw[very thick, gray, opacity=.5] (9.5,0,5.94)--(9.7,0,5.94);
		\draw[very thick, gray, opacity=.8] (10,0,5.94)--(10.2,0,5.94);
		\draw[very thick, gray, opacity=.5] (9.6,0,6.4)--(9.6,0,.66);
		\draw (9.6,0,.66) node[below] {$C_{\eta_k\delta}$};
		\draw[very thick, gray, opacity=.5] (9.5,0,.66)--(9.7,0,.66);
		\draw[very thick, gray, opacity=.5] (9.5,0,1.5)--(9.7,0,1.5);
		\draw[very thick, gray, opacity=.5] (9.5,0,1.75)--(9.7,0,1.75);
		\draw[very thick, gray, opacity=.5] (8.8,0,5.6)--(8.8,0,1.2);
		\draw (8.8,0,1.2) node[below] {$C_{\eta_j\delta}$};
		\draw (.2,0,4.6) node[right] {$\mathrm{Tr}(i,k)$};
		\draw (4.1,0,4.6) node[right] {$\pi``\mathrm{Tr}(i,k)=\mathrm{Tr}^{\delta}(\eta_i,\eta_k)$};
		\draw[very thick, gray, opacity=.5] (8.7,0,1.2)--(8.9,0,1.2);
		\draw[very thick, gray, opacity=.5] (8.7,0,1.5)--(8.9,0,1.5);
		\draw[very thick, gray, opacity=.5] (8.7,0,3.73)--(8.9,0,3.73);
		\draw[very thick, gray, opacity=.5] (8.7,0,4.1)--(8.9,0,4.1);
		\draw[very thick, gray, opacity=.5] (8.7,0,4.35)--(8.9,0,4.35);
		\draw[very thick, gray, opacity=.5] (8.7,0,5.1)--(8.9,0,5.1);
		\draw[very thick, gray, opacity=.5] (7.6,0,3.73)--(7.6,0,1.6);
		\draw[very thick, gray, opacity=.5] (7.5,0,1.6)--(7.7,0,1.6);
		\draw[very thick, gray, opacity=.5] (7.5,0,2.93)--(7.7,0,2.93);
		\draw[very thick, gray, opacity=.5] (7.5,0,3.33)--(7.7,0,3.33);
		\draw[very thick, gray, opacity=.5] (6.8,0,2.93)--(6.8,0,2.57);
		\draw[very thick, gray, opacity=.5] (6.7,0,2.57)--(6.9,0,2.57);
		\draw[very thick, gray, opacity=.5] (6.4,0,2.57)--(6.4,0,1.33);
		\draw[very thick, gray, opacity=.5] (6.3,0,1.33)--(6.5,0,1.33);
		\draw[very thick, gray, opacity=.5] (6.3,0,2)--(6.5,0,2);
		\draw[very thick, gray, opacity=.5] (6.3,0,1.6)--(6.5,0,1.6);
				\draw[very thick, gray, opacity=.8] (10,0,.66)--(10.2,0,.66);
		\draw[very thick, gray, opacity=.8] (10,0,1.5)--(10.2,0,1.5);
		\draw[very thick, gray, opacity=.8] (10,0,1.75)--(10.2,0,1.75);
		\draw[very thick, gray, opacity=.8] (10,0,1.2)--(10.2,0,1.2);
		\draw[very thick, gray, opacity=.8] (10,0,3.73)--(10.2,0,3.73);
		\draw[very thick, gray, opacity=.8] (10,0,4.1)--(10.2,0,4.1);
		\draw[very thick, gray, opacity=.8] (10,0,4.35)--(10.2,0,4.35);
		\draw[very thick, gray, opacity=.8] (10,0,5.1)--(10.2,0,5.1);
		\draw[very thick, gray, opacity=.8] (10,0,1.6)--(10.2,0,1.6);
		\draw[very thick, gray, opacity=.8] (10,0,2.93)--(10.2,0,2.93);
		\draw[very thick, gray, opacity=.8] (10,0,3.33)--(10.2,0,3.33);
		\draw[very thick, gray, opacity=.8] (10.2,0,2.57)--(10,0,2.57);
		\draw[very thick, gray, opacity=.8] (10.2,0,1.33)--(10,0,1.33);
		\draw[very thick, gray, opacity=.8] (10,0,2)--(10.2,0,2);
	\draw[double arrow=2pt colored by black and white]
(3.7,0,5.6) -- node[midway,anchor=center,fill=white,draw=black,line width=.2mm,inner sep=3pt, rounded corners=.5mm]{$\pi:\omega_1\to C_\delta$} (10.07,0,7.33);
\draw[->,dashdotted] (2.93,0,.24)--(9.23,0,.41);
\draw[->,dashdotted] (2.34,0,.65)--(8.42,0,.93);
\end{tikzpicture}
\caption{The order-isomorphism $\pi:\omega_1\to C_\delta:k\mapsto\eta_k$ translating a $C$-sequence on $\omega_1$, and thereby a walk, to one on $C_\delta$. On the left is the standard picture of a walk determined by a $C$-sequence (drawn in gray; notches depict representative elements of the associated $C_\gamma$) on $\omega_1$; we term the walk on the right-hand side \emph{$C_\delta$-internal}. For initial inputs $\beta\notin C_\delta$ such walks require a first step ``up into'' $C_\delta$; this we have depicted as well.}
\label{thepifigure}
\end{figure}
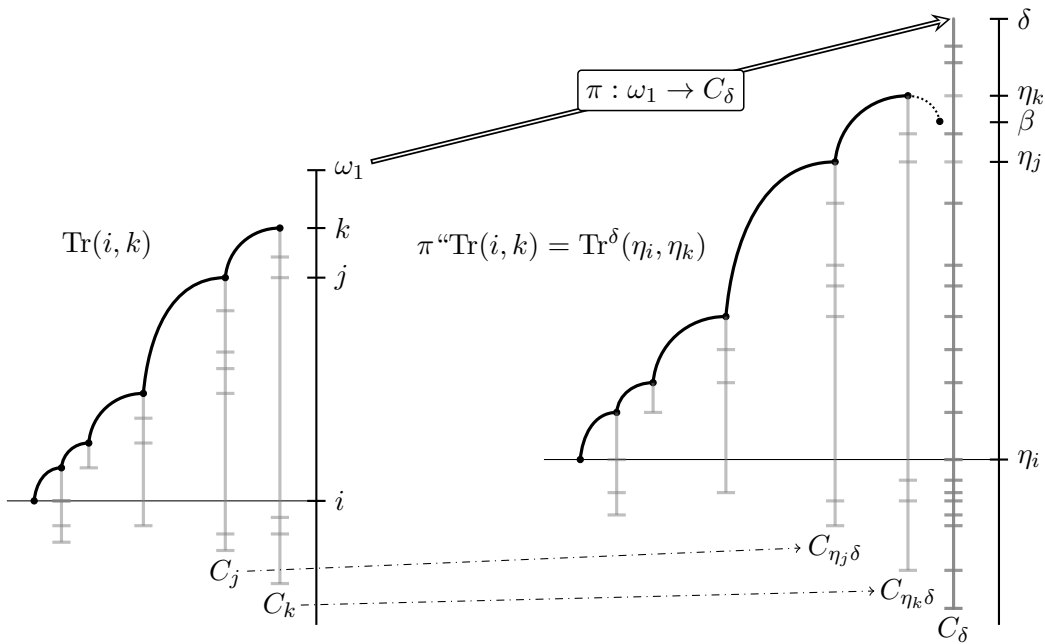

If for $\beta\in C_\delta$ we let $C_{\beta\delta}=\pi`` C_{\pi^{-1}(\beta)}$ then we may define the \emph{$C_\delta$-internal walk} more directly: for $\beta\in C_\delta$ we have
$$\mathrm{Tr}^\delta(\alpha,\beta)=\{\beta\}\,\cup\,\mathrm{Tr}^\delta(\alpha,\min(C_{\beta\delta}\backslash\alpha)).$$
To ground the recursion, let $\mathrm{Tr}^\delta(\alpha,\min(C_\delta\backslash\alpha))=\{\min(C_\delta\backslash\alpha)\}$ for all $\alpha<\delta$. (Similarly for $\beta\not\in C_\delta$, though the notation grows cluttered.\footnote{As the referee has noted, under present definitions, $\mathrm{Tr}^\delta(\beta,\beta)=\{\beta,\mathrm{min}\,C_\delta\backslash\beta\}$ when $\beta\not\in C_\delta$, with the consequence that $\rho_2[\delta](\beta,\beta)=1$, in contrast with the classical fact that $\rho_2(\beta,\beta)$ is always zero. Superficial modifications to these definitions would eliminate this effect without affecting this section's broader argument, and indeed, the term \emph{internal walk} may be taken to refer to any of a family of minor variations on the single basic idea we are describing. The variant we have recorded here has the virtue of corresponding to the values arising at the second or inner coordinate along the internal branches of the three-coordinate higher walks; see Section \ref{thecasehighernagain}.}) Rather than mapping a walk on the countable ordinals to one on those of $C_\delta$, this second framing maps the underlying $C$-sequence on the countable ordinals to the ordinals of $C_\delta$, and walks thereon. See Figure \ref{thepifigure}.

In one view, these \emph{internal walks} are the material of walks of the next higher order; these we describe in Sections \ref{thecasehighernagain} and \ref{lastsection} below. The basic idea of these translations, though, suffuses this work very generally.

\subsection{Simplicial complexes and the systems $\mathbf{P}_n(\varepsilon)$}\label{3210}

By \textit{simplicial complex $B$ on $\beta$} we mean a simplicial complex whose vertices are the elements of $\beta$. We may more generally identify any $n$-dimensional face of $B$ with the size-$(n+1)$ set of its vertices. Best suited for our purposes, in other words, are \textit{abstract simplicial complexes on $\beta$}: $\subseteq$-downward-closed collections of finite subsets of $\beta$. Writing $B^n$ for the set of $n$-dimensional faces of $B$, we then have
 \begin{enumerate}
 \item $B^n\subseteq[\beta]^{n+1}$, and
 \item $\bigcup_{k\leq n} B^k$ is the $n$-skeleton of $B$.
 \end{enumerate}
The \emph{dimension} of a simplicial complex $B$ is $\text{sup}\{ n\mid B^n\neq\varnothing\}$. For any such $B$ let \begin{align}\label{doldkan} C_n(B)=\bigoplus_{B^n}\mathbb{Z}\end{align} and for all $\vec{\alpha}\in B^n$ write $\langle\vec{\alpha}\rangle$ for the associated generator of $C_n(B)$. Writing $\vec{\alpha}^i$ for the $(n-1)$-face of $B$ obtained by omitting the $i^{\mathrm{th}}$ element of $\vec{\alpha}$, the maps \begin{align*}& \langle\vec{\alpha}\rangle\mapsto\displaystyle\sum_{i\leq n}(-1)^i \langle\vec{\alpha}^i\rangle
\end{align*} then induce \emph{boundary homomorphisms} $\partial_n:C_n(B)\rightarrow C_{n-1}(B)$ and \emph{simplicial homology groups} $$\mathrm{H}_n^{\Delta}(B)=\frac{\text{ker}(\partial_n)}{\text{im}(\partial_{n+1})}$$ for each $n\geq 0$. Here $C_{-1}(B)$, and hence $\partial_0$, equals zero. The \emph{reduced simplicial homology groups} $\tilde{\mathrm{H}}_n^{\Delta}(B)$ are similarly defined, but with $C_{-1}(B)=\mathbb{Z}$ and $\partial_0:\langle\alpha\rangle\mapsto 1$ for all $\alpha\in\beta$. Observe that the complex $B$ (or, more precisely, its geometric realization) is connected if and only if $\mathrm{H}_0^{\Delta}(B)=\mathbb{Z}$, or equivalently if and only if $\tilde{\mathrm{H}}_0^{\Delta}(B)=0$.

When $\beta$ is of cofinality $\aleph_k$, its order-structure manifests as a $k$-dimensional combinatorial-topological condition on the family of simplicial complexes $B$ on $\beta$. This is the content of Theorem \ref{omegan} below. The mechanism of this surprising rapport is a \textit{grading} of simplicial complexes, for which inverse systems are a convenient framework.
\begin{defin} Let $\mathcal{C}$ be a category. An \emph{inverse system in $\mathcal{C}$ over $I$} consists of a partially ordered \emph{index-set} $I$, \emph{terms} $X_i$ $(i\in I)$, and \emph{bonding maps} $x_{ij}:X_j\to X_i$ satisfying $x_{ik}=x_{ij}\,x_{jk}$ for all $i\leq j\leq k$ in $I$, where the terms and bonding maps are objects and morphisms in the category $\mathcal{C}$, respectively. We will typically represent such systems as triples $(X_i,x_{ij},I)$ and more abstractly denote inverse systems by boldfaced variables like $\mathbf{X}$; we will also take $\mathcal{C}$ to be the category $\mathsf{Ab}$ of abelian groups except where otherwise indicated below. A morphism between two inverse systems $\mathbf{X}=(X_i,x_{ij},I)$ and $\mathbf{Y}=(Y_i,y_{ij},I)$ is a family $\mathbf{f}=\{f_i:X_i\to Y_i\mid i\in I\}$ of morphisms satisfying $y_{ij}\,f_j=f_i\,x_{ij}$ for all $i\leq j$ in $I$.
\end{defin}
The terms of our central examples are of the following forms: for $n\geq 0$ and $A$ a collection of ordinals, let $$P_n(A)=\bigoplus_{[A]^{n+1}}\mathbb{Z}\hspace{.3 cm}\textnormal{      and      }\hspace{.3 cm}R_n(A)=\displaystyle\prod_{[A]^{n+1}}\mathbb{Z}\,.$$
\noindent In the framework of (\ref{doldkan}) above, $P_n(A)$ is $C_n(B)$, where $B$ is the complete $n$-dimensional simplicial complex on $A$. For both $P_n(A)$ and $R_n(A)$, again write $\langle\vec{\alpha}\rangle$ for the generator associated to $\vec{\alpha}\in[A]^{n+1}$. Again boundary maps on these $\langle\vec{\alpha}\rangle$ determine maps $$d_n: P_n(A)\rightarrow P_{n-1}(A)$$
for $n\geq 1$. For any ordinal $\varepsilon$ and $n\geq 0$ define then the inverse system $$\mathbf{P}_n(\varepsilon)=(P_n([\alpha,\varepsilon)),p_{\alpha\beta},\varepsilon)$$
with $p_{\alpha\beta}:P_n([\beta,\varepsilon))\rightarrow P_n([\alpha,\varepsilon))$ the natural inclusion map, for $\alpha\leq\beta<\varepsilon$, and define $\mathbf{R}_n$ analogously. (Here and below we denote intervals of ordinals just as we would intervals of reals; $[\alpha,\varepsilon)=\{\xi\in\mathrm{Ord}\mid\alpha\leq\xi<\varepsilon\}$, for example.) Observe that $p_{\alpha\beta}$ and $d_n$ commute; hence the maps $d_n$ determine in turn a mapping of inverse systems $$\mathbf{d}_n:\mathbf{P}_n(\varepsilon)\rightarrow\mathbf{P}_{n-1}(\varepsilon)\,.$$
This map may be regarded as a natural transformation between contravariant functors ($\alpha\mapsto P_n([\alpha,\varepsilon))$ and $\alpha\mapsto P_{n-1}([\alpha,\varepsilon))$, respectively) from the partial order $\varepsilon$, viewed as a category, to the category of abelian groups. We write $\mathsf{Ab}^{\varepsilon^{\mathrm{op}}}$ for the category with such functors as objects and natural transformations as morphisms. Observe that $\mathsf{Ab}^{\varepsilon^{\mathrm{op}}}$ is an abelian category; in particular, sums and kernels and quotients and, hence, exact sequences exist therein, and are evaluated pointwise (e.g., the terms of a quotient are the quotients of the corresponding terms). 
We write $\mathbf{\Delta}_\varepsilon(\,\cdot\,)$ for the \textit{diagonal functor} $A\mapsto (A,\text{id},\varepsilon)$ embedding $\mathsf{Ab}$ into $\mathsf{Ab}^{\varepsilon^{\mathrm{op}}}$; in particular, $\mathbf{\Delta}_\varepsilon(\mathbb{Z})$ is the inverse system $(\mathbb{Z},\text{id},\varepsilon)$. The aforementioned objects then assemble in the following exact sequence:
\begin{equation}\label{P}\tag{$\mathsf{P}(\varepsilon)$}\dots\stackrel{\mathbf{d}_{n+1}}{\longrightarrow}\mathbf{P}_n(\varepsilon)\stackrel{\mathbf{d}_{n}}{\longrightarrow}\mathbf{P}_{n-1}(\varepsilon)\stackrel{\mathbf{d}_{n-1}}{\longrightarrow}\dots \stackrel{\mathbf{d}_1}{\longrightarrow} \mathbf{P}_{0}(\varepsilon)\stackrel{\mathbf{d}_0}{\longrightarrow}\mathbf{\Delta}_\varepsilon(\mathbb{Z})\longrightarrow\mathbf{0}\, , \end{equation}
with $\mathbf{d}_0=\{d_{0,\alpha}\,|\,\alpha\in\varepsilon\}$ defined by $d_{0,\alpha}:\langle\beta\rangle\mapsto 1$ for all $\alpha\leq\beta<\varepsilon$. We term this sequence the \emph{standard projective resolution} of $\mathbf{\Delta}_\varepsilon(\mathbb{Z})$.

Simple as it might appear, the sequence \ref{P} will be a main object of study below. A main part of our argument, in other words, will frequently be the manipulation of algebraic relations between $n$-tuples of ordinals. For this work, a clear but flexible notation is critical; we therefore pause to collect and augment its more scattered description above:

\begin{enumerate}
\item For $A$ a collection of ordinals, we write $\vec{\beta}\in[A]^n$ to mean that $\vec{\beta}$ is an increasing $n$-tuple $(\beta_0,\dots,\beta_{n-1})$ of ordinals in $A$. We will typically write a $1$-tuple $(\beta)$ as $\beta$. For $\vec{\beta}\in[A]^n$ and $0\leq i<n$, we write $\vec{\beta}^i$ for $\vec{\beta}$ with the $i^{th}$ coordinate removed. If $\vec{\beta}$ is a $1$-tuple, then $\vec{\beta}^0=\varnothing$. As we did when defining simplicial complexes above, we will sometimes simply view $\vec{\beta}$ as an $n$-element subset of $A$; we write $\vec{\alpha}<\vec{\beta}$ to mean that every element of $\vec{\alpha}$ is less than every element of $\vec{\beta}$. We apply the restriction-notation $B\big|_X$ both to functions and to simplicial complexes; in the latter case, it denotes the simplicial complex comprised of those $x\in B$ satisfying $x\subseteq X$.
\item As for $\mathbf{d}_n$ above, we will define maps among inverse systems largely by way of their actions on generators $\langle\vec{\alpha}\rangle$; at times, we will conflate maps between terms (like $d_n$) and maps between inverse systems (like $\mathbf{d}_n$) as well. Relatedly, we will tend not to formally distinguish between a generator $\langle\vec{\gamma}\rangle\in P_n([\beta,\varepsilon))$ and its images $p_{\alpha\beta}(\langle\vec{\gamma}\rangle)$. When we do, it will be to regard $\langle\vec{\gamma}\rangle$ as an element of the ``highest possible'' term of $\mathbf{P}_n(\varepsilon)$ --- namely, $P_n([\gamma_0,\varepsilon))$.
\item We write $\langle\vec{\alpha},\vec{\beta}\rangle$ for $\langle\vec{\alpha}^\frown\vec{\beta}\rangle$; we will also at times write sums of generators inside the angled brackets, preferring expressions like $\langle d_k \vec{\alpha},\vec{\beta}\rangle$ to $\sum_{i=0}^k(-1)^i\langle\vec{\alpha}^i,\vec{\beta}\rangle$. As they do here, commas can render such expressions more readable. In subscripts, however, such commas typically have more of an effect of clutter.
 In these cases we omit them, denoting concatenations of coordinates, as in $(\beta_0,\dots,\beta_{m-1},\gamma_0,\dots,\gamma_{n-1})$, as concatenations of tuples, as in $\vec{\beta}\vec{\gamma}$. Putting all this together: the tuple $(\beta_0,\beta_2,\delta)$ would typically appear in a subscript as $C_{\vec{\beta}^1\delta}$, for example; it would appear in a generator probably as $\langle\vec{\beta}^1,\delta\rangle$. Lastly, an expression like $d_k\mathcal{B}$ means $\{ d_k\langle\vec{\alpha}\rangle\,|\,\langle\vec{\alpha}\rangle\in\mathcal{B}\}$.
\end{enumerate}

\subsection{Free and projective inverse systems}\label{freeandproj}

\begin{defin} For any object $\mathbf{P}$ in $\mathsf{Ab}^{\varepsilon^{\mathrm{op}}}$, let $\mathbf{id}$ denote the identity morphism. $\mathbf{P}$  is \textit{projective} if for any epimorphism $\mathbf{e}:\mathbf{R}\rightarrow\mathbf{P}$ there exists a morphism $\mathbf{s}:\mathbf{P}\rightarrow\mathbf{R}$ such that $\mathbf{e}\,\mathbf{s}=\mathbf{id}$. We will sometimes term such a right-inverse to an epimorphism a \textit{section}. Dually, we will sometimes term a left-inverse $\mathbf{r}$ to a monomorphism $\mathbf{m}$ a \textit{retract}.

An object $\mathbf{X}$ in $\mathsf{Ab}^{\varepsilon^{\mathrm{op}}}$ is \textit{free} if there exists some $\mathcal{B}\subseteq\bigcup_{\alpha<\varepsilon}X_\alpha$ such that any $x$ in any $X_\alpha$ has a unique $\mathcal{B}$-decomposition
$$x=\displaystyle\sum_{i<k}a_i q_{\alpha\beta_i}(b_i)$$
with $b_i\in X_{\beta_i}\cap\mathcal{B}$ for all $i<k$.
\end{defin}

\begin{exa} The system $\mathbf{\Delta}_\varepsilon(\mathbb{Z})$ is free if and only if $\varepsilon$ is a successor, i.e., if $\text{cf}(\varepsilon)=1$. The system $\mathbf{P}_{n}(\varepsilon)$ is free, on the other hand, for any ordinal $\varepsilon$ and $n\in\omega$. By an argument exactly as in more standard settings, it follows that every $\mathbf{P}_n(\varepsilon)$ is projective as well.\footnote{The interested reader is encouraged to verify these assertions directly; for the first and last of them, though, see also \cite[Example 11.17]{SSH} and Lemma \ref{summandlemma} below, respectively.}
\end{exa}

The reverse question of whether a projective system is free (or, conversely, of whether a nonfree system is nonprojective) is in general much subtler. Even the simplest instance is less than obvious:
\begin{align}\label{projquest} \text{Let $\varepsilon$ be a limit ordinal. Is }\mathbf{\Delta}_\varepsilon(\mathbb{Z})\text{ projective?}
\end{align}
The question involves a different order of quantification from that of freeness: it quantifies over the collection of morphisms in $\mathsf{Ab}^{\varepsilon^{\mathrm{op}}}$. Arguably the obscurity --- or, in another view, the power --- of the notion of projective consists, simply, in this quantification.
To see that a question like (\ref{projquest}) is as much about the ambient category as it is about the object itself, consider the following:

\begin{defin}
For any infinite cardinal $\kappa$ let $\kappa$-$\mathsf{Ab}$ denote the category of abelian groups of size less than $\kappa$, and let $\kappa.g.$-$\mathsf{Ab}$ denote the category of abelian groups with generating sets of size less than $\kappa$. When $\kappa=\omega$, of course, these are the categories of \emph{finite abelian groups} and of \emph{finitely generated abelian groups}, respectively; note also that $\omega$ is the only infinite cardinal $\kappa$ for which the categories $\kappa$-$\mathsf{Ab}$ and $\kappa.g.$-$\mathsf{Ab}$ are distinct.
\end{defin}

It is straightforward to verify that each of the categories defined above is abelian. Recall that a cardinal $\kappa$ has the \emph{tree property} if there exist no $\kappa$-Aronsjazn trees, i.e., if every tree of height $\kappa$ and level-widths all less than $\kappa$ possesses a cofinal branch.

\begin{thm}\label{treepropproj} For any infinite cardinal $\kappa$,
\begin{enumerate}
\item if $\mathbf{\Delta}_\kappa(\mathbb{Z})$ is projective in $(\kappa.g.$-$\mathsf{Ab})^{\kappa^{\mathrm{op}}}$ then $\kappa$ has the tree property, and
\item if $\kappa$ has the tree property then $\mathbf{\Delta}_\kappa(\mathbb{Z})$ is projective in $(\kappa$-$\mathsf{Ab})^{\kappa^{\mathrm{op}}}$.
\end{enumerate}
In particular, an uncountable cardinal $\kappa$ has the tree property if and only if $\mathbf{\Delta}_\kappa(\mathbb{Z})$ is projective in $(\kappa$-$\mathsf{Ab})^{\kappa^{\mathrm{op}}}$.
\end{thm}

By Theorem \ref{treepropproj} and K\"{o}nig's Infinity Lemma, $\mathbf{\Delta}_{\omega}(\mathbb{Z})$ is projective in the category of \textit{height-$\omega$ inverse systems of finite abelian groups}. $\mathbf{\Delta}_{\omega}(\mathbb{Z})$ is \textit{not} projective, however, in the wider category of \textit{height-$\omega$ inverse systems of abelian groups}. We will argue this latter fact as the base case in the inductive proof of Proposition \ref{notproject}; it will follow as well from our results in Section \ref{deromega1}. A subtler point complicating the statement of Theorem \ref{treepropproj} is the fact that $\mathbf{\Delta}_{\omega}(\mathbb{Z})$ is not projective in the category of \textit{height-$\omega$ inverse systems of finitely \textbf{generated} abelian groups}. This subtlety is sufficiently removed from our main concerns, though, that we record just the contours of a counterexample in a footnote.\footnote{Perhaps the simplest counterexample is the inverse sequence of finitely generated abelian groups $\mathbf{Q}=(Q_i,q_{ij},\omega)$ in which each $Q_i=\mathbb{Z}\oplus\mathbb{Z}$ and each $q_{i,i+1}:Q_{i+1}\to Q_i$ is given by the maps $(1,0)\mapsto (2,0)$ and $(0,1)\mapsto (1,1)$. Define an epimorphism $\mathbf{e}:\mathbf{Q}\to\mathbf{\Delta}_\omega(\mathbb{Z})$ by letting each $e_i$ be the map determined by $(1,0)\mapsto 0$ and $(0,1)\mapsto 1$. If $\mathbf{s}$ were right-inverse to $\mathbf{e}$ then $s_0(1)$ would need to fall in $\bigcap\{(1+2^n,1)+2^{n+1}\mathbb{Z}\oplus\{0\}\mid n\in\omega\}$, as the reader may verify. But this intersection is empty.}

\begin{proof}[Proof of Theorem \ref{treepropproj}] First we show item (2). Suppose that $\kappa$ has the tree property, and consider an epimorphism $\mathbf{e}=\{e_\xi:Q_\xi\rightarrow\mathbb{Z}\,|\,\xi<\kappa\}$ from some $\mathbf{Q}=(Q_\xi,q_{\eta\xi},\kappa)$ in $(\kappa$-$\mathsf{Ab})^{\kappa^{\mathrm{op}}}$ to $\mathbf{\Delta}_\kappa(\mathbb{Z})$. We will show that $\mathbf{e}$ has a right-inverse $\mathbf{s}$. As $\mathbf{e}$ is arbitrary, this will show that $\mathbf{\Delta}_\kappa(\mathbb{Z})$ is projective in $(\kappa$-$\mathsf{Ab})^{\kappa^{\mathrm{op}}}$. Observe that $$\big(\bigcup_{\xi<\kappa} e_\xi^{-1}(1),\trianglelefteq\big)$$
defines a $\kappa$-tree $T$, where $x\trianglelefteq y$ iff $q_{\eta\xi}(y)=x$ for some $\eta\leq\xi<\kappa$. By the tree property, $T$ contains a cofinal branch $\{x_\xi\,|\,\xi<\kappa\}$. Setting $s_\xi(1)=x_\xi$ for $\xi<\kappa$ then defines an $\mathbf{s}:\mathbf{\Delta}_\kappa(\mathbb{Z})\rightarrow\mathbf{Q}$ right-inverse to $\mathbf{e}$.

To see item (1), fix a $\kappa$-Aronszajn tree $T$, let $$Q_\xi=\bigoplus_{\text{lev}_\xi(T)}\mathbb{Z}$$ be the free group generated by the $\xi^{\text{th}}$ level of $T$ and define $q_{\eta\xi}:Q_\xi\rightarrow Q_\eta$ by $q_{\eta\xi}(\langle x\rangle)=\,$\textit{the $Q_\eta$-generator corresponding to the $\eta^{\text{th}}$-level predecessor of $x$}. This determines an object $\mathbf{Q}$ of $(\kappa.g.$-$\mathsf{Ab})^{\kappa^{\mathrm{op}}}$. Mappings $\langle x\rangle\mapsto 1$ for $x\in\text{lev}_\xi(T)$ then induce $e_\xi:Q_\xi\rightarrow\mathbb{Z}$, and, hence, an $\mathbf{e}:\mathbf{Q}\rightarrow\mathbf{\Delta}_\kappa(\mathbb{Z})$ with no right-inverse $\mathbf{s}=\{s_\xi:\mathbb{Z}\rightarrow Q_\xi\,|\,\xi<\kappa\}$, since for any such inverse, $\{\text{supp}(s_\xi(1))\,|\,\xi<\kappa\}$ would define a finitely branching subtree, and hence a cofinal branch, in $T$.
\end{proof}

The above remarks and theorem were something of a digression, meant to help frame the recognition below that a number of projective inverse systems are free.\footnote{The systems we consider are indeed ``big,'' so we are recording a fact somewhat described by Hyman Bass's 1963 title \textit{Big projective modules are free} \cite{bass}. Note that by Theorem \ref{treepropproj}, though, that title is far from describing the situation for \emph{inverse systems} in any unqualified generality: assume the tree property of some infinite cardinal $\kappa$ (readers wary of large cardinals may take $\kappa$ to be $\omega$). Then $\mathbf{\Delta}_\kappa(\mathbb{Z})$ is a projective system in $(\kappa$-$\mathsf{Ab})^{\kappa^{\mathrm{op}}}$ which, by Claim \ref{clma3} below, is not projective in $\mathsf{Ab}^{\kappa^{\text{op}}}$, and therefore cannot be free. Mitchell cites the question of when projective objects in categories $\mathsf{Ab}^\mathcal{C}$ are free as motivating \cite{eilmoore} (see \cite[5]{rings}).

Note also that the proof of Theorem \ref{treepropproj} shows that $\lim\,\mathbf{Q}=0$ for the system $\mathbf{Q}$ associated therein to an Aronszajn tree $T$. The question of the values of higher limits of $\mathbf{Q}$-like constructions is a subtler matter and could conceivably shape a productive approach to the study of various set-theoretic trees.} To apply this recognition, we will want the following standard lemma:
\begin{lem}\label{summandlemma} An inverse system $\mathbf{X}$ is projective if and only if $\mathbf{X}$ is a direct summand of a free inverse system $\mathbf{Y}$.
\end{lem}
\begin{proof} For the \textit{only if} direction, fix an epimorphism $\mathbf{e}$ from a free system $\mathbf{Y}$ to $\mathbf{X}$. As $\mathbf{X}$ is projective, $\mathbf{e}$ admits a right-inverse $\mathbf{s}$, so that $\mathbf{Y}\cong\mathbf{s}(\mathbf{X})\oplus\text{ker}(\mathbf{e})\cong\mathbf{X}\oplus\text{ker}(\mathbf{e})$. For the \textit{if} direction, observe that if $\mathbf{Y}=\mathbf{X}\oplus\mathbf{Z}$ then any epimorphism $\mathbf{e}:\mathbf{R}\rightarrow\mathbf{X}$ naturally extends to an epimorphism $\mathbf{e}':\mathbf{R}\oplus\mathbf{Z}\rightarrow\mathbf{X}\oplus\mathbf{Z}$. As $\mathbf{Y}$ is free, $\mathbf{e}'$ has a right-inverse $\mathbf{s}'$, which restricts to an $\mathbf{s}:\mathbf{X}\rightarrow\mathbf{R}$ right-inverse to $\mathbf{e}$.\end{proof}

As the lemma suggests, it is not in general true that subsystems of free inverse systems of abelian groups are free,
 or even projective. A central concern below, in fact, is the question of whether the subsystem $\mathbf{d}_{n}\mathbf{P}_{n}(\varepsilon)$ of the free system $\mathbf{P}_{n-1}(\varepsilon)$ is projective. This question, as we will see, is fundamentally a question about the cofinality of $\varepsilon$. Observe in this connection that we may truncate the exact sequence \ref{P} at any $\mathbf{d}_{n}\mathbf{P}_{n}(\varepsilon)$ to form a shorter exact sequence as follows:
\begin{equation}\label{Q}\dots\rightarrow\mathbf{0}\longrightarrow\mathbf{d}_n\mathbf{P}_n(\varepsilon)\stackrel{\mathbf{i}}{\longrightarrow}\mathbf{P}_{n-1}(\varepsilon)\stackrel{\mathbf{d}_{n-1}}{\longrightarrow}\dots \stackrel{\mathbf{d}_1}{\longrightarrow} \mathbf{P}_{0}(\varepsilon)\stackrel{\mathbf{d}_0}{\longrightarrow}\mathbf{\Delta}_\varepsilon(\mathbb{Z})\longrightarrow\mathbf{0} \end{equation}
If $\mathbf{d}_{n}\mathbf{P}_{n}(\varepsilon)$ is projective then (\ref{Q}) shares with \ref{P} the property that all terms except possibly the ``target'' $\mathbf{\Delta}_\varepsilon(\mathbb{Z})$ are projective.
\begin{defin}\label{projdefinition} A \textit{projective resolution} of an inverse system $\mathbf{X}$ is an exact sequence ending in $\mathbf{X}$ as in \ref{P} or (\ref{Q}), above, in which all nonzero terms except possibly the rightmost are projective. Such resolutions are sometimes written $\mathsf{P}\rightarrow\mathbf{X}\rightarrow\mathbf{0}$. The \textit{length} of $\mathsf{P}$ is the supremum of the indices of its nonzero terms --- where $\mathsf{P}$'s terms are indexed, as above, from right to left, beginning with zero. Possibly all of $\mathsf{P}$'s terms are nonzero; its length in this case is $\infty$. The \textit{projective dimension} of $\mathbf{X}$, written $\text{pd}(\mathbf{X})$, is the minimal length of a projective resolution of $\mathbf{X}$. An equivalent definition is the following: given any projective resolution $\mathsf{P}=\langle\mathbf{P}_n,\mathbf{d}_n\mid n\in\omega\rangle$ of $\mathbf{X}$, the projective dimension of $\mathbf{X}$ is the least $n$ such that $\mathbf{d}_n \mathbf{P}_n$ is projective.
\end{defin}

\begin{exa} $\mathbf{X}$ is projective if and only if $\dots\rightarrow\mathbf{0}\rightarrow\mathbf{X}\stackrel{\mathbf{id}}{\longrightarrow}\mathbf{X}\rightarrow\mathbf{0}$ is a projective resolution, if and only if pd$(\mathbf{X})=0$. More generally, pd($\mathbf{X}$) may be read as quantifying how ``far'' a system $\mathbf{X}$ is from being projective.
\end{exa}

We conclude this section with several summary remarks and with a statement of Mitchell's theorem. Our interest is in projective resolutions of $\mathbf{\Delta}_\varepsilon(\mathbb{Z})$, for two related reasons:
\begin{enumerate}
\item[(i)] They translate order-theoretic information into algebraic information.
\item[(ii)] They are of computational value.
\end{enumerate}
In item (ii), we have in mind the following: the diagonal functor $\mathbf{\Delta}_\varepsilon(\,\cdot\,):\mathsf{Ab}\rightarrow\mathsf{Ab}^{\varepsilon^{\mathrm{op}}}$ is left-adjoint to the inverse limit functor $\lim\,(\,\cdot\,):\mathsf{Ab}^{\varepsilon^{\mathrm{op}}}\rightarrow\mathsf{Ab}$. This has as consequence the formula
\begin{align}\label{adjointb}\text{H}^n(\text{Hom}(\mathsf{P},\mathbf{X}))\cong\text{Ext}^n(\mathbf{\Delta}_\varepsilon(\mathbb{Z}),\mathbf{X})\cong\text{lim}^n\,\mathbf{X}\, .\end{align}
(See \cite[Section 12.2]{SSH}). Here $\mathbf{X}=(X_\alpha,x_{\alpha\beta},\varepsilon)$ is any system in $\mathsf{Ab}^{\varepsilon^{\mathrm{op}}}$ and $\mathsf{P}$ is any projective resolution of $\mathbf{\Delta}_\varepsilon(\mathbb{Z})$, such as \ref{P} or (\ref{Q}) above. $\text{Ext}^n$ and $\text{lim}^n$ are the \emph{higher derived functors} of the functors Hom and lim, respectively (see just below or \cite{weibel} or \cite{SSH} or \cite{jensen} for further discussion). Via equation \ref{adjointb}, the standard projective resolution $\mathsf{P}(\varepsilon)$ of $\mathbf{\Delta}_\varepsilon(\mathbb{Z})$ uniformizes the computation of higher derived limits, providing, in particular, explicit formulae for $\lim^n(\,\cdot\,)$, as we will now describe.

Above, $\text{Hom}(\mathsf{P},\mathbf{X})$ denotes the cochain complex with $n^{\mathrm{th}}$ term $\text{Hom}(P_n,\mathbf{X})$, where $P_n$ is the $n^{\mathrm{th}}$ term of $\mathsf{P}$; the associated coboundary maps are those induced by the boundary maps of $\mathsf{P}$. If $\mathsf{P}=\mathsf{P}(\varepsilon)$ then since $\mathbf{P}_n(\varepsilon)$ is free, elements of $\text{Hom}(\mathbf{P}_n(\varepsilon),\mathbf{X})$ amount simply to a choice of map for each basis-element $\langle\vec{\alpha}\rangle$ of $\mathbf{P}_n(\varepsilon)$. It follows that $\text{lim}^n\mathbf{X}$ may be computed as the $n^{\mathrm{th}}$ cohomology group of the cochain complex we denote $\mathcal{K}(\mathbf{X})$, with cochain groups

\begin{align}\label{readerverify}K^j(\mathbf{X})=\prod_{\vec{\alpha}\in [\varepsilon]^{j+1}}X_{\alpha_0}\end{align}
and coboundary maps $d^j:K^j(\mathbf{X})\rightarrow K^{j+1}(\mathbf{X})$ defined by $$d^j(c)(\vec{\alpha})=x_{\alpha_0\alpha_1}(c(\vec{\alpha}^0))+\sum_{i=1}^j (-1)^i c(\vec{\alpha}^i)\, .$$
We will apply this description in Section \ref{afterbases} below.

Again note, on the other hand, that other resolutions $\mathsf{P}$ of $\mathbf{\Delta}_\varepsilon(\mathbb{Z})$ might be taken in place of $\mathsf{P}(\varepsilon)$ in equation \ref{adjointb}. In particular, the eventual zeros of a resolution like (\ref{Q}) will translate in equation \ref{adjointb} to vanishing cohomology groups, and hence to vanishing higher derived limits for \emph{any} inverse system indexed by $\varepsilon$, for all $n$ above some finite $m$. More precisely:

\begin{lem}\label{projlemma}
The projective dimension of $\mathbf{\Delta}_\varepsilon(\mathbb{Z})$ is $n\in\mathbb{N}$ if and only if $n$ is the largest integer for which $\text{lim}^n(\,\cdot\,)$ is nonvanishing, i.e., for which there exists some $\mathbf{X}$ in $\mathsf{Ab}^{\varepsilon^{\mathrm{op}}}\!$ with $\text{lim}^n(\mathbf{X})\neq 0$.\footnote{A version of this lemma appears as Theorem 13.3 in \cite{SSH}, wherein it is traced to \cite[Theorem 7.20]{bucur}.}
\end{lem}
Sensitivities of $\textnormal{pd}(\mathbf{\Delta}_\varepsilon(\mathbb{Z}))$ to the cofinality of $\varepsilon$ transmit in this manner to functors of broad application and importance, namely, the higher derived limits \textnormal{lim}$^n$. These are of sufficient significance, for example, to warrant the following definition:
\begin{defin}[\cite{latch}] \label{cohomologicaldimension}
The \emph{cohomological dimension} $\text{cd}(\varepsilon)$ of a partial order $\varepsilon$ (or more generally of any small category $\varepsilon$) is the supremum of $\{n\mid \lim^n:\mathsf{Ab}^{\varepsilon^{\text{op}}}\to\mathsf{Ab}\textnormal{ does not equal }0\}$. The supremum of $\mathbb{N}$ is denoted by $\infty$.\footnote{A certain instability of terminology for this invariant of $\varepsilon$ afflicts the literature. As indicated, we follow \cite{latch}, which seems to us the most logical; Definition \ref{cohomologicaldimension} is also equivalent to that appearing in \cite{mitchcohdim}. Note, though, that this same invariant is termed the \emph{homological dimension of $\varepsilon$} in \cite{SSH}, and the \emph{cohomological dimension of $\varepsilon^{\text{op}}$} in \cite{husainov}; as is perhaps apparent, these variations reflect only superficially different approaches to the constitutive contravariance of $\mathsf{Ab}^{\varepsilon^{\text{op}}}$.}
\end{defin}
Before continuing, we pause to recall the most essential feature of higher derived limits: their interrelations in long exact sequences deriving from short exact sequences in $\mathsf{Ab}^{\varepsilon^\text{op}}$. Higher limits are an artifact of the lim functor's ``failure to be exact''; for example, the lim-image of the short exact sequence
\begin{equation}\label{A}\mathbf{0}\longrightarrow\mathbf{P}_k(\varepsilon)\stackrel{\mathbf{i}}{\longrightarrow}\mathbf{R}_{k}(\varepsilon)\stackrel{\mathbf{q}}{\longrightarrow}\mathbf{R}_{k}(\varepsilon)/\mathbf{P}_k(\varepsilon)\longrightarrow\mathbf{0} \end{equation}
may itself be only \emph{half} or \emph{left exact}, meaning that while $\lim\mathbf{i}$ will inherit the injectivity of $\mathbf{i}$, $\lim\mathbf{q}$ may fail to be surjective. On the other hand, a \emph{long} exact sequence extending the lim-image of (\ref{A}) and comprised of higher limits will conserve the exactness of (\ref{A}); its form is the following:
\begin{align}\label{les} &0\to \text{lim}\,\mathbf{P}_{k}(\varepsilon)\to\text{lim}\,\mathbf{R}_{k}(\varepsilon)\to\text{lim}\,\mathbf{R}_{k}(\varepsilon)/\mathbf{P}_{k}(\varepsilon)\to\text{lim}^1\,\mathbf{P}_{k}(\varepsilon)\to\text{lim}^1\,\mathbf{R}_{k}(\varepsilon)\to\dots \\ &\dots\to \text{lim}^n\,\mathbf{R}_{k}(\varepsilon)\to\text{lim}^n\,\mathbf{R}_{k}(\varepsilon)/\mathbf{P}_{k}(\varepsilon)\to\text{lim}^{n+1}\,\mathbf{P}_{k}(\varepsilon)\to \text{lim}^{n+1}\,\mathbf{R}_{k}(\varepsilon)\to\dots\nonumber\end{align}
(We forego discussion of the connecting morphisms, but these also, like each $\lim^n$, are in the proper sense functorial.) The basic heuristic for this phenomenon is that higher limits array, in group form, the information of inverse systems, in a graded and coordinated fashion (with the caveat that some such information may, in the process, be lost).

We turn now, at long last, to Mitchell's theorem.
\begin{thm}[\cite{rings}] \label{mitchellstheorem} Let $\varepsilon$ be a linear order of cofinality $\aleph_\xi$. If $\xi$ is finite then the cohomological dimension of $\varepsilon$ is $\xi+1$. If $\xi$ is infinite then the cohomological dimension of $\varepsilon$ is $\infty$.
\end{thm}
The theorem holds even with the convention that the cofinality of a partial order which has a maximum element is $\aleph_{-1}$. Mitchell extended his theorem in the following year to the generality of \emph{directed partial orders} $\varepsilon$ \cite{mitchcohdim}; not unrelatedly, the functors $\lim^n$ extend to the category pro-$\mathsf{Ab}$, as described in \cite[Section 15]{SSH}. This family of results organizes, in other words, into a core --- the combinatorics of the cardinals $\aleph_\xi$, particularly when $\xi$ is finite --- and various techniques of extension. Our interest is in that core; as will grow clearer, we regard it as at heart expressing the cofinality interrelations among the ordinals, interrelations which $C$-sequences instantiate. For this reason we will focus on the case of Theorem \ref{mitchellstheorem} when $\varepsilon$ is an ordinal; its extension to linear orders will require little more than a comment in Section \ref{afterbases}.

To sum up: the main content of Mitchell's theorem is that objects like $\mathbf{\Delta}_\varepsilon(\mathbb{Z})$ and $\mathsf{P}(\varepsilon)$ exhibit significant sensitivities to order-theoretic considerations. Our aim is to better understand in what these sensitivities consist. Our guiding interest in all that follows, in other words, is point (i) above, in question form: \textit{What is it in the ordinals --- the ordinals $\omega_n$, in particular --- that these algebraic structures are capturing?}

\section{Tail-acyclic simplicial complexes}\label{goodsection}

It is natural to consider, for a given $\gamma\in\omega_1$, the family of all walks $\text{Tr}(\alpha,\gamma)$ from $\gamma$ down to some $\alpha<\gamma$. Such a family is most concisely conceived, perhaps, as the graph \begin{align}\label{tr2}
\displaystyle\bigcup_{\alpha<\gamma}\text{Tr}^2(\alpha,\gamma)\end{align}
on $\gamma+1$, where $\text{Tr}^2(\alpha,\gamma)$ records the steps of $\text{Tr}(\alpha,\gamma)$ as edges $\{\{\gamma_{i+1},\gamma_i\}\,|\,\allowbreak i<\rho_2(\alpha,\gamma)-1\}$. It is an effect of the fact that $\mathrm{Tr}(\beta,\gamma)$ is an initial segment of any $\mathrm{Tr}(\alpha,\gamma)$ passing through $\beta$, together with the fact that walks' steps are always ``from above'', that any such graph is well-behaved or \emph{good} in the following sense:

\begin{defin}\label{goodgraph}
A graph $G$ on an ordinal $\gamma$ is \textit{good} if 
\begin{enumerate}
\item $G$ is cycle-free, and
\item $G\big|_{[\alpha,\gamma)}$ is connected, for all $\alpha<\gamma$.
\end{enumerate}
\end{defin}
\begin{exa} Consider the following graphs on the ordinals $4$ and $3$, respectively:\\

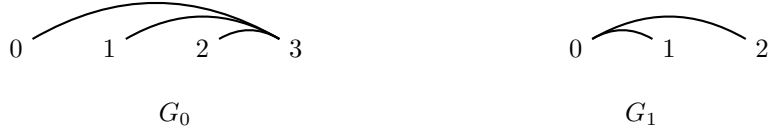
\begin{tikzpicture}[auto, node distance=1.24 cm, every loop/.style={},
                    thick,main node/.style={}]
  \node[main node] (0) {};
  \node[main node] (1) [right of=0] {0};
  \node[main node] (2) [right of=1] {1};
  \node[main node] (3) [right of=2] {2};
  \node[main node] (4) [right of=3] {3};
  \node[main node] (5) [right of=4] {};
\node[main node] (6) [below right of=2] {$G_0$};
  \node[main node] (7) [right of=5] {};
  \node[main node] (8) [right of=7] {0};
  \node[main node] (9) [right of=8] {1};
  \node[main node] (10) [right of=9] {2};
\node[main node] (11) [below right of=8] {$G_1$};
   
  \path[every node/.style={font=\sffamily\small}]
    (1) edge [bend left] node[left] {} (4)
    (2)  edge [bend left] node[left] {} (4)
    (3) edge [bend left] node {} (4)
    (8) edge [bend left] node {} (9)
          edge [bend left] node {} (10);
\end{tikzpicture}

\noindent $G_0$ is a good graph. On the other hand, $G_1\big|_{[1,3)}$ is disconnected, so $G_1$ is not good.
\end{exa}

In fact, $G_1$ is the \textit{forbidden configuration}: \textit{a connected graph is good if and only if it contains no copy of $G_1$ $($i.e., contains no $\{\{\alpha,\beta\},\{\alpha,\gamma\}\}$, for $\alpha<\beta<\gamma)$}. A consequence is the following theorem, one measure of the difficulty of extending the technique of minimal walks beyond the countable ordinals.

\begin{thm} \label{omega1} $\omega_1$ is the least ordinal admitting no good graph. \end{thm}

\begin{proof} Suppose for contradiction that $\omega_1$ admitted a good graph $G$. As $G$ is connected, there exists for each $\gamma\in\text{Lim}\cap\,\omega_1$ some least $\gamma_1\geq\gamma$ such that $\{\xi,\gamma_1\}\in G$ for some $\xi<\gamma$. Let $\gamma_0$ denote the least such $\xi$. The function $\gamma\mapsto\gamma_0$ is then a regressive function and hence, by the Pressing Down Lemma, constantly $\alpha$ on some stationary $S\subseteq\text{Lim}\cap\,\omega_1$. For any $\beta$ and $\gamma$ in $S$ above $\alpha$ with $\beta_1\neq\gamma_1$, then, $\{\{\alpha,\beta_1\},\{\alpha,\gamma_1\}\}$ is a copy of $G_1$ in $G$ --- a contradiction.

On the other hand, (\ref{tr2}) defines a good graph on any countable successor ordinal, and a natural variant of its definition handles the countable limit case as well. In fact, the more elementary $\{\{\alpha,\min(C_\gamma\backslash (\alpha+1))\}\,|\,\alpha<\gamma\}$ defines a good graph on any $\gamma$ of countable cofinality (with $G_0$, above, as a simple instance).
\end{proof}

These phenomena generalize.

\begin{defin}\label{goodndsimpcomp}
An $n$-dimensional simplicial complex $G$ on an ordinal $\gamma$ is \textit{tail-acyclic} if $G^{n-1}=[\gamma]^n$ and for all $\alpha<\gamma$ and $k\geq 0$,
\[
  \tilde{\mathrm{H}}_k^\Delta\big(G\big|_{[\alpha,\gamma)}\big) = 0.
\]
\end{defin}
Tail-acyclic $n$-dimensional $G$ on $\gamma$, in other words, have a complete $(n-1)$-skeleton and are connected and acyclic on any tail of $\gamma$; good graphs are simply the $n=1$ case of this definition.\footnote{The requirement of a complete $(n-1)$-skeleton in Definition \ref{goodgraph} streamlines the argument of Theorem \ref{free}; if every $n$-dimensional $G$ satisfying $\tilde{\mathrm{H}}_k^\Delta\big(G\big|_{[\alpha,\gamma)}\big) = 0$ for all $\alpha<\gamma$ and $k\geq 0$ \emph{extends to} a tail-acyclic $G'\supseteq G$ (i.e., extends to one with a complete $(n-1)$-skeleton) then this requirement is unnecessary. This may in fact be the case, as it clearly is when $n=2$, but for more general $n$ it is at the very least quite tedious to rigorously argue. Hence for now, for simplicity's sake, we adopt this requirement and record the extension problem as one of our concluding questions.\label{footnote}}

\begin{thm} \label{omegan}
Let $n$ be a positive integer. Then $\omega_n$ is the least ordinal supporting no $n$-dimensional tail-acyclic simplicial complex.
\end{thm}

In particular, there is some least number of dimensions --- namely, $n+1$
 --- in which $\omega_n$ \textit{can} support a tail-acyclic simplicial complex.

We will argue Theorem \ref{omegan} by way of an algebraic translation, which we motivate as follows:

\begin{exa}\label{iplus}
Recall that $P_1(\omega)=\bigoplus_{[\omega]^2}\mathbb{Z}$, and let $I=\{\langle i,i+1\rangle\,|\,i\in\omega\}$. For every $j<k$ in $\omega$, the difference $\langle k\rangle - \langle j\rangle$ has a unique $d_1 I$-decomposition $$\sum_{j\leq i<k}(\langle i+1\rangle-\langle i\rangle).$$ In other words, $d_1 I$ is a basis for $d_1 P_1(\omega)$. Pictorially, edges $\{i,i+1\}$ connect the points $j$ and $k$ as below:
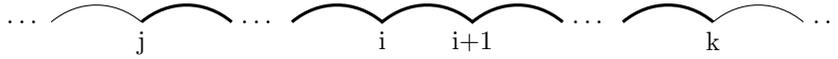
\vspace{.35 cm}

\begin{center}
\begin{tikzpicture}
\draw (0,0) to [out=40,in=140] (1.2,0);
\draw [very thick] (1.2,0) to [out=40,in=140] (2.4,0);
\node [left] at (0,0) {\dots};
\node [below] at (1.2,0) {j};
\node [right] at (2.4,0) {\dots};
\draw [very thick] (3.2,0) to [out=40,in=140] (4.4,0) to [out=40,in=140] (5.6,0) to [out=40,in=140] (6.8,0);
\node [right] at (6.8,0) {\dots};
\draw [very thick] (7.6,0) to [out=40,in=140] (8.8,0);
\draw (8.8,0) to [out=40,in=140] (10,0);
\node [right] at (10,0) {\dots};
\node [below] at (4.4,0) {i};
\node [below] at (5.6,0) {i+1};
\node [below] at (8.8,0) {k};
\end{tikzpicture}
\end{center}

\noindent These edges evidently define a unique path between any two points in $\omega$. Put differently, $d_1I$ defines a good, or tail-acyclic, graph $G_I$ on $\omega$. More precisely, the spanning and linear independence properties of $d_1 I$ manifest as the connectedness of, and lack of cycles in, $G_I$, respectively. ``Goodness'' captures the fact that these properties persist on any restriction of $d_1 I$ and $d_1 P_1(\omega)$ to a tail $[n,\omega)$ of $\omega$ --- the fact, in other words, that $d_1 I$ defines a basis for the \textit{inverse system} $\mathbf{d}_1\mathbf{P}_1(\omega)$.

These seemingly rudimentary considerations are surprisingly sensitive to the cofinality of the index-set of $\mathbf{d}_n\mathbf{P}_n(\gamma)$. For example: by Theorems \ref{omega1} and \ref{free} below, the least ordinal $\gamma$ for which $\mathbf{d}_1\mathbf{P}_1(\gamma)$ is not free is $\omega_1$; in fact, $\mathbf{d}_1\mathbf{P}_1(\omega_1)$ is not even projective.
\end{exa}

\begin{thm}\label{free} For $n\geq 1$, the system $\mathbf{d}_n\mathbf{P}_n(\gamma)$ is free if and only if $\gamma$ admits a tail-acyclic $n$-dimensional simplicial complex.
\end{thm}

\begin{proof}
For the forward direction of the proof, suppose that $\mathbf{d}_n\mathbf{P}_n(\gamma)$ is free. We will want the following fact:
\begin{fac}\label{fackt} If $\mathbf{d}_n\mathbf{P}_n(\gamma)$ is projective and $\text{cf}(\gamma)=\aleph_\xi$, then $\xi<n$.
\end{fac}
\noindent This fact is immediate from Theorem \ref{concludestheproof} together with Lemma \ref{projlemma}, or Proposition \ref{notproject} below. In the following section, we construct for any $\mathbf{d}_n\mathbf{P}_n(\gamma)$ as in Fact \ref{fackt} a basis of the form $d_n\mathcal{B}=\{d_n\langle\vec{\alpha}\rangle\,|\,\vec{\alpha}\in B\}$, with $B\subseteq[\gamma]^{n+1}$. Let $\mathbf{d}_n\mathbf{P}_n(\gamma)$ be free and fix such a basis, and write $\underline{B}$ for the $\subseteq$-downward closure of $B$. In other words, $\underline{B}$ is the natural interpretation of $B$ as a simplicial complex. We show that $\underline{B}$ is tail-acyclic.
As $d_n\mathcal{B}$ is linearly independent,
$$\text{ker}(\partial_n:C_n(\underline{B})\rightarrow C_{n-1}(\underline{B}))=0$$
hence $\tilde{\mathrm{H}}^{\Delta}_n(\underline{B})=0$. 
\begin{clm} $\underline{B}^{n-1}=[\gamma]^n$.
\end{clm}
\begin{proof} Towards contradiction, suppose instead that $\vec{\beta}^i\in[\gamma]^n\backslash\underline{B}^{n-1}$ for some $\vec{\beta}\in [\gamma]^{n+1}$. Then no linear combination of elements of $d_n\mathcal{B}$ can supply the summand $\langle\vec{\beta}^i\rangle$ of $d_n\langle\vec{\beta}\rangle$, hence $d_n\mathcal{B}$ does not span $\mathbf{d}_n\mathbf{P}_n(\gamma)$.
\end{proof}
\noindent By the claim, $\underline{B}^k=[\gamma]^{k+1}$ for all $k<n$. 
Therefore
\begin{enumerate}
\item[(i)] $\underline{B}$ is connected: $\tilde{\mathrm{H}}^\Delta_0(\underline{B})=0$, and
\item[(ii)] $\tilde{\mathrm{H}}^\Delta_k(\underline{B})$ for $0<k<n$ is nothing other than the homology of the chain complex
$$P_n(\gamma)\stackrel{d_{n}}{\longrightarrow}P_{n-1}(\gamma)\stackrel{d_{n-1}}{\longrightarrow}\dots\stackrel{d_2}{\longrightarrow} P_1(\gamma)\stackrel{d_1}{\longrightarrow} P_{0}(\gamma)$$
As noted in Section \ref{3210}, this sequence is exact, so $\tilde{\mathrm{H}}^\Delta_k(\underline{B})=0$. By definition these arguments hold on any tail of $\gamma$; in consequence, $\underline{B}$ is tail-acyclic.
\end{enumerate}
For the reverse direction of the proof, simply observe that the above argument is reversible. In other words, any tail-acyclic $n$-dimensional simplicial complex is determined by its collection, $B$, of $n$-faces, which in turn define a basis $d_n\mathcal{B}=\{d_n\langle\vec{\alpha}\rangle\,|\,\vec{\alpha}\in B\}$ for $\mathbf{d}_n\mathbf{P}_n(\gamma)$, just as above.

To see that $d_n\mathcal{B}$ spans $\mathbf{d}_n\mathbf{P}_n(\gamma)$, observe that any $$\sum_{i<k} b_i d_n\langle\vec{\beta}_i\rangle\in d_n P_n([0,\gamma))$$ may be identified with an element $e$ of $C_{n-1}(\underline{B})$. Since $\partial_{n-1} e=0$ and $\tilde{\mathrm{H}}^\Delta_{n-1}(\underline{B})=0$, we must have $\partial_n f=e$ for some $f\in C_n(\underline{B})$; this $\partial_n f$ naturally corresponds to a linear combination of elements of $d_n\mathcal{B}$, showing that the latter indeed does span the group $d_n P_n([0,\gamma))$. Since $\underline{B}$ is tail-acyclic, this argument applies to the restriction of $\underline{B}$ to any tail of $\gamma$, hence $d_n\mathcal{B}$ spans the \emph{inverse system} $\mathbf{d}_n\mathbf{P}_n(\gamma)$, as desired.

To see that $d_n\mathcal{B}$ is linearly independent, observe that $$\sum_{i<k} a_i d_n\langle\vec{\alpha}_i\rangle=0\hspace{.4 cm}\text{ implies }\hspace{.4 cm}\sum _{i<k} a_i \langle\vec{\alpha}_i\rangle\in\mathrm{ker}(\partial_n),$$
via an identification just as above; in consequence, if the $\langle\vec{\alpha}_i\rangle$ are all distinct, then the coefficients $a_i$ must all equal zero, since $\tilde{\mathrm{H}}^\Delta_n(\underline{B})=\mathrm{ker}(\partial_n)=0$.
\end{proof}

Theorem \ref{omegan} then takes the following form:

\begin{thm}\label{eight} For $n\geq 1$, $\omega_n$ is the least ordinal $\varepsilon$ such that $\mathbf{d}_n\mathbf{P}_n(\varepsilon)$ is not free.
\end{thm}

Theorem \ref{eight} is in fact true even for $n=0$. This theorem, like \ref{omega1} and \ref{omegan} above, conjoins both positive and negative statements, namely that \begin{enumerate}
\item $\mathbf{d}_n\mathbf{P}_n(\varepsilon)$ \textit{does} admit a basis, for $\varepsilon<\omega_n$, while 
\item $\mathbf{d}_n\mathbf{P}_n(\omega_n)$ does not.
\end{enumerate}
We argue (1) and (2), respectively, in Sections \ref{bases} then \ref{afterbases} and \ref{deromega1} below. To show (1), we define an explicit basis for each eligible $\mathbf{d}_n\mathbf{P}_n(\varepsilon)$; this we do by an expanded, or compound, use of $C$-sequences. We then describe ``pullbacks'' $\mathtt{f}_n$ of these $n$-dimensional basis-systems which exhibit nontrivial $n$-dimensional coherence relations, and thereby witness the (hitherto abstract) fact that $\text{pd}(\mathbf{\Delta}_{\omega_n}(\mathbb{Z}))>n$; from this point (2) follows. In the process, we will have shown the core of Mitchell's theorem and a bit more: for each $\varepsilon<\omega_n$, that theorem only implies that $\mathbf{d}_n\mathbf{P}_n(\varepsilon)$ is \emph{projective}, whereas just below, we show in a very concrete fashion that it is free.

\section{Bases from $C$-sequences}\label{bases}

In this section, we will define for each $\varepsilon$ of cofinality $\aleph_k$ and positive $n>k$ a $\mathcal{B}\subseteq\mathbf{P}_n(\varepsilon)$ such that $d_n\mathcal{B}$ is a basis for $\mathbf{d}_n\mathbf{P}_n(\varepsilon)$. Clubs $C_\beta$ on $\beta\in\varepsilon\cap\{\gamma\,|\,\text{cf}(\gamma)<\text{cf}(\varepsilon)\}$ will structure the construction of $\mathcal{B}$. Hence, as above, we begin by fixing for each relevant $\beta$ a closed cofinal $C_\beta\subseteq\beta$ such that $\text{otp}(C_\beta)=\text{cf}(\beta)$, and such that for any limit $\beta$ and $i<\text{cf}(\beta)$ the cofinality of the $i^{th}$ element of $C_\beta$ is $\text{cf}(i)$. In particular, $C_0=\varnothing$; we will also assume that $C_\kappa=\kappa$ for any infinite regular cardinal $\kappa$.
\begin{defin}\label{cmpdlddrs} For any $\beta\in C_\gamma$ define $C_{\beta\gamma}$ to be $\pi^{-1}(C_\alpha)$, where $\alpha=\text{otp}\,(C_\gamma\cap\beta)$ and $\pi$ is the order-isomorphism $C_\gamma\cap\beta \rightarrow \alpha$.

One may continue in this fashion, defining $C_{\vec{\gamma}}$ by induction on the length of $\vec{\gamma}$ much as above: Suppose $C_{\vec{\gamma}}$ is defined and $\beta\in C_{\vec{\gamma}}$. Let $\alpha=\text{otp}(C_{\vec{\gamma}}\cap\beta)$ and let $\pi:C_{\vec{\gamma}}\cap\beta \rightarrow \alpha$ be the order-isomorphism. Define $C_{\beta\vec{\gamma}}$ to be $\pi^{-1}(C_\alpha)$.

Let $C^\gamma(\alpha)=\text{min}(C_\gamma\backslash(\alpha+1))$; more generally, let $C^{\vec{\gamma}}(\alpha)=\text{min}(C_{\vec{\gamma}}\backslash(\alpha+1))$.

Finally, define $\vec{\alpha}=(\alpha_0,\dots,\alpha_n)$ to be \emph{internal to} $C_{\beta}$ iff

\begin{enumerate}
\item $\alpha_n\in C_\beta$,
\item $\alpha_i\in C_{\alpha_{i+1}\dots\alpha_n\beta}$ for all $i<n$, and
\item $\text{cf}(\alpha_0)\leq\text{cf}(\alpha_i)<\text{cf}(\alpha_j)<\text{cf}(\beta)$ for $0<i<j\leq n$.
\end{enumerate}
Observe that if $\beta=\omega_k$ then its role in points (1)-(3) is superficial; $C_{\alpha_{i+1}\dots\alpha_n\beta}=C_{\alpha_{i+1}\dots\alpha_n}$, for example. We will sometimes in this case omit mention of $C_\beta$, terming an $\vec{\alpha}$ as above \emph{internal}, simply.
\end{defin}

\begin{obs}\label{bulletedobservation} The following observations are straightforward:\begin{itemize}
\item If $\varepsilon$ is a successor ordinal, then $C_{\alpha\varepsilon}=\varnothing$ in the one case in which it is defined, namely when $\varepsilon=\alpha+1$.
\item If $\beta_0$ is a limit ordinal above $\alpha$, then $C^{\vec{\beta}}(\alpha)$ is defined and is a successor ordinal.
\item If $\alpha\in \mathrm{Lim}\cap C_{\vec{\beta}}$ then $C_{\alpha\vec{\beta}}$ is a closed unbounded subset of $\alpha$, of ordertype $\mathrm{cf}(\alpha)$.
\item More generally, if $\vec{\beta}$ is a tail of $\vec{\gamma}$ then if $C_{\vec{\gamma}}$ is defined then $C_{\vec{\beta}}$ is as well and $C_{\vec{\gamma}}\subseteq C_{\vec{\beta}}$.
\item If $\varepsilon$, the largest cardinal involved, is of cofinality $\aleph_k$, then the recursions of Definition \ref{cmpdlddrs} are only meaningful for $k+2$ many steps. \textit{$\vec{\beta}$ is internal to $C_\varepsilon$}, in particular, implies that $|\vec{\beta}|\leq k+2$.
\end{itemize}
\end{obs}

The next definition is crucial.

\begin{defin}\label{thedef} For $\varepsilon$ of cofinality $\aleph_k$ and positive $n>k$, let $\mathcal{B}_n(\varepsilon)$ denote the collection of $\langle\vec{\alpha},\vec{\beta}\rangle$ satisfying the following:
\begin{enumerate}
\item $\vec{\alpha}\in [\varepsilon]^{i+1}$ for some $i<n$.
\item $\vec{\beta}\in[\varepsilon]^{n-i}\text{ is internal to }C_\varepsilon$.
\item $C^{\vec{\beta}^0\varepsilon}(\alpha_i)=\beta_0$.
\end{enumerate}
Note as in Observation \ref{bulletedobservation} that our $\mathcal{C}$-sequence conventions entail that any such $\beta_0$ is a successor ordinal. Where we wish to emphasize the choice of the parameter $C_\varepsilon$ in the above definition, we write $\mathcal{B}_n(\varepsilon)[C_\varepsilon]$.\end{defin}
\begin{exa}  \label{succB} For any $n>0$ and $\varepsilon$ a successor ordinal, e.g., $\varepsilon=\delta+1$,
$$\mathcal{B}_n(\varepsilon)=\{\langle\vec{\alpha},\delta\rangle\,|\,\vec{\alpha}\in[\delta]^n\}.$$
Here there is only one possibility for $C_\varepsilon$, and the only $\vec{\beta}$ which is internal to $C_\varepsilon$ is the $1$-tuple $\delta$. Hence the $C^{\vec{\beta}^0\varepsilon}(\alpha_i)$ of Definition \ref{thedef} is constantly equal to $C^\varepsilon(\alpha_i)=\delta$ for all $\vec{\alpha}<\vec{\beta}$.
\end{exa}

\begin{lem} If $\varepsilon$ is a successor ordinal and $n>0$ then $d_n\mathcal{B}_n(\varepsilon)$ is a basis for $\mathbf{d}_n\mathbf{P}_n(\varepsilon)$.
\end{lem}
\begin{proof}
Here and below it will suffice to check that $d_n\mathcal{B}_n(\varepsilon)$ is linearly independent and decomposes each generator $d_n\langle\vec{\beta}\rangle$ of $\mathbf{d}_n\mathbf{P}_n(\varepsilon)$. Again let $\varepsilon=\delta+1$. Linear independence follows from the fact that for any collection of pairwise distinct $\vec{\alpha}_j$, \begin{align}\label{firstreasoning}\sum_{j\in J} z_j d_n\langle\vec{\alpha}_j,\delta\rangle=0\hspace{.4 cm}\text{ implies }\hspace{.4 cm}\sum_{j\in J} z_j \langle\vec{\alpha}_j\rangle=0\hspace{.4 cm}\text{ implies }\hspace{.4 cm}z_j=0\text{ for all }j\in J.\end{align}
Hence the only $d_n\mathcal{B}_n(\varepsilon)$ decomposition of $0$ is the trivial one. Therefore, since $d_n d_{n+1}\langle \vec{\beta},\delta\rangle=0$ for any $\vec{\beta}\in [\delta]^{n+1}$,
\begin{align*}d_n\langle \vec{\beta}\rangle=\sum_{i=0}^{n-1}(-1)^{n+i} d_n\langle \vec{\beta}^i,\delta\rangle
\end{align*}
is the unique $d_n\mathcal{B}_n(\varepsilon)$ decomposition of $d_n\langle \vec{\beta}\rangle$. As any other $\vec{\beta}\in [\varepsilon]^{n+1}$ contains $\delta$ and consequently falls in $\mathcal{B}_n(\varepsilon)$, this completes the argument.
\end{proof} 

\begin{thm}\label{basis} Fix a positive integer $n$. Then for every ordinal $\varepsilon$ for which $\mathrm{cf}(\varepsilon)<\aleph_n$, the collection $d_n(\mathcal{B}_n(\varepsilon)[C_\varepsilon])$ is a basis for $\mathbf{d}_n \mathbf{P}_n(\varepsilon)$. In particular, every such $\mathbf{d}_n \mathbf{P}_n(\varepsilon)$ is free.
\end{thm}

\begin{proof} The proof is by induction on $\varepsilon$. Denote the following inductive hypothesis $\mathrm{IH}(\varepsilon)$:
\bigskip

{\centering
  \emph{If $\delta<\varepsilon$ and $\mathrm{cf}(\delta)=\aleph_k$ and $n>\text{max}\,(0,k)$, then for any club $C_\delta$ on $\delta$,}
  
  \emph{$d_n(\mathcal{B}_n(\delta)[C_\delta])$ is a basis for $\mathbf{d}_n \mathbf{P}_n(\delta)$.}\par
}
\bigskip

\noindent Notice that if $\varepsilon$ is a limit ordinal and $\mathrm{IH}(\xi)$ holds for all $\xi<\varepsilon$, then $\mathrm{IH}(\varepsilon)$ holds. Hence we only need to show that $\mathrm{IH}(\varepsilon)$ implies $\mathrm{IH}(\varepsilon+1)$. If $\varepsilon$ is a successor ordinal, then $\mathrm{IH}(\varepsilon+1)$ follows from $\mathrm{IH}(\varepsilon)$ by the preceding example and lemma. If $\varepsilon$ is a limit ordinal of cofinality greater than $\aleph_\omega$, then $\mathrm{IH}(\varepsilon+1)$ and $\mathrm{IH}(\varepsilon)$ are equivalent assertions. This leaves just one case of interest: limit $\varepsilon$ of cofinality less than $\aleph_\omega$.

First, a lemma:

\begin{lem} \label{deleps} For $\delta\in\mathrm{Lim}\cap C_\varepsilon$,
$$\mathcal{B}_{n-1}(\delta)[C_{\delta\varepsilon}]=\{\langle\vec{\alpha}\rangle\,|\,\langle\vec{\alpha},\delta\rangle\in\mathcal{B}_n(\varepsilon)[C_\varepsilon]\}$$
\end{lem}
\begin{proof}[Proof of Lemma] Term the longest proper tail-segment of $\vec{\alpha}$ which is internal to $C_\varepsilon$ the \textit{$C_\varepsilon$-tail of $\vec{\alpha}$}. Write $(\vec{\beta},\delta)$ for the $C_\varepsilon$-tail of $\langle\vec{\alpha},\delta\rangle\in\mathcal{B}_n(\varepsilon)[C_\varepsilon]$. As $\delta$ is a limit, $\vec{\beta}\neq\varnothing$. Moreover, $\vec{\beta}$ is the $C_{\delta\varepsilon}$-tail of $\langle\vec{\alpha}\rangle$ if and only if $(\vec{\beta},\delta)$ is the $C_\varepsilon$-tail of $\langle\vec{\alpha},\delta\rangle$. The lemma follows.\end{proof}

We return to the proof of the theorem. Assume $\mathrm{IH}(\varepsilon)$, with $\varepsilon$ a limit ordinal of cofinality less than $\aleph_\omega$. Enumerate the elements of $C_\varepsilon$ as $\{\eta_i^\varepsilon\,|\,i\in\text{cf}(\varepsilon)\}$.
\medskip

\noindent\textbf{Claim 1.} $d_n\mathcal{B}_n(\varepsilon)$ is linearly independent.

\begin{proof}[Proof of Claim 1]
Towards contradiction, suppose instead that
\begin{align} \label{claim1} \sum_{j<\ell}z_j d_n\langle\vec{\alpha}_j,\beta_j\rangle=0
\end{align}
for nonzero coefficients $z_j$ $(j<\ell)$ and $\{\langle\vec{\alpha}_j,\beta_j\rangle\,|\,j<\ell\}\subset\mathcal{B}_n(\varepsilon)$. Let $\delta=\mathrm{max}\{\beta_j\,|\,j<\ell\}$, and let $J=\{j\,|\,\beta_j=\delta\}$. Note that all $\beta_j$ are elements of $C_\varepsilon$. In particular, $\delta$ is an element of $C_\varepsilon$. Note also that (without loss of generality) the $\vec{\alpha}_j$ indexed by $J$ are all distinct. By (\ref{claim1}),
\begin{align} \label{claim2}
\sum_{j\in\ell\backslash J}z_j d_n\langle\vec{\alpha}_j,\beta_j\rangle+\sum_{j\in J} z_j\langle d_{n-1}\vec{\alpha}_j,\delta\rangle+(-1)^n\sum_{j\in J} z_j\langle\vec{\alpha}_j\rangle=0
\end{align}
\underline{Case 1}: $\delta$ is a limit ordinal. Then together with the induction hypothesis, $\sum_J z_j\langle d_{n-1}\vec{\alpha}_j,\delta\rangle=0$, which follows from (\ref{claim2}), contradicts Lemma \ref{deleps}.

\medskip
\noindent\underline{Case 2}: $\delta$ is a successor ordinal: $\delta=\eta^{\varepsilon}_{i+1}$, for some $i<\mathrm{cf}(\varepsilon)$. Hence $(\vec{\alpha}_j)_{n-1}\geq\eta_i^\varepsilon$ for $j\in J$. By equation (\ref{claim2}),
\begin{align}\label{friday}
\sum_{j\in\ell\backslash J}z_j d_n\langle\vec{\alpha}_j,\beta_j\rangle+(-1)^n\sum_{j\in J} z_j\langle\vec{\alpha}_j\rangle=0
\end{align}
By definition, $\beta_j\leq\eta_i^\varepsilon$ for $j\in \ell\backslash J$. This implies that $(\vec{\alpha}_j)_{n-1}\leq\eta_i^\varepsilon$ for $j\in J$. Hence $(\vec{\alpha}_j)_{n-1}=\eta_i^\varepsilon$ for $j\in J$. Therefore the $\vec{\alpha}_j^{n-1}$ indexed by $J$ are all distinct. By (\ref{claim2}), though,
\begin{align*}
0 & = \sum_{j\in J} z_j d_{n-1}\langle \vec{\alpha}_j\rangle
\end{align*}
and we may conclude as in equation (\ref{firstreasoning}) that $\sum_{j\in J} z_j \langle \vec{\alpha}_j^{n-1}\rangle=0$, a contradiction.
\end{proof}
\noindent\textbf{Claim 2.} $d_n\mathcal{B}_n(\varepsilon)$ generates $\mathbf{d}_n \mathbf{P}_n(\varepsilon)$.
\begin{proof}[Proof of Claim 2]

We argue by induction on $\delta\in C_\varepsilon$. Let
\begin{align*}
\mathcal{B}_n(\varepsilon)\big|_{\delta}\:=\{\langle\vec{\alpha}\rangle\in\mathcal{B}_n(\varepsilon)\,|\,\alpha_n<\delta\}
\end{align*}
We show that if $d_n\mathcal{B}_n(\varepsilon)\big|_{\gamma+1}$ generates $\mathbf{d}_n \mathbf{P}_n(\gamma+1)$ for all $\gamma\in\delta\cap C_\varepsilon$, then $d_n\mathcal{B}_n(\varepsilon)\big|_{\delta+1}$ generates $\mathbf{d}_n \mathbf{P}_n(\delta+1)$.

The base case, $\delta=\eta_0^\varepsilon$, is exactly as in Example \ref{succB}.

\medskip
\noindent\underline{Case 1}: $\delta$ is a limit ordinal. Consider then $d_n\langle\vec{\alpha},\delta\rangle\in\mathbf{d}_n \mathbf{P}_n(\delta+1)$. Let $$\sum_{j<\ell} z_j d_{n-1}\langle\vec{\beta}_j\rangle$$ be the $\mathcal{B}_{n-1}(\delta)[C_{\delta\varepsilon}]$ decomposition of $d_{n-1}\langle\vec{\alpha}\rangle$. Then
\begin{align*}d_{n}\langle\vec{\alpha},\delta\rangle 
& = \langle d_{n-1}\langle\vec{\alpha}\rangle,\delta\rangle+(-1)^n\langle\vec{\alpha}\rangle \\
 & = \sum_{j<\ell} z_j d_n\langle\vec{\beta_j},\delta\rangle+(-1)^n(\sum_{j<\ell} z_j\langle\vec{\beta}_j\rangle+\langle\vec{\alpha}\rangle)\end{align*}
By Lemma \ref{deleps}, the left-hand summands of the last line are all from $d_n\mathcal{B}_n(\varepsilon)\big|_{\delta+1}$, while the rightmost sum is in $\mathbf{d}_n \mathbf{P}_n(\eta+1)$ for some $\eta\in\delta\cap C_\varepsilon$. By our induction hypothesis, this concludes Case 1.

\medskip
\noindent\underline{Case 2}: $\delta$ is a successor ordinal. Let $\delta=\eta_{i+1}^\varepsilon$. Consider $\vec{\beta}\in[\delta]^{n}$. If \mbox{$\beta_{n-1}\geq\eta_i^\varepsilon$} then $\langle\vec{\beta},\delta\rangle\in\mathcal{B}_n(\varepsilon)$. If $\beta_{n-1}<\eta_i^\varepsilon$, then since $d_n d_{n+1}\langle\vec{\beta},\eta_i^\varepsilon,\delta\rangle=0$,
\begin{align} \label{form1} d_n\langle\vec{\beta},\delta\rangle=(-1)^{n-1}d_n\langle d_{n-1}\vec{\beta},\eta_i^\varepsilon,\delta\rangle+d_n\langle\vec{\beta},\eta_i^\varepsilon\rangle\text{.}\end{align}
Again the rightmost summand decomposes by hypothesis, while each summand of $\langle d_{n-1}\vec{\beta},\eta_i^\varepsilon,\delta\rangle$ is from $\mathcal{B}_n(\varepsilon)\big|_{\delta+1}$.

We will be done if we show that, for any $\vec{\alpha}\in[\delta]^{n+1}$ with $\alpha_{n}>\eta_i^\varepsilon$, $d_n\langle\vec{\alpha}\rangle$ has a $\mathcal{B}_n(\varepsilon)\big|_{\delta+1}$ decomposition. Again, though, since $d_nd_{n+1}\langle\vec{\alpha},\delta\rangle=0$,
$$d_n\langle\vec{\alpha}\rangle=\sum_{j=0}^{n}(-1)^{n+j+1}d_n\langle\vec{\alpha}^j,\delta\rangle$$
and all summands on the right are as discussed above: either of type (\ref{form1}), or from $\mathcal{B}_n(\varepsilon)\big|_{\delta+1}$ directly. This concludes the proof of Claim 2. 
\end{proof} Together with the induction hypothesis $\mathrm{IH}(\varepsilon)$, Claims 1 and 2 establish $\mathrm{IH}(\varepsilon+1)$. This concludes the proof of Theorem \ref{basis}.
\end{proof}

Lemma \ref{deleps} in the above argument bears comparison with square principles (see again  \cite{jensenfine, todpairs, cummings}): structuring both is a certain uniformity at the limit points $C_\varepsilon':=\{\delta\in\varepsilon\mid \sup(C_\varepsilon\cap\delta)=\delta\}$ of a club $C_\varepsilon\subseteq \varepsilon$. In $\square(\kappa)$, this condition takes the form $$C_\varepsilon\cap\delta=C_\delta\text{ for all }\delta\in C'_\varepsilon$$
In our basis construction, it comes at the cost of an additional coordinate:
$$\mathcal{B}_{n-1}(\delta)[C_{\delta\varepsilon}]=\{\langle\vec{\alpha}\rangle\,|\,\langle\vec{\alpha},\delta\rangle\in\mathcal{B}_n(\varepsilon)[C_\varepsilon]\}\text{ for all }\delta\in C'_\varepsilon$$
Moreover, whenever $\delta$ is in $C_\varepsilon$ and $\gamma\in\text{Lim}\cap C_{\delta\varepsilon}$ then
$$\mathcal{B}_{n-2}(\gamma)[C_{\gamma\delta\varepsilon}]=\{\langle\vec{\alpha}\rangle\,|\,\langle\vec{\alpha},\gamma\rangle\in\mathcal{B}_{n-1}(\delta)[C_{\delta\varepsilon}]\}=\{\langle\vec{\alpha}\rangle\,|\,\langle\vec{\alpha},\gamma,\delta\rangle\in\mathcal{B}_n(\varepsilon)[C_\varepsilon]\}$$
Hence these additional coordinates accrue. In other words, more room is needed to carry out the construction on higher cofinality $\varepsilon$; this is one heuristic for the associated rise in cohomological dimension. Internal tails record these accruing coordinates and are the key to our further constructions. More particularly, $C_\varepsilon$-internal tails organize the $d_n\mathcal{B}_n(\varepsilon)$-decomposition of $\mathbf{d}_n\mathbf{P}_n(\varepsilon)$ to such a degree that the associated map $\mathbf{d}_n\mathbf{P}_n(\varepsilon)\rightarrow \mathbf{P}_n(\varepsilon)$ (see equation \ref{sdefinition} below) extends to all of $\mathbf{P}_{n-1}(\varepsilon)$, as we describe in the following section. The basic principle is the following: a fact used in the proof of Lemma \ref{deleps} is that any $\langle\vec{\gamma}\rangle$ has a maximal proper internal tail $\vec{\beta}$. Hence $\langle\vec{\gamma}\rangle=\langle\vec{\alpha},\vec{\beta}\rangle$ for some $i$ and $\vec{\alpha}\in [\varepsilon]^{i+1}$, and $\langle\vec{\gamma}\rangle$ then has some ``nearest'' basis element $\mathtt{b}(\vec{\gamma})$, if $C^{\vec{\beta}\varepsilon}(\alpha_i)$ is defined:
\begin{defin}\label{thefunctionb} Given an $\varepsilon$ and $C_\varepsilon$ as in the above construction, the \textit{maximal proper internal tail} of $\langle\vec{\gamma}\rangle\in\mathbf{P}_n(\varepsilon)$ is the longest tail $\vec{\beta}$ of $\vec{\gamma}$ which is internal to $C_\varepsilon$ and not all of $\vec{\gamma}$. In this case,  if $C^{\vec{\beta}\varepsilon}(\alpha_i)$ is defined, let $$\mathtt{b}(\vec{\gamma})=\langle\vec{\alpha},C^{\vec{\beta}\varepsilon}(\alpha_i),\vec{\beta}\rangle$$If $C^{\vec{\beta}\varepsilon}(\alpha_i)$ is not defined, let $\mathtt{b}(\vec{\gamma})=0$.
\end{defin}
The function $\mathtt{b}$ will feature centrally in the following sections. Note that every element of $\mathcal{B}_n(\varepsilon)$ is of the form $\mathtt{b}(\vec{\gamma})$ for some $\vec{\gamma}\in [\varepsilon]^n$. Observe finally that Theorem \ref{basis} implies the ``upper bound half'' of Mitchell's theorem (due, in fact, originally to Goblot \cite{Goblot}):
\begin{cor}\label{goblotsthm} If the cofinality of an ordinal $\varepsilon$ is $\aleph_k$ then the cohomological dimension of $\varepsilon$ is at most $k+1$.
\end{cor}

\section{Translation to \textit{mod finite} settings}\label{afterbases}

In this section we describe how the basis constructions of Section \ref{bases} induce functions whose nontriviality implies the ``lower bound half,'' and hence the entirety, of  Theorems \ref{mitchellstheorem} (Mitchell's),  \ref{omegan}, and  \ref{eight}. We then define these functions explicitly; we argue their nontriviality in Section \ref{deromega1} below.

\subsection{The argument of the remainder of Mitchell's theorem}\label{theargument}

We begin by returning our attention to the projective resolution
\begin{equation*}\dots\longrightarrow\mathbf{P}_{n+1}(\omega_n)\stackrel{\mathbf{d}_{n+1}}{\longrightarrow}\mathbf{P}_{n}(\omega_n)\stackrel{\mathbf{d}_{n}}{\longrightarrow}\dots \stackrel{\mathbf{d}_1}{\longrightarrow} \mathbf{P}_{0}(\omega_n)\stackrel{\mathbf{d}_0}{\longrightarrow}\mathbf{\Delta}_{\omega_n}(\mathbb{Z})\longrightarrow\mathbf{0}.\end{equation*}
We ``telescope'' the interval $\mathbf{P}_{n+1}(\omega_n)\to\mathbf{P}_{n}(\omega_n)$ to record the existence of a right-inverse, or section, $\mathbf{s}:\mathbf{d}_{n+1}\mathbf{P}_{n+1}(\omega_n)\rightarrow \mathbf{P}_{n+1}(\omega_n)$ of the map $\mathbf{d}_{n+1}$:
\smallskip
\begin{equation*}
\xymatrix@!{\mathbf{P}_{n+1}(\omega_n) \ar [r] | <<<<<<<{\;\mathbf{d}_{n+1}\;} & \mathbf{ d }_{n+1}\mathbf{P}_{n+1}(\omega_n) \ar [r] | >>>>>>>>>{\;\mathbf{i}\;} \ar@/_1pc/ [l] | {\;\mathbf{s}\;} & \mathbf{P}_{n}(\omega_n) }
\end{equation*}
Here $\mathbf{i}$ is the inclusion map. The existence of a section $\mathbf{s}$ as above follows from the fact that $\mathbf{d}_{n+1}\mathbf{P}_{n+1}(\omega_n)$ is projective, which follows in turn from Lemma \ref{summandlemma} and Theorem \ref{basis}. In fact that theorem affords us more; namely, it affords us an explicit description of such an $\mathbf{s}$: simply let
\begin{align}\label{sdefinition} s_\alpha\; :\; x\mapsto\displaystyle\sum_{j=0}^k z_j \langle\vec{\alpha}_j\rangle\end{align}
for any $x\in d_{n+1} P_{n+1}([\alpha,\omega_n))$, where
\begin{align*} x=\displaystyle\sum_{j=0}^k z_j d_{n+1}\langle\vec{\alpha}_j\rangle\end{align*} 
is the $d_{n+1}\mathcal{B}_{n+1}(\omega_n)$-decomposition of $x$.

One of our chief interests below will be functions $\mathbf{f}$ out of $\mathbf{P}_n(\omega_n)$ which extend $\mathbf{s}$, i.e., which satisfy $\mathbf{f}\big|_{\mathbf{i}[\mathbf{d}_{n+1}\mathbf{P}_{n+1}(\omega_n)]}=\mathbf{s}$. It will follow from our analysis in Section \ref{deromega1} that no such $\mathbf{f}$ can map into $\mathbf{P}_{n+1}(\omega_n)$: if $\mathbf{f}$ did, then $\mathbf{d}_{n+1}\mathbf{f}$ would define a retract $\mathbf{r}$ of the map $\mathbf{i}$, but if such an $\mathbf{r}$ exists then $\mathbf{P}_n(\omega_n)\cong\mathbf{d}_{n+1}\mathbf{P}_{n+1}(\omega_n)\bigoplus\mathbf{d}_{n}\mathbf{P}_{n}(\omega_n)$. As $\mathbf{P}_n(\omega_n)$ is free, this would imply that $\mathbf{d}_{n}\mathbf{P}_{n}(\omega_n)$ is projective (by Lemma \ref{summandlemma}), which results below will contradict. We therefore have the following diagram.
\begin{displaymath}
\xymatrix@!{\mathbf{P}_{n+1}(\omega_n) \ar [r] | <<<<<<<{\;\mathbf{d}_{n+1}\;} & \mathbf{ d }_{n+1}\mathbf{P}_{n+1}(\omega_n) \ar [r] | >>>>>>>>>{\;\mathbf{i}\;} \ar@/_1pc/ [l] | {\;\mathbf{s}\;} & \mathbf{P}_{n}(\omega_n) \ar @/_1pc/ @{.>} [l] | {\;\cancel{\mathbf{r}}\;} \ar@/_2pc/@{.>}[ll] |{\;\cancel{\mathbf{f}}\;} }
\end{displaymath}
The issue, as we will see momentarily, is not that there exist no natural extensions $\mathbf{f}$ of $\mathbf{s}:\mathbf{d}_{n+1}\mathbf{P}_{n+1}(\omega_n)\rightarrow \mathbf{P}_{n+1}(\omega_n)$, but that these extensions all output values taking infinite support. In other words, these extensions require a concomitant expansion of the target system from $\mathbf{P}_{n+1}(\omega_n)$ to $\mathbf{R}_{n+1}(\omega_n)$ (here and below, see again Sections \ref{3210} and \ref{freeandproj} for definitions).
\begin{displaymath}
\xymatrixcolsep{3.5pc}\xymatrix@R=.8pc{\mathbf{R}_{n+1}(\omega_n) \\
\mathbf{j} \ar[u] \\
\mathbf{P}_{n+1}(\omega_n) \ar@{-}[u] \ar [r] | <<<<<<<<<{\;\mathbf{d}_{n+1}\;} & \mathbf{d}_{n+1}\mathbf{P}_{n+1}(\omega_n) \ar [r] | >>>>>>>>>{\;\mathbf{i}\;} \ar@/_1pc/ [l] | {\;\mathbf{s}\;} & \mathbf{P}_{n}(\omega_n) \ar@/_1pc/@{.>}[lluu] |{\;\mathbf{f}_n\;} }
\end{displaymath}
Here $\mathbf{j}$ is the inclusion map; below we will define functions $\mathbf{f}_n\in\text{Hom}(\mathbf{P}_n(\omega_n),\mathbf{R}_{n+1}(\omega_n))$ making the above quadrilateral commute, i.e., satisfying $\mathbf{j}\,\mathbf{s}=\mathbf{f}_n\,\mathbf{i}$. As noted in Section \ref{freeandproj}, these functions may be identified with elements $\mathtt{f}_n$ of the cochain group $K^n(\mathbf{R}_{n+1}(\omega_n))$ via the simple equation $\mathtt{f}_n(\vec{\alpha})=\mathbf{f}_n(\langle\vec{\alpha}\rangle)$. These in turn determine elements $[\mathtt{f}_n]$ of the cochain group $K^n(\mathbf{R}_{n+1}(\omega_n)/\mathbf{P}_{n+1}(\omega_n))$ (this quotient should be read as encoding \emph{mod finite} relations, thereby reconnecting with the material of Section \ref{walkssection}, as we will see). Moreover, since for all $\vec{\alpha}\in [\omega_n]^{n+2}$,
\begin{align}\label{fcoherence} d^n[\mathtt{f}_n](\vec{\alpha}) & = \sum_{i=0}^{n+1}(-1)^i[\mathtt{f}_n](\vec{\alpha}^i) \\ \nonumber & = \Big[\sum_{i=0}^{n+1}(-1)^i \mathtt{f}_n(\vec{\alpha}^i)\Big] \\ \nonumber & = [\mathbf{f}_n\mathbf{d}_{n+1}(\langle\vec{\alpha}\rangle)] \\ \nonumber & = [\mathbf{s}\,\mathbf{d}_{n+1}(\langle\vec{\alpha}\rangle)] \\ \nonumber & = [0],
\end{align}
these $[\mathtt{f}_n]\in K^n(\mathbf{R}_{n+1}(\omega_n)/\mathbf{P}_{n+1}(\omega_n))$ are cocycles. If it also happens that \begin{align}\label{showingthis}[\mathtt{f}_n]\neq d^{n-1}[\mathtt{e}_{n-1}]\text{ 
for any }\mathtt{e}_{n-1}\in K^{n-1}(\mathbf{R}_{n+1}(\omega_n))\end{align} then $[\mathtt{f}_n]$ will represent a nonzero element of $$H^n(\mathcal{K}(\mathbf{R}_{n+1}(\omega_n)/\mathbf{P}_{n+1}(\omega_n))=\text{lim}^n(\mathbf{R}_{n+1}(\omega_n)/\mathbf{P}_{n+1}(\omega_n)).$$
Showing (\ref{showingthis}) is our main object in Section \ref{deromega1}. Therein, in analogy with Section \ref{walkssection}, we will often term the phenomena of equations (\ref{showingthis}) and (\ref{fcoherence}) \emph{nontriviality} and \emph{coherence} relations, respectively.

We describe now how showing $\text{lim}^n(\mathbf{R}_{n+1}(\omega_n)/\mathbf{P}_{n+1}(\omega_n))\neq 0$ for all $n$ in fact implies the remainder of Theorems \ref{mitchellstheorem}, \ref{omegan}, and \ref{eight}. We focus first on the cases of $\varepsilon=\omega_n$. As Theorem \ref{basis} implies that $\text{cd}(\omega_n)\leq n+1$ for all $n\in\omega$, we need only to show that $\text{cd}(\omega_n)\geq n+1$ for all $n\in\omega$; for this it will suffice to show for each such $n$ that $\text{lim}^{n+1}\,\mathbf{X}\neq 0$ for some $\mathbf{X}$ indexed by $\omega_n$, as argued in Section \ref{freeandproj}. For the $n=0$ case, simply observe, for example, that $$\text{lim}^1(\mathbb{Z}\stackrel{\times 2}{\longleftarrow}\mathbb{Z}\stackrel{\times 2}{\longleftarrow}\dots)=\mathbb{Z}_2/\mathbb{Z}\neq 0.$$
(See \cite[Example 3.5.5]{weibel}.) For positive $n$, $\text{cd}(\omega_n)\geq n+1$ follows immediately from $$\text{lim}^n(\mathbf{R}_{n+1}(\omega_n)/\mathbf{P}_{n+1}(\omega_n))\neq 0,$$together with the existence of the short exact sequence
\begin{align*}0\to\mathbf{P}_{n+1}(\omega_n)\to\mathbf{R}_{n+1}(\omega_n)\to\mathbf{R}_{n+1}(\omega_n)/\mathbf{P}_{n+1}(\omega_n)\to0.\end{align*}
For, letting $k=n+1$ and $\varepsilon=\omega_n$ in the long exact sequence (\ref{les}), we see that
\begin{align*}\dots\to \text{lim}^n\,\mathbf{R}_{n+1}(\omega_n)\to\text{lim}^n\,\mathbf{R}_{n+1}(\omega_n)/\mathbf{P}_{n+1}(\omega_n)\to\text{lim}^{n+1}\,\mathbf{P}_{n+1}(\omega_n)\to \text{lim}^{n+1}\,\mathbf{R}_{n+1}(\omega_n))\to\dots\end{align*}
is exact. It is not difficult to see that $\text{lim}^j\,\mathbf{R}_{n+1}(\omega_n)=0$ for all $j>0$ (direct computational verification via $\mathcal{K}(\mathbf{R}_{n+1}(\omega_n))$ is straightforward). Exactness then entails that $\text{lim}^n\,\mathbf{R}_{n+1}(\omega_n)/\mathbf{P}_{n+1}(\omega_n)\cong\text{lim}^{n+1}\,\mathbf{P}_{n+1}(\omega_n)$, hence $\mathbf{P}_{n+1}(\omega_n)$ is just such an $\mathbf{X}$ as we had desired. In fact, the nontrivial functions $\mathtt{f}_n$ we define below will correspond under this isomorphism to $\mathbf{s}\,\mathbf{d}_{n+1}$, underscoring the canonical relationship between these two functions (and their canonical relationship, in turn, with the underlying choice of $C$-sequence) and certifying the latter as very concrete witnesses to cd($\omega_n)\geq n+1$. Observe also that by way of Definitions \ref{projdefinition} and \ref{cohomologicaldimension} and Lemma \ref{projlemma}, the fact that cd($\omega_n)\geq n+1$ for all $n\in\omega$ together with Theorem \ref{basis} immediately implies Theorem \ref{eight} and, hence, Theorem \ref{omegan} as well.

Extending our results to the full statement of Theorem \ref{mitchellstheorem} is then straightforward; there are two broad cases to check:
\begin{enumerate}
\item \emph{Linear orders $\varepsilon$ of cofinality $\aleph_n$ for some finite $n$}. As the reader may verify, the argument of Theorem \ref{basis} is easily adapted to apply to linear orders $\varepsilon$, implying that cd($\varepsilon)\leq n+1$. Witnesses to $\text{lim}^{n+1}\mathbf{P}_{n+1}(\omega_n)\neq 0$ readily relativize to any ordertype-$\omega_n$ subset $\delta$ of $\varepsilon$; if $\delta$ is cofinal in $\varepsilon$ then they extend in turn to systems indexed by $\varepsilon$, implying that cd($\varepsilon)\geq n+1$. Alternately, having shown that cd$(\omega_n)=n+1$ for all $n\in\omega$, appeal to a theorem like \cite[Theorem 15.5]{SSH} will immediately extend the result to linear orders $\varepsilon$ of cofinality $\aleph_n$.
\item \emph{Linear orders $\varepsilon$ of cofinality $\kappa\geq\aleph_\omega$}. Suppose that cd$(\varepsilon)=n<\infty$ and hence that $\mathbf{d}_n\mathbf{P}_n(\varepsilon)$ is projective. We may then construct a direct summand $\mathbf{d}_n\mathbf{P}_n(X)$ of $\mathbf{d}_n\mathbf{P}_n(\varepsilon)$ with cf$(X)=\aleph_n$; by Lemma \ref{summandlemma}, $\mathbf{d}_n\mathbf{P}_n(X)$ is then projective, i.e., cd$(X)\leq n$, contradicting point (1) above. (This is a mild abuse: under our conventions, $\mathbf{d}_n\mathbf{P}_n(X)$ isn't an object of $\mathsf{Ab}^{\varepsilon^{\text{op}}}$, but it naturally identifies with one.) It suffices in fact to take any $X\subseteq \varepsilon$ of cofinality $\aleph_n$ such that $\mathbf{P}_n(X)$ is closed with respect to the function $\mathbf{s}\,\mathbf{d}_n$, where $\mathbf{s}$ is the section $\mathbf{d}_n\mathbf{P}_n(\varepsilon)\to\mathbf{P}_n(\varepsilon)$ furnished by our assumption. For then, letting $\mathbf{p}:\mathbf{P}_n(\varepsilon)\to\mathbf{P}_n(X)$ denote the natural projection, $\mathbf{d}_n\,\mathbf{p}\,\mathbf{s}$ defines a retract of the natural inclusion $\mathbf{j}:\mathbf{d}_n\mathbf{P}_n(X)\to\mathbf{d}_n\mathbf{P}_n(\varepsilon)$, implying that $\mathbf{d}_n\mathbf{P}_n(X)$ is a direct summand of $\mathbf{d}_n\mathbf{P}_n(\varepsilon)$, as desired.

This argument resembles nearly enough the first part of the induction step of Proposition \ref{notproject} below that the reader is referred there for further details and a diagram.
\end{enumerate}

\subsection{The functions $\mathtt{f}_n$}\label{section52}
We now describe the functions $\mathtt{f}_n$ that will form the focus of the remainder of the paper, beginning with the case of $n=0$.
\begin{exa}{\textbf{The case of $\omega\,$:}} \label{thecaseofomega} Here the $C$-sequence $\{C_{i+1}=\{i\}\,|\,i<\omega\}$ determines the basis $\mathcal{B}_1(\omega)=\{\langle i,i+1\rangle\,|\, i<\omega\}$; hence 
$$\mathbf{s}\,\mathbf{d}_1(\langle j,k \rangle)=\displaystyle\sum_{i=j}^{k-1}\,\langle i,i+1\rangle$$
for any $j<k<\omega$. 
 An $\mathbf{f}:\mathbf{P}_0(\omega)\rightarrow\mathbf{R}_1(\omega)$ for which $\mathbf{f}\big|_{\mathbf{d}_1\mathbf{P}_1(\omega)}\,=\mathbf{s}$ would satisfy
$$\mathbf{f}(\langle j+1\rangle-\langle j \rangle)=\mathbf{s}(\langle j+1\rangle-\langle j \rangle)=\langle j, j+1\rangle$$
for any $j<\omega$. This, though, amounts to a definition: the implicit formula
$$\mathbf{f}(\langle j\rangle)=-\langle j, j+1\rangle+\mathbf{f}(\langle j+1\rangle)$$
in fact fully determines $\mathbf{f}$. This is because $\mathbf{f}=\{f_j:P_0([j,\omega))\rightarrow R_1([j,\omega))\,|\,j<\omega\}$, and $\mathbf{f}(\langle j\rangle)$ by definition is $f_j(\langle j\rangle)$. This value, falling in $R_1([j,\omega))=\prod_{[[j,\omega)]^2}\mathbb{Z}$, can involve no coordinates less than $j$. Hence the formula
\begin{align}\label{stupid}\mathbf{f}(\langle 0\rangle)=-\langle 0,1\rangle+\mathbf{f}(\langle 1\rangle)\end{align} entirely determines the ``$0$-column'' of $\mathbf{f}(\langle 0\rangle)$. Similarly, the formula
$$\mathbf{f}(\langle 1\rangle)=-\langle 1,2\rangle+\mathbf{f}(\langle 2\rangle)$$ entirely determines the ``$1$-column'' of $\mathbf{f}(\langle 1\rangle)$ and hence, by (\ref{stupid}), that of $\mathbf{f}(\langle 0\rangle)$ as well --- and so on. This defines $\mathbf{f}$ on $\{\langle j\rangle\,|\,j\in\omega\}$ and therefore on all of $\mathbf{P}_0(\omega)$; as described in the previous subsection, it determines a function $\mathtt{f}_0\in  K^0(\mathbf{R}_1(\omega))$ as well. It is not difficult to see, in fact, that $\mathtt{f}_0$ represents a nonzero element of $\text{lim} (\mathbf{R}_1(\omega)/\mathbf{P}_1(\omega))$.\footnote{It is gratuitous but tempting and possibly clarifying to write the relationship of these functions  as follows: $\mathtt{f}_0(k)-\mathtt{f}_0(j)=\int_j^k\mathbf{s}\,\mathbf{d}_1$ for all $j<k<\omega$.}
\end{exa}

This technique very generally applies. Recall from Definition \ref{thefunctionb} the function $\mathtt{b}$, which via the addition of a single coordinate into $\vec{\alpha}$, if possible,  converts $\vec{\alpha}$ to some ``nearest'' $\mathtt{b}(\vec{\alpha})\in\mathcal{B}$.\footnote{Since $\mathtt{b}(\vec{\alpha})$ is a generator, not an $n$-tuple, the notation $\mathtt{b}(\vec{\alpha})^i$ hereabouts is a minor abuse; still, its meaning should be clear.} Recall also that $\mathbf{s}$ was defined so that $\mathbf{s}\,\mathbf{d}_{n+1}\big|_{\mathcal{B}_{n+1}}=\mathbf{id}\big|_{\mathcal{B}_{n+1}}$ (see equation \ref{sdefinition}). Therefore if $\mathbf{f}_n\,:\,\mathbf{P}_n(\omega_n)\rightarrow\mathbf{R}_{n+1}(\omega_n)$ extends $\mathbf{s}$, then
\begin{align}\label{eff}\displaystyle\sum_{i=0}^{n+1}\,(\text{-}1)^i \mathbf{f}_n (\mathtt{b}(\vec{\alpha})^i)=\mathbf{f}_n(\mathbf{d}_{n+1}(\mathtt{b}(\vec{\alpha})))=\mathbf{s}(\mathbf{d}_{n+1}(\mathtt{b}(\vec{\alpha})))=\mathtt{b}(\vec{\alpha})\end{align}
Recall that $d_{n+1}\mathcal{B}_{n+1}$ defines a basis for $\mathbf{d}_{n+1}\mathbf{P}_{n+1}(\omega_n)$, and that every element of $\mathcal{B}_{n+1}$ is of the form $\mathtt{b}(\vec{\alpha})$ for some $\vec{\alpha}$. It follows that equation \ref{eff} is both a sufficient and a necessary condition for $\mathbf{f}_n$ to extend $\mathbf{s}$, and may even, as in Example \ref{thecaseofomega}, be read as \textit{defining} such an $\mathbf{f}_n$: let $\vec{\alpha}=(\vec{\beta},\vec{\gamma})\in[\omega_n]^{n+1}$, with $|\vec{\beta}|=j+1$ and $\vec{\gamma}$ the maximal proper internal tail of $\vec{\alpha}$, so that $\mathtt{b}(\vec{\alpha})$ is either $0$ or $\langle\vec{\beta},C^{\vec{\gamma}}(\beta_j),\vec{\gamma}\rangle$. 
In the former case, let
\begin{align}\label{efff}
\mathbf{f}_n(\langle\vec{\alpha}\rangle)=0\end{align}
In the latter case, $\langle\vec{\alpha}\rangle=\mathtt{b}(\vec{\alpha})^{j+1}$, hence equation \ref{eff} entails that
\begin{align}\label{effff}
\mathbf{f}_n(\langle\vec{\alpha}\rangle)=(\text{-}1)^{j+1}\Big[\mathtt{b}(\vec{\alpha})-\displaystyle\sum_{i=0}^{j}\,(\text{-}1)^i \mathbf{f}_n(\mathtt{b}(\vec{\alpha})^i)-\displaystyle\sum_{i=j+2}^{n+1}\,(\text{-}1)^i \mathbf{f}_n(\mathtt{b}(\vec{\alpha})^i)\Big]
\end{align}
Unlike in Example \ref{thecaseofomega}, equations \ref{efff} and \ref{effff} alone do not fully determine $\mathbf{f}_n$. 
  However, these equations \textit{do} share with that of Example \ref{thecaseofomega} a canonical solution, namely the function associating to $\langle\vec{\alpha}\rangle$ just those generators $\mathtt{b}(\,\cdot\,)\in\mathcal{B}_{n+1}(\omega_n)$ appearing in the full formal expansion of equation \ref{effff}. More precisely, we identify the function $\mathtt{f}_n(\vec{\alpha})$ with the pointwise limit of the generator-sums appearing in the possibly infinitely many steps of the recursive expansion of equation \ref{effff}. Below, we duly argue that this operation is meaningful, but first-time readers might proceed directly to Section \ref{deromega1}; therein, its nature should rapidly grow intuitively clear.

Technically speaking, $\mathtt{f}_n$ is an element of $K^n(\mathbf{R}_{n+1}(\omega_n))$, hence $\mathtt{f}_n(\vec{\alpha})\in R_{n+1}([\alpha_0,\omega_n))$ for each $\vec{\alpha}\in [\omega_n]^{n+1}$. It is often simpler, though, to regard $\mathtt{f}_n$ as a function $[\omega_n]^{n+1}\to R_{n+1}([0,\omega_n))=\prod_{[\omega_n]^{n+2}}\mathbb{Z}$ via tacit applications of the inclusion maps $p_{0,\alpha_0}:R_{n+1}([\alpha_0,\omega_n))\to R_{n+1}([0,\omega_n))$, and we will tend to do so below. Observe that statements like ``$\mathtt{f}_n(\vec{\alpha})\in R_{n+1}([\alpha_0,\omega_n))$'' convert under this convention to statements about the support of $\mathtt{f}_n(\vec{\alpha})$. The legitimacy of the previous paragraph's definition of $\mathtt{f}_n$ derives from the following lemmas, which collect some useful general information about these functions along the way. By \emph{a branch of the formal expansion of equation \ref{effff}} we mean a sequence of the form
\begin{align}\label{bees}\vec{\alpha}\rightarrow \mathtt{b}(\vec{\alpha})\rightarrow \mathtt{b}(\vec{\alpha})^{k_1}\rightarrow \mathtt{b}(\mathtt{b}(\vec{\alpha})^{k_1})\rightarrow \mathtt{b}(\mathtt{b}(\vec{\alpha})^{k_1})^{k_2}\rightarrow\dots \end{align}
where each $k_i\leq n+1$ is other than the index of the coordinate added by the $i^{th}$ application of $\mathtt{b}$.
\begin{lem} Any generator $\mathtt{b}(\vec{\alpha})\in\prod_{[\omega_n]^{n+2}}\mathbb{Z}$ appears at most once in any branch of the (possibly infinite) formal expansion of equation \ref{effff}.
\end{lem}\label{nocircularity}
\begin{proof} We will show slightly more, namely that (\ref{bees}) is nonrepeating. Begin much as before, by letting $\vec{\alpha}=(\vec{\beta}, \vec{\gamma})$ with $\vec{\beta}$ and $\vec{\gamma}$ of lengths $\ell+1$ and $m+1$ respectively and $\vec{\gamma}$ the maximum proper internal tail of $\vec{\alpha}$, so that $\mathtt{b}(\vec{\alpha})=\langle\vec{\beta},C^{\vec{\gamma}}(\beta_\ell),\vec{\gamma}\rangle$. (If $\mathtt{b}(\vec{\alpha})=0$, of course, we are done.)

\noindent \underline{Case 1}: $\mathtt{b}(\vec{\alpha})^{k_1}=(\vec{\beta},C^{\vec{\gamma}}(\beta_\ell),\vec{\gamma}^j)$ for some $j$.

Observe first that if $j=m$ then the coordinate $\gamma_m$ will never reappear in the sequence (\ref{bees}); in consequence, neither will $\vec{\alpha}$ nor $\mathtt{b}(\vec{\alpha})$. Now suppose $j<m$ and let $\gamma_{-1}=C^{\vec{\gamma}}(\beta_\ell)$ and observe that as $(\gamma_{j-1},\gamma_{j+1},\dots,\gamma_m)$ is internal, so long as this tuple remains a tail of the entries in the sequence (\ref{bees}), no application of $\mathtt{b}$ can recover the coordinate $\gamma_j$. Observe also that if only $\gamma_{j-1}$ is ever removed from this tuple then subsequent applications of $\mathtt{b}$ will only ever introduce coordinates $\xi\leq\gamma_{j-1}<\gamma_j$ to the sequence (\ref{bees}). Hence only if some $k_i$ removes a coordinate $\gamma_{j'}>\gamma_j$ may the coordinate $\gamma_j$, and hence $\vec{\alpha}$ or $\mathtt{b}(\vec{\alpha})$, possibly reappear in the sequence (\ref{bees}) --- but this argument then applies to $\gamma_{j'}$, and so on, and can only end with the coordinate $\gamma_m$. As we noted at the outset of this case, though, $\gamma_m$, once lost, is irrecoverable.

\noindent \underline{Case 2}: $\mathtt{b}(\vec{\alpha})^{k_1}=(\vec{\beta}^j,C^{\vec{\gamma}}(\beta_\ell),\vec{\gamma})$ for some $j$.

Observe that either $\gamma_0$ or $C^{\vec{\gamma}}(\beta_\ell)$ is a successor. If $\gamma_0$ is a successor and $j\leq \ell$ then $\mathtt{b}(\mathtt{b}(\vec{\alpha})^{k_1})=0$. If $C^{\vec{\gamma}}(\beta_\ell)$ is a successor and $j<\ell$ then it is again clear that $\mathtt{b}(\mathtt{b}(\vec{\alpha})^{k_1})=0$. This leaves only the case in which $\gamma_0$ is a limit and $j=\ell$. In this case if $\mathtt{b}(\mathtt{b}(\vec{\alpha})^{k_1})\neq 0$ then it equals $\langle\vec{\beta}^\ell,\eta,C^{\vec{\gamma}}(\beta_\ell),\vec{\gamma}\rangle$ with $\eta<\beta_\ell$. Now $(\eta,C^{\vec{\gamma}}(\beta_\ell),\vec{\gamma})$ is internal and we may argue just as in Case 1 that the coordinate $\beta_\ell$ can only reappear at the expense of later coordinates, so that again, neither $\vec{\alpha}$ nor $\mathtt{b}(\vec{\alpha})$ can recur in the sequence (\ref{bees}).

This concludes the proof.
\end{proof}

The above reasoning carries yet stronger implications, namely:
\begin{lem}\label{terminatesorstabilizeslemma}
A branch in the formal expansion of equation \ref{effff} either
\begin{enumerate}
\item terminates, in the sense that $\mathtt{b}(\vec{\beta})=0$ for some $\vec{\beta}$ in its sequence, or
\item stabilizes, in the sense that there exists an $m\in\omega$ such that $k_j=0$ for all $j\geq m$.
\end{enumerate}
\end{lem}
To see this, observe that if item (1) of the lemma fails, then there must exist an $m(n+1)$ such that $j>m(n+1)$ implies $k_j<n+1$: if there were not, then the last coordinates of the terms of (\ref{bees}) would contain an infinite descending sequence of ordinals. One may then similarly deduce that $k_j<n$ for all $j$ above some $m(n)\geq m(n+1)$, and so on, down to $n=1$; let $m=m(1)$ (we are light on the details here because they so resemble those of the arguments of Lemmas \ref{nocircularity} and \ref{lastin5}; see particularly the close of the proof of Lemma \ref{lastin5}). Note that as the argument of $m$ descends or ``moves to the left,'' so too does a stable internal tail within the elements of (\ref{bees}), so that we can be quite concrete about the eventual form of Lemma \ref{terminatesorstabilizeslemma}'s case (2); ultimately, it takes the shape of
\begin{align}\label{morebees}\cdots\rightarrow(\eta_i,\vec{\beta})\rightarrow \mathtt{b}((\eta_i,\vec{\beta}))=\langle\eta_i,\eta_{i+1},\vec{\beta}\rangle\rightarrow (\eta_{i+1},\vec{\beta})\rightarrow \mathtt{b}((\eta_{i+1},\vec{\beta}))=\langle\eta_{i+1},\eta_{i+2},\vec{\beta}\rangle\rightarrow\cdots \end{align}
wherein $\vec{\beta}$ is the maximal internal tail of each $(\eta_j,\vec{\beta})$. Since $\mathrm{cf}(\beta_{n-1})<\aleph_n$, the cofinality of $\beta_0$ is $\aleph_0$ (if it were less, (\ref{morebees}) would terminate), and it is the supremum of the sequence $(\eta_j)_{j\in\omega}\subseteq C_{\vec{\beta}}$.

We will apply the following lemma to conclude that the function $\mathtt{f}_n$ is well-defined.

\begin{lem}\label{lastin5}
If $\mathtt{b}(\vec{\xi})=\langle\vec{\gamma}\rangle$ descends from $\vec{\alpha}$ via a path of type (\ref{bees}) then $\gamma_0<\alpha_1$.
\end{lem}
We apply the lemma as follows: suppose for contradiction that the formal expansion of some $\mathtt{f}_n(\vec{\alpha})$ outputs the generator $\langle\vec{\gamma}\rangle$ infinitely often. Then this expansion must contain an infinite branch of type (\ref{nocircularity}) whose elements all count $\langle\vec{\gamma}\rangle$ among their descendents. By Lemma \ref{terminatesorstabilizeslemma}, this branch eventually assumes the form (\ref{morebees}); by Lemma \ref{lastin5} then, $\gamma_0<\beta_0$. This, though, is a contradiction, since once $\eta_j>\gamma_0$ there is no longer any path of type (\ref{bees}) from $(\eta_j,\vec{\beta})$ to $\langle\vec{\gamma}\rangle$.
\begin{proof}[Proof of Lemma \ref{lastin5}] Begin with $\vec{\alpha}=(\vec{\upsilon},\vec{\varepsilon})$, where $\vec{\varepsilon}$ is the maximal internal tail of $\vec{\alpha}$ and the length of $\vec{\upsilon}$ is $\ell+1$, so that $\mathtt{b}(\vec{\alpha})$, if nonzero, is $\langle\vec{\upsilon}, C^{\vec{\varepsilon}}(\upsilon_\ell),\vec{\varepsilon}\rangle$. A key point is that by the definition of $C^{\vec{\varepsilon}}(\upsilon_\ell)$, the branch (\ref{bees}) can continue --- meaning $\mathtt{b}(\mathtt{b}(\vec{\alpha})^k)\neq 0$ --- only if $k$ removes either an element of $\vec{\varepsilon}$, or $\upsilon_\ell$. In the former case, the initial segment $\vec{\upsilon}$ of $\vec{\alpha}$ is unaffected; in particular, all elements of the sequence $\vec{\alpha}\to\mathtt{b}(\vec{\alpha})\to\mathtt{b}(\vec{\alpha})^k$ satisfy the conclusion of the lemma. Hence it's only via steps of the latter sort that that conclusion may conceivably fail; suppose therefore that $k$ removes $\upsilon_\ell$. There are then two possibilities, depending on the length of $\vec{\upsilon}$. If $\ell>0$ then, letting $\vec{\eta}=(C^{\vec{\varepsilon}}(\upsilon_\ell),\vec{\varepsilon})$, we have that $\mathtt{b}(\mathtt{b}(\vec{\alpha})^k)$, if nonzero, equals $\langle\vec{\upsilon}^\ell,C^{\vec{\eta}}(\upsilon_{\ell-1}),C^{\vec{\varepsilon}}(\upsilon_\ell),\vec{\varepsilon}\rangle$; moreover, as both $C$-terms are successors, it is only if $k'$ removes either $C^{\vec{\varepsilon}}(\upsilon_\ell)$ or an element of $\vec{\varepsilon}$ that $\mathtt{b}(\mathtt{b}(\mathtt{b}(\vec{\alpha})^k)^{k'})$ may fail to equal zero. Here again all elements of the sequence $\vec{\alpha}\to\mathtt{b}(\vec{\alpha})\to\mathtt{b}(\vec{\alpha})^k\to\mathtt{b}(\mathtt{b}(\vec{\alpha})^k)\to \mathtt{b}(\mathtt{b}(\vec{\alpha})^k)^{k'}$ satisfy the conclusion of the lemma; note furthermore that each coordinate of $\mathtt{b}(\mathtt{b}(\vec{\alpha})^k)^{k'}$ is less than or equal to the corresponding coordinate of $\vec{\alpha}$ (with at least one coordinate strictly less). This leaves only the case of $\ell=0$. Here again though, clearly, the conclusion of the lemma holds throughout the sequence $\vec{\alpha}\to\mathtt{b}(\vec{\alpha})\to\mathtt{b}(\vec{\alpha})^k$. As the analysis we've just described will reapply to the last term of each of the  sequences we've considered, this concludes the argument.

In its course, we showed that in the ``$\ell>0$ case'' corresponding coordinates descend, some strictly, in the passage from $\vec{\alpha}$ to $\mathtt{b}(\mathtt{b}(\vec{\alpha})^k)^{k'}$. This also clearly holds of the passage from $\vec{\alpha}$ to $\mathtt{b}(\vec{\alpha})^k$ in what was termed ``the former case'' above, and these two recognitions together suffice for the argument of Lemma \ref{terminatesorstabilizeslemma}.
\end{proof} 

Note in conclusion that we might read Lemma \ref{terminatesorstabilizeslemma} as pointing to something ``essentially finitary'' about the function $\mathtt{f}(\vec{\alpha})$, in the sense that its expansion has no truly \emph{interesting} infinite branches, and that we might in turn read the most fundamental implication of the above arguments --- namely, that the coefficient of any generator $\langle\vec{\gamma}\rangle$ in $\mathtt{f}_n(\vec{\alpha})$ may be computed in finitely many steps --- as among its effects. This is a perspective that the higher walks of this paper's Section \ref{highertraces} may be regarded as formalizing.\\

To recapitulate: the primary task of the following section is to show that the functions $\mathtt{f}_n$ are \emph{nontrivial} in the sense of formula \ref{showingthis} above. This fact together with Theorem \ref{basis} will then immediately imply Theorems \ref{mitchellstheorem},  \ref{omegan}, and  \ref{eight}, in just the fashion described in Subsection \ref{theargument}. In the process, the walks material of Sections \ref{walkssection} and \ref{internalsection} will begin to reappear, along with its higher-order analogues.
\section{The coherence and nontriviality of the functions $\mathtt{f}_n$}\label{deromega1}

In what follows, the letters $x$, $y$, and $z$ will correspond to the first, second, and third coordinate-places in ordered triples; more generally, $z$ will denote the last coordinate-position in any ordered $n$-tuple below, with prior coordinate-positions then labeled in descending alphabetic order. We write $=^{*}$ to denote equality modulo finite differences. As discussed above, although $\mathtt{f}_n$ is an element of $K^n(\mathbf{R}_{n+1}(\omega_n))$, we nevertheless regard each $\mathtt{f}_n(\vec{\alpha})$ as a function $[\omega_n]^{n+2}\to \mathbb{Z}$; similarly for each $\mathtt{e}_{n-1}(\vec{\alpha})$. This approach entails minor abuses, but appears to be the simplest. If any function in the equations below is restricted, then the comparison $=^*$ should be read as taking place on that restriction, but in this section it will be equally valid, and sometimes more telling, to read an expression like $\mathtt{f}\big|_A$ as the function $[\omega_n]^{n+2}\to \mathbb{Z}$ coinciding with $\mathtt{f}$ on $A$ and outputting zero elsewhere, and to read $=^*$, in conjunction, as applying over all of $[\omega_n]^{n+2}$.

\subsection{The case of $n=1$}\label{thecasen1}

Here the function $\mathtt{f}_n$ of the previous section specializes to a function $\mathtt{f}_1\in K^1(\mathbf{R}_2(\omega_1))$ with the property that
\begin{align}\label{41}\mathtt{f}_1(\beta,\gamma)-\mathtt{f}_1(\alpha,\gamma)+\mathtt{f}_1(\alpha,\beta)=^*0\textnormal{ for all }\alpha<\beta<\gamma<\omega_1
\end{align}
In fact this difference from zero is precisely $\mathbf{s}\,\mathbf{d}_2(\langle\alpha,\beta,\gamma\rangle)$. The \textit{coherence}, in other words, of the system $\mathtt{f}_1$ amounts simply to the fact that $\mathbf{s}\,\mathbf{d}_2$-images have finite supports. What remains to be shown is its \textit{nontriviality}, namely, the fact that no $\mathtt{e}_0\in K^0(\mathbf{R}_2(\omega_1))$ satisfies the following property:
\begin{align}\label{42}
\mathtt{e}_0(\beta)-\mathtt{e}_0(\alpha)=^*\mathtt{f}_1(\alpha,\beta)\text{ for all }\alpha<\beta<\omega_1
\end{align}
As noted, this will establish the case $n=1$ of Mitchell's theorem, and its argument will furnish the template for the cases of higher $n$. As noted as well, statements like ``$\mathtt{f}_1\in K^1(\mathbf{R}_2(\omega_1))$'' describe the supports of $\mathtt{f}_1$ (i.e., $\mathtt{f}_1(\alpha,\beta)$ may be identified with an element of $R_2([\alpha,\omega_1))$ for all $\alpha<\beta<\omega_1$), but we can be much more precise: when $n=1$, the definition of $\mathtt{f}_n$ via equations \ref{efff} and \ref{effff} assumes a particularly straightforward form:
\begin{align}\label{43}
  \mathtt{f}_1(\alpha,\beta) = \left\{\def\arraystretch{1}%
  \begin{array}{@{}c@{\quad}l@{}}
    0 & \hspace{.3 cm}\textnormal{if }\beta=\alpha+1\\
    -\langle\alpha,C^\beta(\alpha),\beta\rangle+\mathtt{f}_1(C^\beta(\alpha),\beta)+\mathtt{f}_1(\alpha,C^\beta(\alpha)) & \hspace{.3 cm}\mathrm{otherwise}\\
  \end{array}\right.
\end{align}
It follows immediately that \begin{align}\label{44}\text{supp}(\mathtt{f}_1(\alpha,\beta))\subseteq [\,[\alpha,\beta]\,]^3\end{align}
This facilitates sufficiently ``spatial'' readings that we introduce the following notation: for $A,B\subseteq [\xi]^{<\omega}$, let $A\otimes B$ denote the collection of tuples $(\vec{\alpha},\vec{\beta})\in A\times B$ for which $\vec{\alpha}<\vec{\beta}$. Extensions of this notation should be self-explanatory. For example, it follows from equations \ref{41} and \ref{44} that
\begin{align}\label{45}\mathtt{f}_1(\alpha,\beta)-\mathtt{f}_1(\alpha,\gamma)\big|_{[\alpha,\beta)\otimes[ \omega_1]^2}\,=^*0\textnormal{ for all }\alpha<\beta<\gamma<\omega_1
\end{align}
and, hence, that
\begin{align}\label{46}\mathtt{f}_1(\alpha,\gamma)\big|_{[\alpha,\beta)\otimes[(\beta,\omega_1)]^2}\,=^*0\textnormal{ for all }\alpha<\beta<\gamma<\omega_1
\end{align}
It follows also from equation \ref{44} that for any ``trivializing'' $\mathtt{e}_0$ as in (\ref{42}),
\begin{align}\label{47}
\mathtt{e}_0(\beta)=^*\mathtt{e}_0(\alpha)\big|_{[(\beta,\omega_1)]^3}\textnormal{ for all }\alpha<\beta<\omega_1
\end{align}
hence the data of such an $\mathtt{e}_0$ is entirely present (mod finite) in $\mathtt{e}_0(0)$. In other words, there exists an $\mathtt{e}_0$ as in (\ref{42}) if and only if for some $\mathtt{e}_0(0)\in\prod_{[\omega_1]^3}\mathbb{Z}$ 
\begin{align}\label{48}
\mathtt{e}_0(0)\big|_{\beta\otimes[\omega_1]^2}\,=^*\mathtt{f}_1(0,\beta)\textnormal{ for all }\beta<\omega_1
\end{align}
(For the ``if'' direction, let $\mathtt{e}_0(\beta)=\mathtt{e}_0(0)\big|_{[\beta,\omega_1)\otimes[\omega_1]^2}$.) 
 We will derive a contradiction from the existence of such an $\mathtt{e}_0(0)$; this will conclude the $n=1$ step of our proof of Theorems \ref{mitchellstheorem},  \ref{omegan}, and \ref{eight}.

The reader is now referred to Figure \ref{walksclubs}. The $x$, $y$, and $z$ axes therein are each of ``length'' $\omega_1$ (the $y$ axis drifts back, into the page).  The spaces of the supports of $\mathtt{f}_1(0,\alpha)$, $\mathtt{f}_1(0,\beta)$, and $\mathtt{f}_1(\alpha,\beta)$ are plotted along these axes as the tetrahedra $[\,[0,\alpha]\,]^3$, $[\,[0,\beta]\,]^3$, and $[\,[\alpha,\beta]\,]^3$, respectively. Depicted as well are the outputs of each of these three functions $\mathtt{f}_1(\eta,\xi)$ within the distinguished planes $x=\eta$ and $z=\xi$ (shaded in Figure \ref{walksclubs}); by equation \ref{43}, these have the following general forms:
\begin{align}
\mathtt{f}_1(\alpha,\beta)\big|_{\{\alpha\}\otimes[\omega_1]^2}\,= & \,-\langle\alpha,C^\beta(\alpha),\beta\rangle-\langle\alpha,C^{C^\beta(\alpha)}(\alpha),C^\beta(\alpha)\rangle-\dots\label{4105} \\
\mathtt{f}_1(\alpha,\beta)\big|_{[\omega_1]^2\otimes\{\beta\}}\,= & \,-\langle\alpha,C^\beta(\alpha),\beta\rangle-\langle C^\beta(\alpha),C^\beta(C^\beta(\alpha)),\beta\rangle-\dots\label{4110}
\end{align} 
Line \ref{4105}, restricted to either the $2^{\mathrm{nd}}$ or $3^{\mathrm{rd}}$ coordinate, bears copies (minus the first or last element, respectively) of the walk from $\beta$ down to $\alpha+1$. 
Line \ref{4110}, similarly, is an image of the club $C_\beta$ above $\alpha$ (here our convention that $\mathrm{otp}(C_\beta)\leq\omega$ for all countable $\beta$ is essential). For limit $\beta$, of course, these clubs $C_\beta$ are infinite; by equation \ref{45}, $\mathtt{f}_1(\alpha,\gamma)$ must contain all but finitely much of each of these $C_\beta$-images, where $\beta$ ranges through $(\alpha,\gamma)\cap\text{Lim}$. This is a requirement in some tension with equation \ref{46}, a tension manifesting as the nontriviality of the system $\mathtt{f}_1$. 

\begin{figure}
\centering
\begin{tikzpicture}[MyPersp,scale=.92]
	\coordinate (A) at (0,0,0);
	\coordinate (B) at (0,0,4);
	\coordinate (C) at (0,4,4);
	\coordinate (D) at (4,4,4);
	\coordinate (E) at (0,0,6);
	\coordinate (F) at (0,6,6);
	\coordinate (G) at (6,6,6);
	\coordinate (H) at (4,4,6);
	\coordinate (I) at (4,6,6);
	\coordinate (J) at (0,4,6);
	\coordinate (K) at (0,1.8,4);
	\coordinate (L) at (0,1.1,2.2);
	\coordinate (M) at (0,.8,1.6);
	\coordinate (N) at (0,.4,.8);
	\coordinate (O) at (0,.2,.4);
	\coordinate (P) at (0,2.8,6);
	\coordinate (Q) at (0,1.9,4.9);
	\coordinate (R) at (0,8,0);
	\coordinate (S) at (0,0,8);
	\coordinate (T) at (9,0,0);
	\coordinate (U) at (0,4,0);
	\coordinate (V) at (0,6,0);
	\coordinate (W) at (4,0,0);
	\coordinate (X) at (6,0,0);
	\coordinate (K2) at (.6,2.15,4);
	\coordinate (K3) at (1.2,2.5,4);
	\coordinate (K4) at (1.8,2.85,4);
	\coordinate (K5) at (2.4,3.2,4);
	\coordinate (K6) at (2.8,3.4,4);
	\coordinate (K7) at (3.2,3.629,4);
	\coordinate (P2) at (1.5,3.53,6);
	\coordinate (P3) at (3,4.5,6);
	\coordinate (P4) at (4.41,5.3,6);
	\coordinate (P5) at (4.75,5.55,6);
	\coordinate (P6) at (4.92,5.65,6);
	\coordinate (P7) at (5.18,5.83,6);
	\coordinate (Y) at (4,4.6,6);
	\coordinate (Y2) at (4,4.4,5.2);
	\coordinate (Y3) at (4,4.27,4.9);
	\coordinate (Y4) at (4,4.15,4.6);
	\coordinate (B-) at (-.13,0,4);
	\coordinate (B+) at (.13,0,4);
	\coordinate (E-) at (-.13,0,6);
	\coordinate (E+) at (.13,0,6);	
	\coordinate (C-) at (0,4,-.13);
	\coordinate (C+) at (0,4,.13);
	\coordinate (F-) at (0,6,-.13);
	\coordinate (F+) at (0,6,.13);
	\coordinate (W-) at (4,0,-.13);
	\coordinate (W+) at (4,0,.13);
	\coordinate (X-) at (6,0,-.13);
	\coordinate (X+) at (6,0,.13);
	
	\draw (U) node[inner sep=6pt, below] {$\alpha$};
	\draw (V) node[inner sep=6pt, below] {$\beta$};
	\draw (W) node[inner sep=6pt, below] {$\alpha$};
	\draw (X) node[inner sep=6pt, below] {$\beta$};
	\draw (A)--(B)--(C)--cycle;
	\draw (A)--(B)--(D)--cycle;
	\draw (A)--(C)--(D)--cycle;
	\draw[thick] (B-)--(B+);
	\draw[thick] (E-)--(E+);
	\draw[thick] (C-)--(C+);
	\draw[thick] (F-)--(F+);
	\draw[thick] (W-)--(W+);
	\draw[thick] (X-)--(X+);
	\draw (B) node[inner sep=6pt, left] {$\alpha$}--(C)--(D)--cycle;
	\fill[fill=gray, opacity=.09] (B)--(C)--(D);
	\fill[fill=gray, opacity=.06] (A)--(E)--(F);
	\fill[fill=gray, opacity=.09] (E)--(F)--(G);
	\fill[fill=gray, opacity=.06] (D)--(H)--(I);
	\draw (A)--(E)--(F)--cycle;
	\draw (A)--(F)--(G)--cycle;
	\draw (A)--(E)--(G)--cycle;
	\draw (D)--(H)--(I)--cycle;
	\draw (E) node[inner sep=6pt, left] {$\beta$}--(F)--(G)--cycle;
	\draw (D)--(H)--(G)--cycle;
	\draw (H)--(I)--(G)--cycle;
	\draw (D)--(H);
	\draw[->] (A)--(R) node[below right] {y};
	\draw[->] (A)--(S) node[left] {z};
	\draw[->] (A)--(T) node[below] {x};
	\draw[dashed] (H)--(J)--(C);
	\draw (P) node[circle, very thick,draw,fill=white,fill opacity=.8]{};
	\draw[densely dotted, thick, out=225, in=90] (P) to (Q);
	\draw (Q) node[circle, very thick,draw,fill=white,fill opacity=.8]{};
	\draw (M) node[circle, very thick,draw,fill=white,fill opacity=.8]{};
\draw (N) node[circle, very thick,draw,fill=white,fill opacity=.8]{};	
	\draw (O) node[circle, very thick,draw,fill=white,fill opacity=.8]{};
	\draw[densely dotted, thick, out=230, in=90] (K) node[fill,circle, gray,opacity=.8]{} to (L);
	\draw[densely dotted, thick, out=210, in=90] (L) node[fill,circle, gray,opacity=.8]{} to (M);
	\draw[densely dotted, thick, out=230, in=90] (M) node[fill,circle, gray,opacity=.8]{} to (N);
	\draw[densely dotted, thick, out=215, in=90] (N) node[fill,circle, gray,opacity=.8]{} to (O) node[fill,circle, gray,opacity=.8]{};
	\draw[densely dotted, thick, out=235, in=90] (Q) to (M);
	\draw (K2) node[fill,circle, gray,opacity=.8]{};
     \draw (K3) node[fill,circle, gray,opacity=.8]{};
     \draw (K4) node[fill,circle, gray,opacity=.8]{};
     \draw (K5) node[fill,circle, gray,opacity=.8]{};
     \draw[dashdotted] (K)  to (K6);
     \draw[dashdotted] (P)  to (P6);
     \draw[dashdotted] (Y)  to (P6);
     \draw[->][thick, loosely dotted] (K6) to (K7);
     \draw[->][thick, loosely dotted] (P6) to (P7);
     \draw (P2) node[circle, very thick,draw,fill=white,fill opacity=.8]{};
     \draw (P3) node[circle, very thick,draw,fill=white,fill opacity=.8]{};
     \draw (P4) node[circle, very thick,draw,fill=white,fill opacity=.8]{};
     \draw (P5) node[circle, very thick,draw,fill=white,fill opacity=.8]{};
     \draw (Y) node[circle,draw,fill=black,fill opacity=.8,scale=.55]{};
     \draw (Y2) node[circle,draw,fill=black,fill opacity=.8,scale=.55]{};
     \draw (Y3) node[circle,draw,fill=black,fill opacity=.8,scale=.55]{};
     \draw (Y4) node[circle,draw,fill=black,fill opacity=.8,scale=.55]{};
          \draw (P4) node[circle,draw,fill=black,fill opacity=.8,scale=.55]{};
     \draw (P5) node[circle,draw,fill=black,fill opacity=.8,scale=.55]{};
     \draw[densely dotted, thick, out=215, in=90] (Y) to (Y2);
     \draw[densely dotted, thick, out=200, in=90] (Y2) to (Y3);
     \draw[densely dotted, thick, out=200, in=90] (Y3) to (Y4);
     \node [matrix,draw=black,rounded corners, row sep=3mm, thick] at (5.7,4.7,0.9){\draw (0.2,0) node[fill,circle, gray,opacity=.8]{}; & \node {$\mathtt{f}_1(0,\alpha)$};        \\ \draw (0.2,0) node[circle, very thick,draw,fill=white,fill opacity=.8]{}; & \node {$\mathtt{f}_1(0,\beta)$}; \\ \draw (0.2,0) node[circle,draw,fill=black,fill opacity=.8,scale=.55]{}; & \node {$\mathtt{f}_1(\alpha,\beta)$}; \\ \draw[densely dotted, thick] (0.05,0) to (0.4,0); & \node[align=left] {the path \\ of a walk}; \\ \draw[dashdotted] (0.05,0) to (0.4,0); & \node[align=left] {the line \\ of a club}; \\};
\end{tikzpicture}
\caption{Walks and clubs in $\mathtt{f}_1$}
\label{walksclubs}
\end{figure}
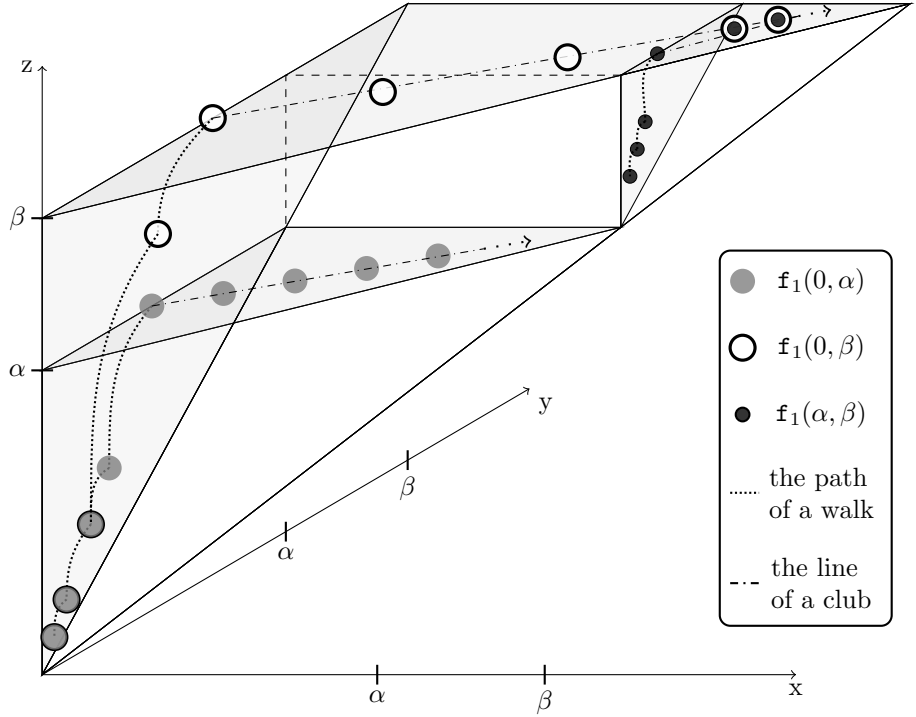

This nontriviality may be seen in either of (at least) two ways; both generalize to higher dimensions. For the first, observe that, by equations \ref{45} and \ref{4110}, we may define the following function for any $\beta<\gamma<\omega_1$ in which $\beta$ is a limit ordinal:
\begin{align}\label{mfunction}
m(\beta,\gamma)=\min \{\eta_i^\beta\,|\,(\eta_j^\beta,\eta_{j+1}^\beta,\beta) \in \text{supp}(\mathtt{f}_1(0,\gamma))\text{ for all }j\geq i\}
\end{align}
(Recall that $(\eta_j^\beta)$ is the increasing enumeration of $C_\beta$.) Visually, $m(\beta,\gamma)$ may be identified with the leftmost point in the $z=\beta$ plane on the left side of Figure \ref{yslices}. Equation \ref{46} applied to $0<\alpha<\gamma<\omega_1$ entails that
\begin{align}
\{\beta\in(\alpha,\gamma)\cap\text{Lim}\,|\,m(\beta,\gamma)<\alpha\}\text{ is finite for any }\alpha<\gamma<\omega_1\label{n11}
\end{align}

\hspace{-.3 cm} Now suppose that some $\mathtt{e}_0(0)\in\prod_{[\omega_1]^3}\mathbb{Z}$ satisfied equation \ref{48}. Then for all limit ordinals $\beta<\omega_1$, the function
\begin{align*}
n(\beta)=\min \{\eta_i^\beta\,|\,(\eta_j^\beta,\eta_{j+1}^\beta,\beta) \in \text{supp}(\mathtt{e}_0(0))\text{ for all }j\geq i\}
\end{align*}
would be defined. By the Pressing Down Lemma, the function $\text{Lim}\cap\omega_1\to\omega_1 :\beta\mapsto n(\beta)$ would be constantly $\eta$ on some stationary $S\subseteq\omega_1$. Now take $\gamma$ with $S\cap\gamma$ infinite, and $\alpha$ in the interval $(\eta,\gamma)$. By equation \ref{48}, $m(\beta,\gamma)=\eta$ for infinitely many $\beta\in(\alpha,\gamma)$, contradicting fact \ref{n11}.
\begin{figure}
\begin{tikzpicture}[MyPersp,font=\large]
	\coordinate (A) at (0,0,0);
	\coordinate (B) at (3.3,0,3.3);
	\coordinate (E) at (0,0,3.3);
	\coordinate (e) at (0,0,2.4);
	\coordinate (b) at (2.4,0,2.4);
	\coordinate (C) at (1,0,1);
	\coordinate (c) at (1,0,0);
	\coordinate (d) at (0,0,1);
	\coordinate (D) at (1,0,2.4);
	\coordinate (f) at (2.4,0,0);
	\coordinate (e1) at (0,0,2.4);
	\coordinate (e2) at (.22,0,2.4);
	\coordinate (e3) at (.44,0,2.4);
	\coordinate (e4) at (.66,0,2.4);
	\coordinate (e5) at (.88,0,2.4);
	\coordinate (e6) at (1.1,0,2.4);
	\coordinate (e7) at (1.32,0,2.4);
	\coordinate (e8) at (1.54,0,2.4);
	\coordinate (e9) at (1.76,0,2.4);
	\coordinate (e10) at (1.98,0,2.4);
	\coordinate (e11) at (2.2,0,2.4);
	\coordinate (e12) at (1.32,0,2.4);
	\coordinate (j5) at (.9,0,2.1);
	\coordinate (j6) at (1.15,0,2.1);
	\coordinate (j7) at (1.4,0,2.1);
	\coordinate (j8) at (1.65,0,2.1);
	\coordinate (j9) at (1.9,0,2.1);
	\coordinate (k5) at (1.14,0,1.8);
	\coordinate (k6) at (1.31,0,1.8);
	\coordinate (k7) at (1.48,0,1.8);
	\coordinate (k8) at (1.65,0,1.8);
	\coordinate (l5) at (.84,0,1.5);
	\coordinate (l6) at (1.1,0,1.5);
	\coordinate (l7) at (1.36,0,1.5);
	\coordinate (m1) at (.3,0,1.2);
	\coordinate (m2) at (.5,0,1.2);
	\coordinate (m3) at (.7,0,1.2);
	\coordinate (m4) at (.89,0,1.2);
	\coordinate (m5) at (1.08,0,1.2);
	\coordinate (n1) at (.2,0,.9);
	\coordinate (n2) at (.37,0,.9);
	\coordinate (n3) at (.54,0,.9);
	\coordinate (n4) at (.71,0,.9);
	\coordinate (o1) at (.1,0,.6);
	\coordinate (o2) at (.28,0,.6);
	\coordinate (o3) at (.46,0,.6);
	\coordinate (p1) at (.14,0,.3);

	\coordinate (Aa) at (6,0,0);
	\coordinate (Ba) at (9.3,0,3.3);
	\coordinate (Ea) at (6,0,3.3);
	\coordinate (ea) at (6,0,1.8);
	\coordinate (ba) at (8.4,0,2.4);
	\coordinate (Ca) at (7,0,1);
	\coordinate (ca) at (7,0,0);
	\coordinate (da) at (6,0,1);
	\coordinate (Da) at (7,0,3.3);
	\coordinate (fa) at (8.4,0,0);
	\coordinate (e1a) at (6,0,2.4);
	\coordinate (e2a) at (6.22,0,2.4);
	\coordinate (e3a) at (6.44,0,2.4);
	\coordinate (e4a) at (6.658,0,2.4);
	\coordinate (e5a) at (6.88,0,2.4);
	\coordinate (e6a) at (7.1,0,2.4);
	\coordinate (e7a) at (7.32,0,2.4);
	\coordinate (e8a) at (7.54,0,2.4);
	\coordinate (e9a) at (7.76,0,2.4);
	\coordinate (e10a) at (7.98,0,2.4);
	\coordinate (e11a) at (8.2,0,2.4);
	\coordinate (e12a) at (7.32,0,2.4);
	\coordinate (j5a) at (6.9,0,2.1);
	\coordinate (j6a) at (7.15,0,2.1);
	\coordinate (j7a) at (7.4,0,2.1);
	\coordinate (j8a) at (7.65,0,2.1);
	\coordinate (j9a) at (7.9,0,2.1);
	\coordinate (k5a) at (7.14,0,1.8);
	\coordinate (k6a) at (7.31,0,1.8);
	\coordinate (k7a) at (7.48,0,1.8);
	\coordinate (k8a) at (7.65,0,1.8);
	\coordinate (l5a) at (6.84,0,1.5);
	\coordinate (l6a) at (7.1,0,1.5);
	\coordinate (l7a) at (7.36,0,1.5);
	\coordinate (m1a) at (6.3,0,1.2);
	\coordinate (m2a) at (6.5,0,1.2);
	\coordinate (m3a) at (6.7,0,1.2);
	\coordinate (m4a) at (6.89,0,1.2);
	\coordinate (m5a) at (7.08,0,1.2);
	\coordinate (n1a) at (6.2,0,.9);
	\coordinate (n2a) at (6.37,0,.9);
	\coordinate (n3a) at (6.54,0,.9);
	\coordinate (n4a) at (6.71,0,.9);
	\coordinate (o1a) at (6.1,0,.6);
	\coordinate (o2a) at (6.28,0,.6);
	\coordinate (o3a) at (6.46,0,.6);
	\coordinate (p1a) at (6.14,0,.3);
	\coordinate (q1a) at (6.65,0,2.7);
	\coordinate (q2a) at (6.89,0,2.7);
	\coordinate (q3a) at (7.13,0,2.7);
	\coordinate (q4a) at (7.37,0,2.7);
	\coordinate (q5a) at (7.61,0,2.7);
	\coordinate (q6a) at (7.85,0,2.7);
	\coordinate (q7a) at (8.09,0,2.7);
	\coordinate (q8a) at (8.33,0,2.7);
	\coordinate (q9a) at (8.57,0,2.7);
	\coordinate (r1a) at (6.26,0,3);
	\coordinate (r2a) at (6.58,0,3);
	\coordinate (r3a) at (6.9,0,3);
	\coordinate (r4a) at (7.22,0,3);
	\coordinate (r5a) at (7.54,0,3);
	\coordinate (r6a) at (7.86,0,3);
	\coordinate (r7a) at (8.18,0,3);
	\coordinate (r8a) at (8.5,0,3);
	\coordinate (r9a) at (8.82,0,3);

\coordinate (S) at (4,0,0);
\coordinate (Sa) at (10,0,0);
\coordinate (T) at (0,0,4);
\coordinate (Ta) at (6,0,4);
\coordinate (X) at (3.3,0,0);
\coordinate (Xa) at (9.3,0,0);

	\coordinate (E-) at (-.1,0,3.3);
	\coordinate (E+) at (.1,0,3.3);	

	\coordinate (W-) at (3.3,0,-.1);
	\coordinate (W+) at (3.3,0,.1);
	
	\coordinate (d-) at (-.1,0,1);
	\coordinate (d+) at (.1,0,1);	

	\coordinate (c-) at (1,0,-.1);
	\coordinate (c+) at (1,0,.1);

	\coordinate (e-) at (-.1,0,2.4);
	\coordinate (e+) at (.1,0,2.4);	
	
	\coordinate (ee-) at (-.1,0,1.8);
	\coordinate (ee+) at (.1,0,1.8);	
	\coordinate (ee) at (0,0,1.8);	

	\coordinate (f-) at (2.4,0,-.1);
	\coordinate (f+) at (2.4,0,.1);
	
		\coordinate (Ea-) at (5.9,0,3.3);
	\coordinate (Ea+) at (6.1,0,3.3);	

	\coordinate (Wa-) at (9.3,0,-.1);
	\coordinate (Wa+) at (9.3,0,.1);
	
	\coordinate (da-) at (5.9,0,1);
	\coordinate (da+) at (6.1,0,1);	

	\coordinate (ca-) at (7,0,-.1);
	\coordinate (ca+) at (7,0,.1);

	\coordinate (ea-) at (5.9,0,1.8);
	\coordinate (ea+) at (6.1,0,1.8);	

	\coordinate (fa-) at (8.4,0,-.1);
	\coordinate (fa+) at (8.4,0,.1);

	\draw (e1) node[fill,circle,inner sep=0pt,minimum size=4pt, black, opacity=.95]{};
		\draw (e2) node[fill,circle,inner sep=0pt,minimum size=4pt, black, opacity=.95]{};
			\draw (e3) node[fill,circle,inner sep=0pt,minimum size=4pt, black, opacity=.95]{};
				\draw (e4) node[fill,circle,inner sep=0pt,minimum size=4pt, black, opacity=.95]{};
					\draw (e5) node[fill,circle,inner sep=0pt,minimum size=4pt, black, opacity=.95]{};
						\draw (e6) node[fill,circle,inner sep=0pt,minimum size=4pt, black, opacity=.95]{};
							\draw (e7) node[fill,circle,inner sep=0pt,minimum size=4pt, black, opacity=.95]{};
			\draw (e8) node[fill,circle,inner sep=0pt,minimum size=4pt, black, opacity=.95]{};
					\draw (e9) node[fill,circle,inner sep=0pt,minimum size=4pt, black, opacity=.95]{};
						\draw (e10) node[fill,circle,inner sep=0pt,minimum size=4pt, black, opacity=.95]{};
							\draw (e11) node[fill,circle,inner sep=0pt,minimum size=4pt, black, opacity=.95]{};
							\draw (j5) node[fill,circle,inner sep=0pt,minimum size=4pt, black, opacity=.95]{};
		\draw (j6) node[fill,circle,inner sep=0pt,minimum size=4pt, black, opacity=.95]{};
			\draw (j7) node[fill,circle,inner sep=0pt,minimum size=4pt, black, opacity=.95]{};
				\draw (j8) node[fill,circle,inner sep=0pt,minimum size=4pt, black, opacity=.95]{};
					\draw (j9) node[fill,circle,inner sep=0pt,minimum size=4pt, black, opacity=.95]{};
			\draw (k5) node[fill,circle,inner sep=0pt,minimum size=4pt, black, opacity=.95]{};
		\draw (k6) node[fill,circle,inner sep=0pt,minimum size=4pt, black, opacity=.95]{};
			\draw (k7) node[fill,circle,inner sep=0pt,minimum size=4pt, black, opacity=.95]{};
				\draw (k8) node[fill,circle,inner sep=0pt,minimum size=4pt, black, opacity=.95]{};
\draw (l5) node[fill,circle,inner sep=0pt,minimum size=4pt, black, opacity=.95]{};
						\draw (l6) node[fill,circle,inner sep=0pt,minimum size=4pt, black, opacity=.95]{};
							\draw (l7) node[fill,circle,inner sep=0pt,minimum size=4pt, black, opacity=.95]{};
								\draw (m1) node[fill,circle,inner sep=0pt,minimum size=4pt, black, opacity=.95]{};
		\draw (m2) node[fill,circle,inner sep=0pt,minimum size=4pt, black, opacity=.95]{};
			\draw (m3) node[fill,circle,inner sep=0pt,minimum size=4pt, black, opacity=.95]{};
				\draw (m4) node[fill,circle,inner sep=0pt,minimum size=4pt, black, opacity=.95]{};
					\draw (m5) node[fill,circle,inner sep=0pt,minimum size=4pt, black, opacity=.95]{};	
		\draw (n1) node[fill,circle,inner sep=0pt,minimum size=4pt, black, opacity=.95]{};
		\draw (n2) node[fill,circle,inner sep=0pt,minimum size=4pt, black, opacity=.95]{};
		\draw (n3) node[fill,circle,inner sep=0pt,minimum size=4pt, black, opacity=.95]{};
		\draw (n4) node[fill,circle,inner sep=0pt,minimum size=4pt, black, opacity=.95]{};			
				\draw (o1) node[fill,circle,inner sep=0pt,minimum size=4pt, black, opacity=.95]{};
		\draw (o2) node[fill,circle,inner sep=0pt,minimum size=4pt, black, opacity=.95]{};
		\draw (o3) node[fill,circle,inner sep=0pt,minimum size=4pt, black, opacity=.95]{};
		\draw (p1) node[fill,circle,inner sep=0pt,minimum size=4pt, black, opacity=.95]{};		
				
	\draw (X) node[inner sep=6pt, below] {$\omega_1$};
	\draw (f) node[inner sep=6pt, below] {$\gamma$};
	\draw (c) node[inner sep=6pt, below] {$\alpha$};
	\draw[thick] (ee-)--(ee+);	
	\draw[thick] (E-)--(E+);
	\draw[thick] (W-)--(W+);
	\draw[thick] (d-)--(d+);
	\draw[thick] (c-)--(c+);
	\draw[thick] (e-)--(e+);
	\draw[thick] (f-)--(f+);
	\draw (E) node[inner sep=6pt, left] {$\omega_1$};
		\draw (d) node[inner sep=6pt, left] {$\alpha$};
		\draw (ee) node[inner sep=6pt, left] {$\beta$};
			\draw (e) node[inner sep=6pt, left] {$\gamma$};
	\draw[->] (A)--(S) node[below] {x};
	\draw[->] (A)--(T) node[left] {z};
\draw[->] (A)--(B);
\draw (A)--(e)--(b)--cycle;
\draw[dashed] (d)--(C)--(D);
\fill[fill=gray, opacity=.25] (A)--(d)--(C);
\fill[fill=gray, opacity=.25] (C)--(D)--(b);
\fill[fill=gray, opacity=.1] (d)--(C)--(D)--(e);

	\draw (Xa) node[inner sep=6pt, below] {$\omega_1$};
	\draw (fa) node[inner sep=6pt, below] {$\gamma$};
	\draw (ca) node[inner sep=6pt, below] {$\alpha$};
	\draw[thick] (Ea-)--(Ea+);
	\draw[thick] (Wa-)--(Wa+);
	\draw[thick] (da-)--(da+);
	\draw[thick] (ca-)--(ca+);
	\draw[thick] (ea-)--(ea+);
	\draw[thick] (fa-)--(fa+);
	\draw (Ea) node[inner sep=6pt, left] {$\omega_1$};
		\draw (da) node[inner sep=6pt, left] {$\alpha$};
			\draw (ea) node[inner sep=6pt, left] {$\beta$};
	\draw[->] (Aa)--(Sa) node[below] {x};
	\draw[->] (Aa)--(Ta) node[left] {z};
\draw[dotted] (Ea)--(Ba);
\draw (Ea)--(Aa)--(Ba);
\draw[dashed] (da)--(Ca)--(Da);
\fill[fill=gray, opacity=.25] (Aa)--(da)--(Ca);
\fill[fill=gray, opacity=.25] (Ca)--(Da)--(Ba);
\fill[fill=gray, opacity=.1] (da)--(Ca)--(Da)--(Ea);

		\draw (e2a) node[fill,circle,inner sep=0pt,minimum size=4pt, black, opacity=.95]{};
			\draw (e3a) node[fill,circle,inner sep=0pt,minimum size=4pt, black, opacity=.95]{};
				\draw (e4a) node[fill,circle,inner sep=0pt,minimum size=4pt, black, opacity=.95]{};
					\draw (e5a) node[fill,circle,inner sep=0pt,minimum size=4pt, black, opacity=.95]{};
						\draw (e6a) node[fill,circle,inner sep=0pt,minimum size=4pt, black, opacity=.95]{};
							\draw (e7a) node[fill,circle,inner sep=0pt,minimum size=4pt, black, opacity=.95]{};
			\draw (e8a) node[fill,circle,inner sep=0pt,minimum size=4pt, black, opacity=.95]{};
					\draw (e9a) node[fill,circle,inner sep=0pt,minimum size=4pt, black, opacity=.95]{};
						\draw (e10a) node[fill,circle,inner sep=0pt,minimum size=4pt, black, opacity=.95]{};
							\draw (e11a) node[fill,circle,inner sep=0pt,minimum size=4pt, black, opacity=.95]{};
							\draw (j5a) node[fill,circle,inner sep=0pt,minimum size=4pt, black, opacity=.95]{};
		\draw (j6a) node[fill,circle,inner sep=0pt,minimum size=4pt, black, opacity=.95]{};
			\draw (j7a) node[fill,circle,inner sep=0pt,minimum size=4pt, black, opacity=.95]{};
				\draw (j8a) node[fill,circle,inner sep=0pt,minimum size=4pt, black, opacity=.95]{};
					\draw (j9a) node[fill,circle,inner sep=0pt,minimum size=4pt, black, opacity=.95]{};
			\draw (k5a) node[fill,circle,inner sep=0pt,minimum size=4pt, black, opacity=.95]{};
		\draw (k6a) node[fill,circle,inner sep=0pt,minimum size=4pt, black, opacity=.95]{};
			\draw (k7a) node[fill,circle,inner sep=0pt,minimum size=4pt, black, opacity=.95]{};
				\draw (k8a) node[fill,circle,inner sep=0pt,minimum size=4pt, black, opacity=.95]{};
\draw (l5a) node[fill,circle,inner sep=0pt,minimum size=4pt, black, opacity=.95]{};
						\draw (l6a) node[fill,circle,inner sep=0pt,minimum size=4pt, black, opacity=.95]{};
							\draw (l7a) node[fill,circle,inner sep=0pt,minimum size=4pt, black, opacity=.95]{};
								\draw (m1a) node[fill,circle,inner sep=0pt,minimum size=4pt, black, opacity=.95]{};
		\draw (m2a) node[fill,circle,inner sep=0pt,minimum size=4pt, black, opacity=.95]{};
			\draw (m3a) node[fill,circle,inner sep=0pt,minimum size=4pt, black, opacity=.95]{};
				\draw (m4a) node[fill,circle,inner sep=0pt,minimum size=4pt, black, opacity=.95]{};
					\draw (m5a) node[fill,circle,inner sep=0pt,minimum size=4pt, black, opacity=.95]{};	
		\draw (n1a) node[fill,circle,inner sep=0pt,minimum size=4pt, black, opacity=.95]{};
		\draw (n2a) node[fill,circle,inner sep=0pt,minimum size=4pt, black, opacity=.95]{};
		\draw (n3a) node[fill,circle,inner sep=0pt,minimum size=4pt, black, opacity=.95]{};
		\draw (n4a) node[fill,circle,inner sep=0pt,minimum size=4pt, black, opacity=.95]{};			
				\draw (o1a) node[fill,circle,inner sep=0pt,minimum size=4pt, black, opacity=.95]{};
		\draw (o2a) node[fill,circle,inner sep=0pt,minimum size=4pt, black, opacity=.95]{};
		\draw (o3a) node[fill,circle,inner sep=0pt,minimum size=4pt, black, opacity=.95]{};
		\draw (p1a) node[fill,circle,inner sep=0pt,minimum size=4pt, black, opacity=.95]{};	
		\draw (q1a) node[fill,circle,inner sep=0pt,minimum size=4pt, black, opacity=.95]{};
		\draw (q2a) node[fill,circle,inner sep=0pt,minimum size=4pt, black, opacity=.95]{};
		\draw (q3a) node[fill,circle,inner sep=0pt,minimum size=4pt, black, opacity=.95]{};
		\draw (q4a) node[fill,circle,inner sep=0pt,minimum size=4pt, black, opacity=.95]{};	
		\draw (q5a) node[fill,circle,inner sep=0pt,minimum size=4pt, black, opacity=.95]{};	
		\draw (q6a) node[fill,circle,inner sep=0pt,minimum size=4pt, black, opacity=.95]{};
			\draw (q7a) node[fill,circle,inner sep=0pt,minimum size=4pt, black, opacity=.95]{};
				\draw (q8a) node[fill,circle,inner sep=0pt,minimum size=4pt, black, opacity=.95]{};
					\draw (q9a) node[fill,circle,inner sep=0pt,minimum size=4pt, black, opacity=.95]{};	
							\draw (r1a) node[fill,circle,inner sep=0pt,minimum size=4pt, black, opacity=.95]{};
		\draw (r2a) node[fill,circle,inner sep=0pt,minimum size=4pt, black, opacity=.95]{};
		\draw (r3a) node[fill,circle,inner sep=0pt,minimum size=4pt, black, opacity=.95]{};
		\draw (r4a) node[fill,circle,inner sep=0pt,minimum size=4pt, black, opacity=.95]{};	
		\draw (r5a) node[fill,circle,inner sep=0pt,minimum size=4pt, black, opacity=.95]{};	
		\draw (r6a) node[fill,circle,inner sep=0pt,minimum size=4pt, black, opacity=.95]{};
			\draw (r7a) node[fill,circle,inner sep=0pt,minimum size=4pt, black, opacity=.95]{};
				\draw (r8a) node[fill,circle,inner sep=0pt,minimum size=4pt, black, opacity=.95]{};
					\draw (r9a) node[fill,circle,inner sep=0pt,minimum size=4pt, black, opacity=.95]{};

\end{tikzpicture}
\caption{Schematic profile views of $\mathtt{f}_1(0,\gamma)$ and a candidate trivialization $\mathtt{e}_0(0)$. In each, the lighter shaded box denotes the restricted domain $\alpha\otimes\omega_1 \otimes(\alpha,\omega_1)$. By equation \ref{46}, the support therein of $\mathtt{f}_1(0,\gamma)$ (depicted as black dots) is finite for any $\alpha<\gamma$, but the support therein of any trivializing $\mathtt{e}_0(0)$ must be uncountable for some $\alpha<\omega_1$, entailing contradiction.}
\label{yslices}
\end{figure}
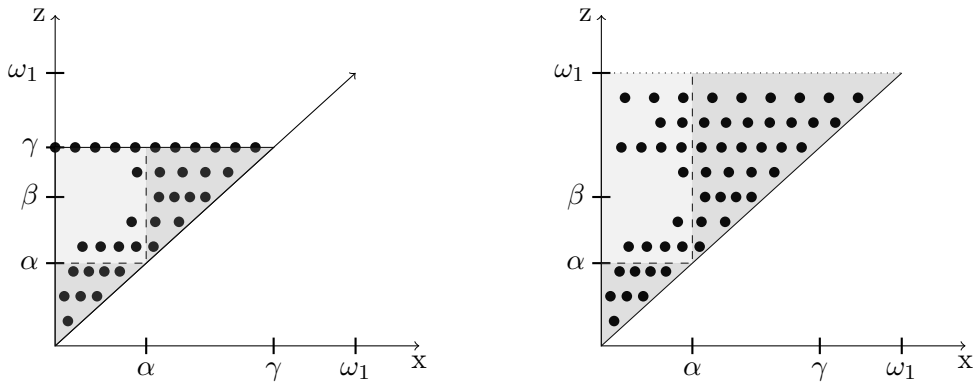

For the second argument, observe simply that any $\mathtt{e}_0(0)$ satisfying equation \ref{48} satisfies\begin{align*}
\mathtt{e}_0(0)\big|_{\beta\otimes\omega_1\otimes(\beta,\omega_1)}\,=^*0,
\end{align*}
by equation \ref{44}. Hence the function
\begin{align*}
g:\omega_1\to\omega_1\,:\,\beta\mapsto\min\Big\{\eta\,\Big|\,\mathtt{e}_0(0)\big|_{\xi\otimes\omega_1\otimes[\eta,\omega_1)}=0\text{ for all }\xi<\beta\Big\}
\end{align*}
is a well-defined increasing continuous function. Let $\gamma$ denote an infinite-cofinality fixed point of this function. Then $\mathtt{e}_0(0)$ restricted to the $z=\gamma$ plane is $0$; hence $\mathtt{e}_0(0)$ disagrees infinitely often thereon with $\mathtt{f}_1(0,\gamma)$, violating equation \ref{48}.
\subsection{The cases of higher $n$}\label{efffn}

The fundamentals of the higher-order cases are all visible already in the case of $n=2$. Here the coherence of the function $\mathtt{f}_2:[\omega_2]^3\rightarrow\prod_{[\omega_2]^4}\mathbb{Z}$ takes the following form:
\begin{align}\label{419}\mathtt{f}_2(\beta,\gamma,\delta)-\mathtt{f}_2(\alpha,\gamma,\delta)+\mathtt{f}_2(\alpha,\beta,\delta)-\mathtt{f}_2(\alpha,\beta,\gamma)=^*0\textnormal{ for all }\alpha<\beta<\gamma<\delta<\omega_2
\end{align}
$\mathtt{f}_2$ is nontrivial if there exists no $\mathtt{e}_1\in K^1(\mathbf{R}_3(\omega_2))$ satisfying
\begin{align}\label{429}
\mathtt{e}_1(\beta,\gamma)-\mathtt{e}_1(\alpha,\gamma)+\mathtt{e}_1(\alpha,\beta)=^*\mathtt{f}_2(\alpha,\beta,\gamma)\text{ for all }\alpha<\beta<\gamma<\omega_2
\end{align}
Again the statement ``$\mathtt{e}_1\in K^1(\mathbf{R}_3(\omega_2))$'' abbreviates \begin{align}\label{4291}
\mathtt{e}_1:[\omega_2]^2\rightarrow\prod_{[\omega_2]^4}\mathbb{Z}\,\textnormal{ and }\,\textnormal{supp}(\mathtt{e}_1(\beta,\gamma))\subseteq [\,[\beta,\omega_2)\,]^4\textnormal{ for all }\beta<\gamma<\omega_2
\end{align}
Again $\mathtt{f}_2$ admits a straightforward definition, that of (\ref{efff}) and (\ref{effff}):
\begin{align}\label{439}
  \mathtt{f}_2(\alpha,\beta,\gamma) = \left\{\def\arraystretch{1}%
  \begin{array}{@{}c@{\quad}l@{}}
    0 & \hspace{.45 cm}\textnormal{if }\beta\in C_\gamma\textnormal{ but }\\ & \hspace{.4 cm} C^{\beta\gamma}(\alpha)\textnormal{ is undefined} \\
    -\langle\alpha,C^{\beta\gamma}(\alpha),\beta,\gamma\rangle+\mathtt{f}_2(C^{\beta\gamma}(\alpha),\beta,\gamma) & \\ \;+\mathtt{f}_2(\alpha,C^{\beta\gamma}(\alpha),\gamma)-\mathtt{f}_2(\alpha,C^{\beta\gamma}(\alpha),\beta) & \hspace{.3 cm}\textnormal{ if }\beta\in C_\gamma\textnormal{ and }\\ & \hspace{.4 cm} C^{\beta\gamma}(\alpha)\textnormal{ is defined} \\
     \langle\alpha,\beta,C^\gamma(\beta),\gamma\rangle-\mathtt{f}_2(\beta,C^\gamma(\beta),\gamma) & \\ \;+\mathtt{f}_2(\alpha,C^\gamma(\beta),\gamma)+\mathtt{f}_2(\alpha,\beta,C^\gamma(\beta)) & \hspace{.3 cm}\textnormal{ if }\beta\notin C_\gamma 
  \end{array}\right.
\end{align}
(If $\beta\in C_\gamma$, then $C^{\beta\gamma}(\alpha)$ is undefined when $C_{\beta\gamma}\subseteq\alpha+1$. The case $\gamma=\beta+1$, wherein $C_{\beta\gamma}=\varnothing$, is an instance.) As before it is immediate from this definition that
\begin{align}\label{449}\text{supp}(\mathtt{f}_2(\alpha,\beta,\gamma))\subseteq [\,[\alpha,\gamma]\,]^4\end{align}
for all $\alpha<\beta<\gamma<\omega_2$. Support considerations again then afford us a reduction of equation \ref{419}:
\begin{align}\label{efftwo}(\mathtt{f}_2(0,\beta,\gamma)-\mathtt{f}_2(0,\alpha,\gamma)+\mathtt{f}_2(0,\alpha,\beta))\big|_{\alpha\otimes[\omega_2]^3}=^*0\textnormal{ for all }\alpha<\beta<\gamma<\omega_2\end{align}
Observe now that for any $\hat{\mathtt{e}}\in K^1(\mathbf{R}_3(\omega_2))$ satisfying equation \ref{429}, the function $$\mathtt{e}_1(0,\,\cdot\,):\omega_2\to \prod_{[\omega_2]^4}\mathbb{Z}\,:\,\beta\mapsto -\hat{\mathtt{e}}_1(0,\beta)$$ will satisfy
\begin{align} (\mathtt{e}_1(0,\beta)-\mathtt{e}_1(0,\alpha))\big|_{\alpha\otimes[\omega_2]^3}=^{*}\mathtt{f}_2(0,\alpha,\beta)\big|_{\alpha\otimes[\omega_2]^3}\hspace{.3 cm}\textnormal{ for all }\alpha<\beta<\omega_2,\label{e1}
\end{align}
by equation \ref{4291}. (In fact there exists an $\mathtt{e}_1$ as in equation \ref{e1} if and only if there exists an $\mathtt{e}_1$ as in equation \ref{429}, just as in the $n=1$ case, as the reader may verify). Hence to show the nontriviality of $\mathtt{f}_2$, it will suffice to show that no $\mathtt{e}_1(0,\,\cdot\,)$ as in equation \ref{e1} can exist.

As in the $n=1$ case, this nontriviality derives ultimately from the \emph{lower-order nontriviality} manifesting in the $z=\beta$ hyperplanes of the system $\mathtt{f}_2$. Again a picture may be of use; what Figure \ref{zslices} aims above all to convey is the following: for any $\beta\in\omega_2$ of uncountable cofinality, the family $\{\mathtt{f}_2(0,\alpha,\beta)\big|_{\alpha\otimes[\omega_2]^2\otimes\{\beta\}}\;|\;\alpha\in C_\beta\}$ is a copy --- or what might be more precisely described as a relativization to the club $C_\beta$ --- of the nontrivial family $\{\mathtt{f}_1(0,\alpha)\;|\;\alpha\in \omega_1\}$. This the reader may verify by comparing the first three entries of the second alternative in equation \ref{439} to equation \ref{43}; this is the grounding recognition for either of two arguments paralleling (in reverse order) those for the nontriviality of $\mathtt{f}_1$.
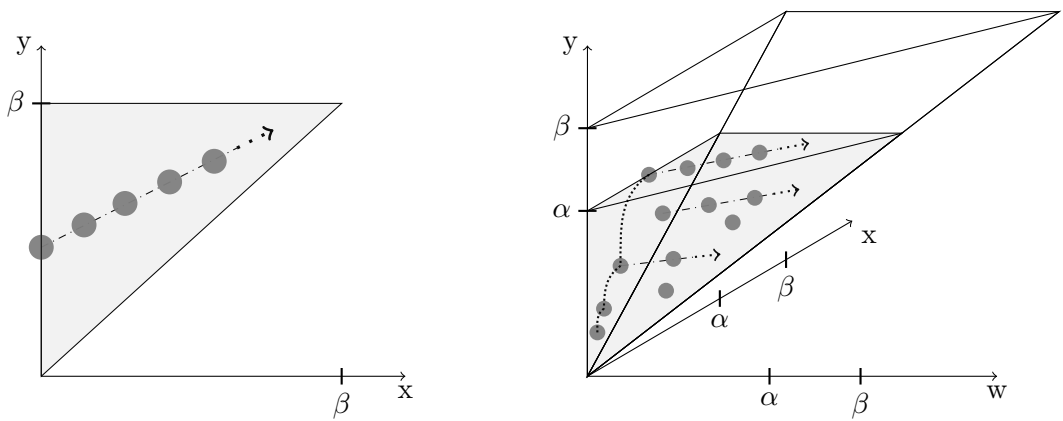
\begin{figure}
\begin{tikzpicture}[MyPersp,font=\large]
	\coordinate (A) at (0,0,0);
	\coordinate (B) at (3.3,0,3.3);
	\coordinate (D) at (0,0,1.56);
	\coordinate (E) at (0,0,3.3);
	\coordinate (F) at (0.47,0,1.83);	
	\coordinate (G) at (0.92,0,2.09);	
	\coordinate (H) at (1.41,0,2.35);	
	\coordinate (I) at (1.9,0,2.6);	
	\coordinate (J) at (2.15,0,2.76);
	\coordinate (K) at (2.55,0,2.98);
	
	\coordinate (AA) at (6,0,0);
	\coordinate (BB) at (6,0,2);
	\coordinate (CC) at (6,2,2);
	\coordinate (DD) at (8,2,2);
	\coordinate (EE) at (6,0,3);
	\coordinate (FF) at (6,3,3);
	\coordinate (GG) at (9,3,3);
	\coordinate (HH) at (8,2,3);
	\coordinate (II) at (8,3,3);
	\coordinate (JJ) at (6,2,3);
	
	\coordinate (F0) at (6,2,0);	
	\coordinate (F1) at (8,0,0);
	\coordinate (F2) at (6,3,0);	
	\coordinate (F3) at (9,0,0); 
	
	\coordinate (K0) at (6,.93,2);	
	\coordinate (K1) at (6.3,1.1,2); 	
 	\coordinate (K2) at (6.55,1.3,2);	
	\coordinate (K3) at (6.8,1.5,2);
	\coordinate (K4) at (6.95,1.6,2);
	\coordinate (K5) at (7.17,1.74,2);
	\coordinate (L1) at (6.1,1,1.5);
	\coordinate (L2) at (6.45,1.215,1.5);
	\coordinate (L3) at (6.82,1.4,1.5);
	\coordinate (L4) at (6.95,1.5,1.5);
	\coordinate (L5) at (7.15,1.62,1.5);
	\coordinate (M1) at (6,.5,1.1);
	\coordinate (M2) at (6.45,.68,1.1);
	\coordinate (M3) at (6.65,.73,1.1);
	\coordinate (M4) at (6.88,.8,1.1);
	\coordinate (N) at (6.65,1.3,1.25);	
	\coordinate (O1) at (6,.25,.7);	
	\coordinate (O2) at (6.5,.5,.8);	
	\coordinate (O3) at (6,.15,.46);
	
	\coordinate (S) at (0,0,4);
	\coordinate (T) at (4,0,0);
	\coordinate (X) at (3.3,0,0);
	
	\coordinate (RR) at (6,4,0);
	\coordinate (SS) at (6,0,4);
	\coordinate (TT) at (10.5,0,0);
	\coordinate (UU) at (6,2,0);
	\coordinate (VV) at (6,3,0);
	\coordinate (WW) at (8,0,0);
	\coordinate (XX) at (9,0,0);

	\coordinate (E-) at (-.1,0,3.3);
	\coordinate (E+) at (.1,0,3.3);	

	\coordinate (W-) at (3.3,0,-.1);
	\coordinate (W+) at (3.3,0,.1);
	
	\coordinate (BB-) at (5.9,0,2);
	\coordinate (BB+) at (6.1,0,2);
	
	\coordinate (EE-) at (5.9,0,3);
	\coordinate (EE+) at (6.1,0,3);

	\coordinate (F0-) at (6,2,-.1);
	\coordinate (F0+) at (6,2,.1);
	
	\coordinate (F1-) at (8,0,.1);
	\coordinate (F1+) at (8,0,-.1);
	
	\coordinate (F2-) at (6,3,-.1);
	\coordinate (F2+) at (6,3,.1);
	
	\coordinate (F3-) at (9,0,.1);
	\coordinate (F3+) at (9,0,-.1);

	\draw[dashdotted] (D)--(J);
	\draw[dashdotted] (K0)--(K4);
	\draw[dashdotted] (L1)--(L4);
	\draw[dashdotted] (M1)--(M3);
	\draw (K0) node[fill,circle,inner sep=0pt,minimum size=6pt, gray,opacity=.95]{};
	\draw (K1) node[fill,circle,inner sep=0pt,minimum size=6pt, gray,opacity=.95]{};
	\draw (K2) node[fill,circle,inner sep=0pt,minimum size=6pt, gray,opacity=.95]{};
	\draw (K3) node[fill,circle,inner sep=0pt,minimum size=6pt, gray,opacity=.95]{};
	\draw (L1) node[fill,circle,inner sep=0pt,minimum size=6pt, gray,opacity=.95]{};
	\draw (L2) node[fill,circle,inner sep=0pt,minimum size=6pt, gray,opacity=.95]{};
	\draw (L3) node[fill,circle,inner sep=0pt,minimum size=6pt, gray,opacity=.95]{};
	\draw (M1) node[fill,circle,inner sep=0pt,minimum size=6pt, gray,opacity=.95]{};	
	\draw (M2) node[fill,circle,inner sep=0pt,minimum size=6pt, gray,opacity=.95]{};	
	\draw (N) node[fill,circle,inner sep=0pt,minimum size=6pt, gray,opacity=.95]{};	
	\draw (O1) node[fill,circle,inner sep=0pt,minimum size=6pt, gray,opacity=.95]{};	
	\draw (O2) node[fill,circle,inner sep=0pt,minimum size=6pt, gray,opacity=.95]{};	
	\draw (O3) node[fill,circle,inner sep=0pt,minimum size=6pt, gray,opacity=.95]{};	
 	\draw[->] (A)--(S) node[left] {y};
	\draw[->, thick, dotted] (K4)--(K5);
	\draw[->, thick, dotted] (L4)--(L5);
	\draw[->, thick, dotted] (M3)--(M4);
	\draw (D) node[fill,circle, gray,opacity=.95]{};	
	\draw (F) node[fill,circle, gray,opacity=.95]{};
	\draw (G) node[fill,circle, gray,opacity=.95]{};	
	\draw (H) node[fill,circle, gray,opacity=.95]{};
	\draw (I) node[fill,circle, gray,opacity=.95]{};
	\draw[->, very thick, loosely dotted] (J)--(K);
	\draw[densely dotted, thick, out=215, in=90] (K0) to (M1);
	\draw[densely dotted, thick, out=215, in=90] (M1) to (O1);
	\draw[densely dotted, thick, out=215, in=90] (O1) to (O3);
	\draw (X) node[inner sep=6pt, below] {$\beta$};
	\draw (F0) node[inner sep=6pt, below] {$\alpha$};
	\draw (F1) node[inner sep=6pt, below] {$\alpha$};
	\draw[thick] (E-)--(E+);
	\draw[thick] (W-)--(W+);
	\draw[thick] (BB-)--(BB+);
	\draw[thick] (EE-)--(EE+);
	\draw[thick] (F0-)--(F0+);
	\draw[thick] (F1-)--(F1+);
	\draw[thick] (F2-)--(F2+);
	\draw[thick] (F3-)--(F3+);
	\draw (F2) node[inner sep=6pt, below] {$\beta$};
	\draw (F3) node[inner sep=6pt, below] {$\beta$};
	\draw (E) node[inner sep=6pt, left] {$\beta$};
	\draw (EE) node[inner sep=6pt, left] {$\beta$};
	\draw (BB) node[inner sep=6pt, left] {$\alpha$};
	\draw[->] (A)--(T) node[below] {x};
\draw (A)--(B)--(E)--cycle;
\fill[fill=gray, opacity=.1] (A)--(B)--(E);
\fill[fill=gray, opacity=.1] (AA)--(BB)--(CC)--(DD);
\draw (AA)--(BB)--(CC)--cycle;
	\draw (AA)--(BB)--(DD)--cycle;
	\draw (AA)--(CC)--(DD)--cycle;
	\draw (AA)--(EE)--(FF)--cycle;
	\draw (AA)--(FF)--(GG)--cycle;
	\draw (AA)--(EE)--(GG)--cycle;
		\draw[->] (AA)--(RR) node[below right] {x};
	\draw[->] (AA)--(SS) node[left] {y};
	\draw[->] (AA)--(TT) node[below] {w};
\end{tikzpicture}
\caption{The supports of the $z=\beta$ ``slices'' of $\mathtt{f}_1(0,\beta)$ and $\mathtt{f}_2(0,\alpha,\beta)$, respectively. On the left, the axes are of length $\omega_1$ and $\beta$ is of cofinality $\aleph_0$; on the right, axes are of length $\omega_2$ and $\beta$ is of cofinality $\aleph_1$ and $\alpha$ is in $C_\beta$ and of cofinality $\aleph_0$. Arrowed ellipses connote continuation in the direction indicated: the nodes on the left form an image of $C_\beta$, as observed above, while arrowed subcollections of those on the right form images of $C_{\alpha\beta}$, etc. In the $w=0$ hyperplane on the right is an image of the $C_\beta$-internal walk from $\alpha$ down to $1$.}
\label{zslices}
\end{figure}

For the first argument, observe that equations \ref{e1} and \ref{449} together imply that $\mathtt{e}_1(0,\alpha)\big|_{\alpha\otimes [\omega_2]^3}$ and $\mathtt{e}_1(0,\beta)\big|_{\alpha\otimes [\omega_2]^3}$ perfectly agree on all but finitely many hyperplanes $z=\gamma$ satisfying $\gamma>\beta$. More precisely,
\begin{align}
err(\alpha,\beta)=\{\gamma>\beta\;|\:(\mathtt{e}_1(0,\beta)-\mathtt{e}_1(0,\alpha))\big|_{\alpha\otimes[\omega_2]^2\otimes\{\gamma\}}\neq 0\}\text{ is finite for all }\alpha<\beta<\omega_2
\label{err}
\end{align}
Hence
\begin{align*}
g:\omega_2\to\omega_2\,:\,\xi\mapsto\min\{\eta\,|\,err(\alpha,\beta)\cap [\eta,\omega_2)=\varnothing\text{ for all }\alpha<\beta<\xi\}
\end{align*}
is a continuous increasing function. Denote its collection of fixed points $E_g$ and take $\gamma\in E_g\cap S_1^2$; now
\begin{align*}
\big(\mathtt{e}_1(0,\gamma)\big|_{\gamma\otimes[\omega_2]^2\otimes\{\gamma\}}-\bigcup_{\xi<\gamma}\mathtt{e}_1(0,\xi)\big|_{\xi\otimes[\omega_2]^2\otimes\{\gamma\}}\big)\big|_{\beta\otimes[\omega_2]^2\otimes\{\gamma\}}=^*\mathtt{f}_2(0,\beta,\gamma)\big|_{\beta\otimes[\omega_2]^2\otimes\{\gamma\}}
\end{align*}
for each $\beta\in C_\gamma$. This, though, implies that the family $\{\mathtt{f}_2(0,\beta,\gamma)\big|_{\beta\otimes [\omega_2]^2\otimes\{\gamma\}}\;|\;\beta\in C_\gamma\}$ is trivial, contradicting our observation in the preceding paragraph.

For the second argument, assume again for contradiction the existence of an $\mathtt{e}_1(0,\,\cdot\,)$ satisfying equation \ref{e1}. Observe that for each $\gamma\in S_1^2$, as $F_\gamma:=\{\mathtt{f}_2(0,\beta,\gamma)\big|_{\beta\otimes[\omega_2]^2\otimes\{\gamma\}}\;|\;\beta\in C_\gamma\}$ is nontrivial, there exist $\alpha_\gamma<\beta_\gamma$ in $C_\gamma$ such that
\begin{align}\label{adapt}
\big(\mathtt{e}_1(0,\beta_\gamma)-\mathtt{e}_1(0,\alpha_\gamma)\big)\big|_{\alpha_\gamma\otimes [\omega_2]^2\otimes\{\gamma\}}\neq 0.
\end{align}
(Again if there were no such $\alpha_\gamma<\beta_\gamma$ then the function $\mathtt{e}_1(0,\gamma)\big|_{[\omega_2]^3\otimes\{\gamma\}}-\bigcup_{\beta\in C_\gamma}\mathtt{e}_1(0,\beta)\big|_{\beta\otimes [\omega_2]^2\otimes \{\gamma\}}$ would trivialize $F_\gamma$.) Hence by the Pressing Down Lemma, there exists a stationary $S\subseteq S_1^2$ and $\alpha<\beta$ such that $\alpha_\gamma=\alpha$ and $\beta_\gamma=\beta$ for each $\gamma\in S$. This, though, contradicts observation \ref{err} above.

These two arguments are, arguably, in principle the same. Each merely emphasizes a different bar to triviality; in the latter case, lower-order nontriviality in $z=\gamma$ hyperplanes $(\gamma\in S_1^2)$ enforces disagreement thereon between some lower-index $\mathtt{e}_1(0,\,\cdot\,)$ terms, which the Pressing Down Lemma then concentrates on a pair of terms for a contradiction. In the other argument, the eventual finitude of the disagreements between pairs of $\mathtt{e}_1(0,\,\cdot\,)$ terms allows for closing-off arguments, entailing contradiction when they intersect with $S_1^2$. In each case, the nontriviality principle is one first appearing at $\omega_2$; as the reader may verify, all the ``initial'' families $\{\mathtt{f}_2(0,\vec{\beta})\mid\vec{\beta}\in [\delta\backslash\{0\}]^2\}$ \emph{are} trivial in the sense of (\ref{e1}).\footnote{For any such $\delta\in\omega_2$, the assignment $\mathtt{e}_1(0,\beta)=\mathtt{f}_2(0,\beta,\gamma)$ defines a trivialization of $\{\mathtt{f}_2(0,\vec{\beta})\mid\vec{\beta}\in [\delta\backslash\{0\}]^2\}$ in the sense of (\ref{e1}), by equation \ref{efftwo}.}

In the interests of clarity, we conclude this section by simply outlining the argument of the nontriviality of the functions $\mathtt{f}_n$ $(n>2)$; at least on a first reading, the greater detail in which it is recorded in Appendix \ref{appendixAA} may serve only to obscure the key points. The argument in each case begins with a coherence equation (as in (\ref{41}), (\ref{419})) and the nontriviality relation (as in (\ref{42}), (\ref{429})) which we aim to prove; each derives in a straightforward way from the argument and equations \ref{fcoherence} and \ref{showingthis} of Section \ref{theargument}. Similarly, explicit recursive definitions of the functions $\mathtt{f}_n$ (as in (\ref{43}), (\ref{439})) derive from equation \ref{effff}; support considerations then afford us ``degree reductions'' of the coherence and nontriviality relations at hand (as in (\ref{45}), (\ref{48}), (\ref{efftwo}), (\ref{e1})). More importantly, as a comparison of the explicit $\mathtt{f}_n$ and $\mathtt{f}_{n-1}$ definitions makes plain, for any $\gamma\in S_{n-1}^n$ the family $\{\mathtt{f}_n(\vec{\beta},\gamma)\mid\vec{\beta}\in [\gamma]^n\}$ is in essence a relativization to $C_\gamma$ of the family $\{\mathtt{f}_n(\vec{\beta})\mid\vec{\beta}\in [\omega_{n-1}]^n\}$.\footnote{We are eliding the matter of ``short'' internal tails, e.g. cases when coordinates of $\vec{\beta}$ do not fall within $C_\gamma$. Over suitably restricted domains, those arguments expand into $C_\gamma$, in the sense that after only finitely many ``corrective'' steps, their outputs agree with the $C_\gamma$-relativization we evoke. This is one meaning of the phrase ``in essence'' above. Alternatively, one may confine attention to the stricter relativizations $\{\mathtt{f}_n(0,\vec{\beta},\gamma)\mid\vec{\beta}\in [C_\gamma]^{n-1}\}$; the nontriviality of these families alone is sufficient for the remainder of our argument.} It will be our inductive assumption that the latter family, and hence the former, is nontrivial. The second argument for the nontriviality of $\mathtt{f}_2$ is then the simplest to generalize: for each $\gamma\in S_{n-1}^n$ the nontriviality of $\{\mathtt{f}_n(\vec{\beta},\gamma)\mid\vec{\beta}\in [\gamma]^n\}$ implies that there exists some $\vec{\alpha}_\gamma\in [C_\gamma]^n$ such that, just as in equation \ref{adapt},
$$\Bigg(\sum_{i=0}^{n-1}(-1)^i \mathtt{e}_{n-1}(0,\vec{\alpha}_\gamma^i)\Bigg)\Big|_{\vec{\alpha}_{\gamma}(0)\otimes [\omega_n]^n\otimes\{\gamma\}}\neq0.$$
(Here $\vec{\alpha}_{\gamma}(0)$ denotes the minimum element of $\vec{\alpha}_{\gamma}$.) Hence, just as before, by the Pressing Down Lemma there exists a stationary $S\subseteq S_{n-1}^n$ and $\vec{\alpha}\in [\omega_n]^n$ such that $\vec{\alpha}_\gamma=\vec{\alpha}$ for all $\gamma\in S$. However, also just as before, our triviality equations will imply that
\begin{align}
err(\vec{\alpha}):=\Big\{\gamma>\vec{\alpha}\;|\:\Big(\sum_{i=0}^{n-1}(-1)^i \mathtt{e}_{n-1}(0,\vec{\alpha}_\gamma^i)\Big)\Big|_{\vec{\alpha}_{\gamma}(0)\otimes [\omega_n]^n\otimes\{\gamma\}}\neq0\Big\} \text{ is finite for all }\vec{\alpha}\in[\omega_n]^n.
\label{err1}
\end{align}
This contradiction shows that no $\mathtt{e}_{n-1}$ can trivialize $\mathtt{f}_n$. As argued in Section \ref{theargument}, we have shown the following:
\begin{thm}\label{concludestheproof} For all $k\in\omega$ the cohomological dimension of $\omega_k$ is greater than or equal to $k+1$.
\end{thm}
As described, this theorem together with Corollary \ref{goblotsthm} then concludes the proof of Theorems \ref{mitchellstheorem},  \ref{omegan}, and  \ref{eight}.\\

In the remaining sections we foreground some of the more intriguing combinatorial phenomena manifesting in or by way of the functions $\mathtt{f}_n$ and the higher-order variants of familiar objects which they articulate.
\section{Trees of trees, cohomology, and higher coherence in various guises}\label{treessection}

\subsection{The cohomology of the ordinals}\label{71cohomology}
It will be useful henceforth to adopt a more systematic usage of the terms \emph{coherence} and \emph{triviality}, and of their order-$n$ instances. The following definition generalizes the \emph{mod finite} nontrivial coherence relations of Section \ref{walkssection}; here and in all subsequent definitions, $A$ denotes an arbitrary abelian group.

\begin{defin}[\cite{CoOI}]\label{highernontriv}
A function $\varphi:\varepsilon\to A$ is \emph{$0$-coherent} if $$\varphi\big|_{\beta}=^{*}0$$
or, in other words, if $\varphi\big|_{\beta}$ is finitely supported, for every $\beta<\varepsilon$. If $\varphi$ itself is finitely supported then it is \emph{$0$-trivial}. Observe that this definition of $0$-trivial readily applies to functions of more general domain as well.

A family of functions $\Phi_1=\{\varphi_\gamma:\gamma\to A\mid \gamma\in\varepsilon\}$ is \emph{$1$-coherent} if 
$$\varphi_\gamma\big|_{\beta}-\varphi_\beta=^{*} 0$$ for all $\beta<\gamma<\varepsilon$. The family $\Phi_1$ is \emph{$1$-trivial} if there exists a $\varphi:\varepsilon\to A$ such that $$\varphi\big|_{\beta}-\varphi_\beta=^{*}0$$
for all $\beta<\varepsilon$.

For $n>1$ a family of functions $\Phi_n=\{\varphi_{\vec{\beta}}:\beta_0\rightarrow A\,|\,\vec{\beta}\in [\varepsilon]^{n}\}$ is \emph{$n$-coherent} if
\begin{align*}
			\sum_{i=0}^{n} (\text{-}1)^i\varphi_{\vec{\alpha}^i}=^{*} 0
		\end{align*}
		for all $\vec{\alpha}\in [\varepsilon]^{n+1}$. (For readability, here and below we've suppressed restriction-notations, understanding equations to hold on the intersection of their constituent functions' domains). The family $\Phi_n$ is \emph{$n$-trivial} (or simply \emph{trivial}, when the $n$ is clear) if there exists a $\Psi_{n-1}=\{\psi_{\vec{\alpha}}:\alpha_0\rightarrow A\,|\,\vec{\alpha}\in [\varepsilon]^{n-1}\}$ 
				such that
				\begin{align*}
					\sum_{i=0}^{n-1} (\text{-}1)^i\psi_{\vec{\alpha}^i}=^{*} \varphi_{\vec{\alpha}}
				\end{align*}
				for all $\vec{\alpha}\in [\varepsilon]^{n}$.
				
				We call any of the aforementioned functions or families of functions \emph{$A$-valued}, and \emph{of height $\varepsilon$}.
\end{defin} 

Observe that the operation of pointwise addition determines group structures on the both the set $\mathsf{coh}(n,A,\varepsilon)$ of $n$-coherent $A$-valued height-$\varepsilon$ families of functions and the set $\mathsf{triv}(n,A,\varepsilon)$ of $n$-trivial $A$-valued height-$\varepsilon$ families of functions; observe moreover that the latter then forms a subgroup of the former. The following is shown in \cite[Theorem 2.30]{CoOI}:
\begin{thm}\label{cooicohtriv} Let $\check{\mathrm{H}}^n(\varepsilon;\mathcal{A})$ denote the $n^{\mathrm{th}}$ \v{C}ech cohomology group of the ordinal $\varepsilon$ (endowed with its usual order-topology) with respect to the sheaf $\mathcal{A}$ of locally constant functions to $A$. Then for all ordinals $\varepsilon$ and positive integers $n$ and abelian groups $A$, $$\check{\mathrm{H}}^n(\varepsilon;\mathcal{A})\cong \frac{\mathsf{coh}(n,A,\varepsilon)}{\mathsf{triv}(n,A,\varepsilon)}\,.$$
\end{thm}

In particular, $\check{\mathrm{H}}^n(\varepsilon;\mathcal{A})\neq 0$ if and only if there exists an $A$-valued nontrivial $n$-coherent family of functions of height $\varepsilon$. In \cite{CoOI} the following is shown as well:\footnote{The theorem and its argument differ only cosmetically from Goblot's \cite{Goblot}.}
\begin{thm} For any abelian group $A$ and positive integer $n$ and ordinal $\varepsilon$ of cofinality less than $\aleph_n$, we have $\check{\mathrm{H}}^n(\varepsilon;\mathcal{A})=0$.
\end{thm}

In \cite{CoOI} it is also asserted that there exist groups $A$ for which $\check{\mathrm{H}}^n(\omega_n;\mathcal{A})\neq 0$, and hence that $\omega_n$ is the least ordinal with nonvanishing constant-sheaf $\check{\mathrm{H}}^n$, but no proof is given. We now show that from the work of the previous section, this deduction is easy.

Recall first that a first step in the analysis of the functions $\mathtt{f}_n$ above was a reduction in ``coherence degree''; $\mathtt{f}_1$, for example, satisfies relations (\ref{41}) closest in form to the \emph{2}-coherence of Definition \ref{highernontriv} above, but is best regarded as a collection of $1$-coherent families of functions, as in equation \ref{45}. This nontrivial $1$-coherence permeates $\mathtt{f}_1$ ``in every direction''; for example, let
\begin{align*}
\varphi^x_\beta(\alpha):=\mathtt{f}_1(0,\beta)\big|_{\{\alpha\}\otimes[\omega_1]^2}
\end{align*}
for $\alpha<\beta<\omega_1$. By equation \ref{45} and the non-existence of an $\mathtt{e}_0$ as in (\ref{48}),
\begin{lem}\label{261}
$\{\varphi^x_\beta\,|\,\beta\in\omega_1\}$ is a nontrivial coherent family of functions.
\end{lem}
The reason for this is that any $\varphi$ trivializing this family naturally identifies with an $\hat{\mathtt{e}}_0(0)\in\prod_{[\omega_1]^3}\mathbb{Z}$ such that for all limit $\beta<\omega_1$ and all but a finite set $a(\beta)$ of $\alpha<\beta$, $$\hat{\mathtt{e}}_0(0)\big|_{\{\alpha\}\otimes [\omega_1]^2}=\mathtt{f}_1(0,\beta)\big|_{\{\alpha\}\otimes [\omega_1]^2}.$$
By the Pressing Down Lemma, there then exists a stationary $S\subseteq\omega_1$ and finite set $a$ such that $a(\beta)=a$ for all $\beta\in S$. Let $\gamma=\max a$ (do nothing if $a=\varnothing$) and modify $\hat{\mathtt{e}}_0(0)$ on $\{\{\alpha\}\otimes [\omega_1]^2\mid \alpha\in a\}$ to agree with $\mathtt{f}_1(0,\gamma+1)$; this defines an $\mathtt{e}_0(0)$ as in (\ref{48}), a contradiction.

The superscript ``$x$'' in Lemma \ref{261} indicates, of course, that $\varphi^x_\beta(\alpha)$ outputs the $x=\alpha$ ``slice'' of $\mathtt{f}_1(0,\beta)$.
Lemma \ref{261} holds equally for families $\{\varphi^y_\beta\,|\,\beta\in\omega_1\}$ and $\{\varphi^z_\beta\,|\,\beta\in\omega_1\}$, defined by
\begin{align*}
\varphi^y_\beta(\alpha):=&\, \mathtt{f}_1(0,\beta)\big|_{y=\alpha},\textnormal{ and} \\
\varphi^z_\beta(\alpha):=&\, \mathtt{f}_1(0,\beta)\big|_{z=\alpha},\textnormal{ respectively.}
\end{align*}
In the $x$ and $y$ cases, the codomain may be uniformly viewed as $\bigoplus_{[\omega_1]^2}\mathbb{Z}$, or, hence, as $\bigoplus_{\omega_1}\mathbb{Z}$, simply. The $z=\alpha$ slices may be construed as smaller, being bounded by $\alpha$
 --- but for small codomain we can do much better: fix a bijection $\theta:\omega_1\rightarrow [\,\omega_1]^3$. Then $E_\theta=\{\beta\,|\,\theta''\beta=[\beta]^3\}$ is a club subset of $\omega_1$; for $\beta\in E_\theta$, let
\begin{align}\label{51}
  \varphi^\theta_\beta(\alpha):= \left\{\def\arraystretch{1}%
  \begin{array}{@{}c@{\quad}l@{}}
    1 & \hspace{.3 cm}\textnormal{if }\theta(\alpha)\in\text{supp}(\mathtt{f}_1(0,\beta))\\
    0 & \hspace{.3 cm}\mathrm{otherwise}\\
  \end{array}\right.
\end{align}
Then as above, the following is straightforward to see:\footnote{The point is that although $\varphi^{\theta}_\beta$ records only a strictly ``initial tetrahedron'' of the output of $\mathtt{f}_1(0,\beta)$ for any $\beta\in E_\theta$, coherence ensures that any trivialization of $\Phi$ indeed translates, via $\theta$, to a trivialization of $\mathtt{f}_1$. Note that $\varphi^\theta_\beta$ does fully record that tetrahedron since the range of any $\mathtt{f}_1(0,\beta)$, viewed as a function $[\beta+1]^3\to\mathbb{Z}$, is $\{-1,0\}$, as is immediate from equations (\ref{43}) and (\ref{44}); in other words, there is no information loss in the passage from $\mathbb{Z}$ to $\mathbb{Z}/2\mathbb{Z}$ \emph{per se}.}
\begin{lem}\label{thetalemma} $\Phi:=\{\varphi^\theta_\beta:\beta\rightarrow\mathbb{Z}/2\mathbb{Z}\,|\,\beta\in E_\theta\}$ defines a nontrivial coherent family of functions.
\end{lem}
In particular, $\check{\mathrm{H}}^1(\omega_1;\mathcal{A})\neq 0$ for $A=\mathbb{Z}/2\mathbb{Z}$. For higher $n$, conversions patterned on (\ref{51}) deriving from bijections $\theta:\omega_n\to[\omega_n]^{n+2}$ omit more and more of the data of $\mathtt{f}_n$; consequently, though such functions will be $n$-coherent, they no longer so clearly inherit the nontriviality of $\mathtt{f}_n$. Just as in Lemma \ref{261}, however, families of functions recording first-coordinate slices of $\mathtt{f}_n$, $$\varphi^{\star}_{\vec{\beta}}:\alpha\mapsto\mathtt{f}_n(0,\vec{\beta})\big|_{\{\alpha\}\otimes [\omega_n]^{n+1}}\,,$$ do inherit the nontrivial coherence relations (as in (\ref{45}), (\ref{48}), (\ref{efftwo}), (\ref{e1})) of the functions $\mathtt{f}_n$. The verification of this fact, as well the natural identification of the codomain of the functions $\varphi^{\star}_{\vec{\beta}}$ with $\bigoplus_{[\omega_n]^{n+1}}\mathbb{Z}\cong\bigoplus_{\omega_n}\mathbb{Z}$, is straightforward and left to the reader (alternately, see \cite[Theorem 3.8.4]{dimords}). The following is then immediate:
\begin{thm}\label{cechtheorem} $\{\varphi^{\star}_{\vec{\beta}}\mid\vec{\beta}\in[\omega_n]^n\}$ is a nontrivial $n$-coherent family of functions. In particular, $\check{\mathrm{H}}^n(\omega_n;\mathcal{A})\neq 0$ for $A=\bigoplus_{\omega_n}\mathbb{Z}$. Hence $\omega_n$ is the least ordinal with a nontrivial \v{C}ech group $\check{\mathrm{H}}^n$ with respect to any constant sheaf $\mathcal{A}$.
\end{thm}
We return to the question of the \emph{integral} cohomology groups of the ordinals $\omega_n$ in our conclusion below. Here we remark simply that it is consistent with the $\mathsf{ZFC}$ axioms that $\check{\mathrm{H}}^n(\omega_n;\mathcal{A})\neq 0$ for $A=\mathbb{Z}$ and all $n\in\omega$ (this follows from $\mathsf{ZFC}+V=L$, for example; see \cite{CoOI}); the question of whether this is a $\mathsf{ZFC}$ \emph{theorem} seems a very good test of our understanding of the higher-dimensional combinatorics that form the present work's theme. Unsurprisingly, this question may also be phrased in terms of higher derived limits; it is in fact equivalent to the following:
\begin{quest} Let $\mathbf{Q}(\varepsilon)$ be the inverse system $(Q_\alpha,q_{\alpha\beta},\varepsilon)$ with $Q_\alpha=\bigoplus_\alpha\mathbb{Z}$ and $q_{\alpha\beta}:Q_\beta\to Q_\alpha$ the natural projection. Is it a $\mathsf{ZFC}$ theorem that $\mathrm{lim}^n\mathbf{Q}(\omega_n)\neq 0$ for each $n\in\omega$?
\end{quest}
This question in turn appears closely related to the sensitivity of the vanishing of $\lim^n \mathbf{A}$ to the \emph{dominating number} $\mathfrak{d}=\text{cf}(^\omega\omega,\leq)$ and its relation to the cardinal $\aleph_n$, where $\mathbf{A}$ is the inverse system featuring centrally in \cite{marpra, DSV, todder, SHDLST, SVHDL}. Indeed, this sensitivity was a main initial motivation for the present line of investigation.

\subsection{Trees of trees}\label{treessubsection}

In this section, a mild modification of Definition \ref{highernontriv} will facilitate description of $n$-dimensional generalizations of coherent Aronszajn trees, each of which makes its first $\mathsf{ZFC}$ appearance at $\omega_n$; classical coherent Aronszajn trees themselves comprise the $n=1$ case. The modification is simply to allow more general domains (still depending on $\gamma_0$) for the functions $\varphi_{\vec{\gamma}}$; as before, the comparison of such functions will always take place on the intersection of their domains.

For motivation, observe that the family $$\{\mathtt{f}_1(0,\beta)\big|_{[\alpha]^3}\mid\alpha\leq\beta+1<\omega_1\}$$
itself defines a coherent $\omega_1$-Aronszajn tree: simply view its elements, i.e.\ the nodes of any level $\alpha$ of the tree, as functions $[\alpha]^3\to\mathbb{Z}$, and order these nodes by inclusion. As a cofinal branch in this tree would render $\mathtt{f}_1$ trivial, this tree is Aronszajn.

On the left-hand side of Figure \ref{fig5} below is the more or less standard visualization of such a tree. On the right is a complementary visualization, one organized to foreground the essential mechanics of the nontrivial coherence of the system $\mathtt{f}_1(0,\,\cdot\,)$. As the analysis of Section \ref{thecasen1} made clear, those mechanics concentrate on the planes $\{z=\beta\mid\beta\in\text{Lim}\cap\omega_1\}$; more precisely, they concentrate on distinguished copies of $C_\beta$ therein. Schematically, then, we might view $\{\mathtt{f}_1(0,\gamma)\big|_{z=\beta}\mid\beta\leq\gamma\}$ as a family $s^1_\gamma$ of $0$-coherent functions $\{s^1_\gamma(\,\cdot\,,\beta):\beta\to\mathbb{Z}\mid\beta\leq\gamma\}$ which are non-$0$-trivial on the limit ordinals $\beta\leq\gamma$. Any of several approaches might effect this sort of identification; in perhaps the simplest, 
\[
  s^1_\gamma(\alpha,\beta) =
  \begin{cases}
                                   1 & \text{if $\beta$ is a limit ordinal and $\alpha\in  C_\beta\backslash m(\beta,\gamma)$} \\
                                   1 & \text{if $\beta$ is a successor ordinal and $\beta=\alpha+1$} \\
                                   0 & \text{otherwise}
  \end{cases}
\]
for each $\alpha<\beta$; here $m(\beta,\gamma)$ is the function defined in (\ref{mfunction}) above. This abstraction of $\mathtt{f}_1(0,\gamma)$, of course, is essentially that of Figure \ref{yslices} reflected through the graph of $z=x$. 
The point of all this is simply the following: $\{s^1_\gamma\mid\gamma<\omega_1\}$ is a natural recasting of the family $\{\mathtt{f}_1(0,\gamma)\mid\gamma<\omega_1\}$ in which:
\begin{enumerate}
\item Each $\{s_\gamma^1(\,\cdot\,,\beta)\mid\beta\leq\gamma\}$ is a family of $0$-coherent functions, each of which is non-$0$-trivial when $\beta$ is a limit.
\item For any $\gamma\leq\delta<\omega_1$ the functions $s^1_\delta\big|_{[\gamma+1]^2}$ and $s^1_\gamma$ differ by a $0$-trivial function.
\item There exists no $t^1:[\omega_1]^2\to\mathbb{Z}$ such that for all $\gamma<\omega_1$ the functions $t^1\big|_{[\gamma+1]^2}$ and $s^1_\gamma$ differ by a $0$-trivial function.
\end{enumerate}
Just as above, $\{s^1_\gamma\mid\gamma<\omega_1\}$ (together with its functions' restrictions to $[\beta]^2$ $(\beta<\omega_1)$) is readily identified with an $\omega_1$-Aronszajn tree.

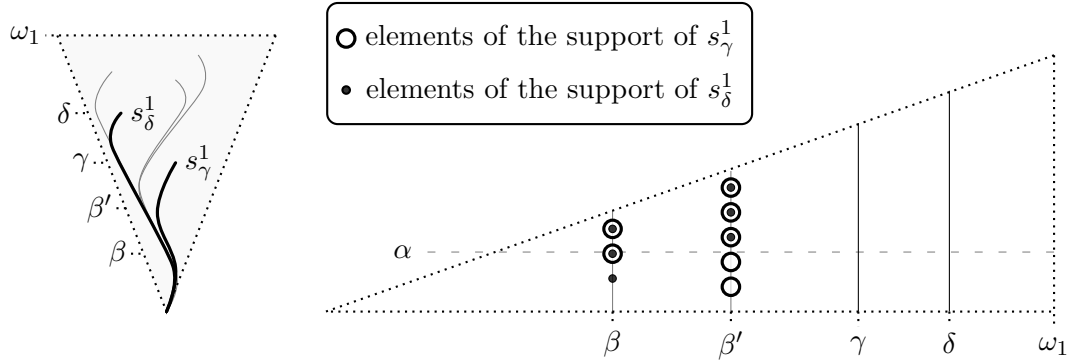
\begin{figure}
\begin{tikzpicture}[MyPersp,font=\large]
	\coordinate (A) at (1.7,0,0);
	\coordinate (B) at (2.9,0,3.34);
	\coordinate (C) at (0.5,0,3.34);
	\coordinate (C') at (0.37,0,3.34);
	\coordinate (D) at (1.8,0,1.8);	
	\coordinate (E) at (1.2,0,2.4);	
	\coordinate (Y) at (11.45,0,.72);	
	\coordinate (Z) at (4.5,0,.72);	
	
	\coordinate (AA) at (3.45,0,0);
	\coordinate (BB) at (11.45,0,3.1);
	\coordinate (CC) at (11.45,0,0);
	\coordinate (DD) at (7.9,0,0);
	\coordinate (EE) at (7.9,0,1.7);
	\coordinate (FF) at (9.3,0,0);
	\coordinate (GG) at (9.3,0,2.28);
	\coordinate (HH) at (10.3,0,0);
	\coordinate (II) at (10.3,0,2.67);
	\coordinate (JJ) at (6.6,0,0);
	\coordinate (KK) at (6.6,0,1.22);
	
	\coordinate (D1) at (6.6,0,.4);
	\coordinate (D2) at (6.6,0,.7);	
	\coordinate (D3) at (6.6,0,1);	
	\coordinate (C1) at (7.9,0,.3);
	\coordinate (C2) at (7.9,0,.6);
	\coordinate (C3) at (7.9,0,.9);
	\coordinate (C4) at (7.9,0,1.2);
	\coordinate (C5) at (7.9,0,1.5);

	\coordinate (DD-) at (7.9,0,-.12);

	\coordinate (FF-) at (9.3,0,-.18);
	
	\coordinate (JJ-) at (6.6,0,-.12);

	\coordinate (HH-) at (10.3,0,-.14);
	
	\coordinate (F-) at (11.45,0,-.2);
	
	\coordinate (g1) at (.9,0,1.8);
	\coordinate (g2) at (1.05,0,1.8);
	
	\coordinate (d1) at (.7,0,2.4);
	\coordinate (d2) at (.85,0,2.4);
	
	\coordinate (b1) at (1.31,0,.7);
	\coordinate (b2) at (1.46,0,.7);

	\coordinate (b3) at (1.15,0,1.25);
	\coordinate (b4) at (1.3,0,1.25);


	
\fill[fill=gray, opacity=.045] (A)--(B)--(C);
\draw[gray, line cap=round] plot [smooth, tension=.35] coordinates { (1.7,0,0) (1.8,0,.5) (.93,0,2.4) (1.1,0,2.9)};
\draw[gray, line cap=round] plot [smooth, tension=.5] coordinates { (1.7,0,0) (1.8,0,.5) (1.4,0,1.55) (1.9,0,2.5) (1.8,0,2.8)};
\draw[gray, line cap=round] plot [smooth, tension=.5] coordinates { (1.7,0,0) (1.8,0,.5) (1.4,0,1.55) (2.1,0,2.7) (2.05,0,3.1)};
\draw[very thick, line cap=round] plot [smooth, tension=.7] coordinates { (1.7,0,0) (1.8,0,.53) (1.6,0,1.2) (1.8,0,1.8)};
\draw[very thick, line cap=round] plot [smooth, tension=.43] coordinates { (1.7,0,0) (1.77,0,.55) (1.1,0,2) (1.2,0,2.4)};

\draw[loosely dashed, gray] (Y)--(Z);
\draw (Z) node[left] {$\alpha$};
\draw[dotted, thick] (B)--(A)--(C);
\draw[dotted, thick] (B)--(C');
\draw[dotted, thick] (JJ)--(JJ-) node[below] {$\beta$};
\draw[dotted, thick] (g1)--(g2) node[left] {$\gamma\;$};
\draw[dotted, thick] (b1)--(b2) node[left] {$\beta\;$};
\draw[dotted, thick] (b3)--(b4) node[left] {$\beta'\;$};

\draw[dotted, thick] (d1)--(d2) node[left] {$\delta\;$};
\draw[dotted, thick] (FF)--(FF-) node[below] {$\gamma$};
\draw[dotted, thick] (DD)--(DD-) node[below] {$\beta'$};
\draw[dotted, thick] (HH)--(HH-) node[below] {$\delta$};
\draw[gray] (DD)--(EE);
\draw (FF)--(GG);
\draw (HH)--(II);
\draw (C') node[left] {$\omega_1\!$};
\draw (D) node[right] {$s_\gamma^1$};
\draw (E) node[right] {$s_\delta^1$};
\draw[gray] (JJ)--(KK);
\draw (D2) node[circle, very thick,draw,fill=white,fill opacity=.8, scale=.7]{};
\draw (D3) node[circle, very thick,draw,fill=white,fill opacity=.8, scale=.7]{};
\draw (D1) node[circle,draw,fill=black,fill opacity=.8,scale=.3]{};
\draw (D2) node[circle,draw,fill=black,fill opacity=.8,scale=.3]{};
\draw (D3) node[circle,draw,fill=black,fill opacity=.8,scale=.3]{};
\draw (C1) node[circle, very thick,draw,fill=white,fill opacity=.6, scale=.7]{};
\draw (C2) node[circle, very thick,draw,fill=white,fill opacity=.8, scale=.7]{};
\draw (C3) node[circle, very thick,draw,fill=white,fill opacity=.8, scale=.7]{};
\draw (C4) node[circle, very thick,draw,fill=white,fill opacity=.8, scale=.7]{};
\draw (C5) node[circle, very thick,draw,fill=white,fill opacity=.8, scale=.7]{};
\draw (C3) node[circle,draw,fill=black,fill opacity=.8,scale=.3]{};
\draw (C4) node[circle,draw,fill=black,fill opacity=.8,scale=.3]{};
\draw (C5) node[circle,draw,fill=black,fill opacity=.8,scale=.3]{};

\draw[dotted, thick] (BB)--(AA)--(CC);
\draw[dotted, thick] (BB)--(F-) node[below] {$\omega_1$};
	
	     \node [matrix,draw=black,rounded corners, row sep=0mm, thick] at (5.8,0,3){\draw (0.2,0) node[circle, very thick,draw,fill=white,fill opacity=.8, scale=.7]{}; & \node {elements of the support of $s_\gamma^1$}; \\ \draw (0.2,0) node[circle,draw,fill=black,fill opacity=.8,scale=.3]{}; & \node {elements of the support of $s_\delta^1$}; \\};
\end{tikzpicture}
\caption{Two visualizations of the coherent Aronszajn tree $T$ deriving from $\mathtt{f}_1$. On the left is the standard view: $T$ is defined by its branches $s_\gamma^1$  ($\gamma<\omega_1$), i.e., by height-$\gamma$ functions which encode the data of $\mathtt{f}_1(0,\gamma)$. On the right is an alternative visualization: here $s_\gamma^1$ is pictured as the wedge to the left of the vertical line at $\gamma$. At each $\beta\leq\gamma$, the function $s_\gamma^1$ outputs a $0$-coherent function, namely the characteristic function of a tail of $C_\beta$. Moreover, any $s_\gamma^1$ and $s_\delta^1$ disagree on only finitely many columns in their common domain (and disagree only finitely often thereupon as well), a relationship to the functions $s_\gamma^1$ which no length-$\omega_1$ $t^1$ can globally replicate.}\label{fig5}
\end{figure}

For the purposes of generalization, it is convenient to redefine the functions $s^1_\gamma$ as one-coordinate functions taking any $\beta\leq\gamma$ to $\mathtt{f}_1(0,\gamma)\big|_{[\beta]^2\otimes\{\beta\}}$; this evidently preserves, in principle, points (1) through (3) of the prior definition. We then define a family of functions $\{s^2_\delta\mid\delta<\omega_2\}$, each with domain $(\delta+1)\times\delta$, by $$s^2_\delta(\beta,\gamma)=\mathtt{f}_2(0,\gamma,\delta)\big|_{\min\{\beta,\gamma\}\otimes[\beta]^2\otimes\{\beta\}}.$$
More generally, we define families of functions $\{s^n_\delta\mid\delta<\omega_n\}$, each with domain $(\delta+1)\times[\delta]^{n-1}$, by $$s^n_\delta(\beta,\vec{\gamma})=\mathtt{f}_n(0,\vec{\gamma},\delta)\big|_{\min\{\beta,\gamma_0\}\otimes[\beta]^n\otimes\{\beta\}}.$$
In other words, $s^n_\delta(\beta,\vec{\gamma})$ records the $\gamma_0$-restriction of the $z=\beta$ hyperplane of $\mathtt{f}_n(0,\vec{\gamma},\delta)$. In consequence, each function $s^n_\delta(\,\cdot\,,\vec{\gamma})$ may be identified with a function $\varphi^{\delta}_{\vec{\gamma}}$ which maps the union of these hyperplanes, namely $\gamma_0\otimes [\delta+1]^{n+1}$, to the integers. As indicated in this section's introduction, the notions of Definition \ref{highernontriv} extend in a straightforward manner to families of functions, like these, with non-ordinal domains; $n$-coherence and $n$-triviality are, as before, questions of \emph{mod finite} agreement on the intersection of the associated domains. We may now describe how for $n>1$ the families $\{s^n_\delta\mid\delta<\omega_n\}$ satisfy the following higher-dimensional analogues of points (1) through (3) above:
\begin{enumerate}
\item[(1')] Each $\{s_\delta^n(\beta,\vec{\gamma})\mid\vec{\gamma}\in[\beta]^{n-1}\}$ is an $(n-1)$-coherent family of functions which is nontrivial if $\text{cf}(\beta)=\aleph_{n-1}$.
\item[(2')] For any $\gamma\leq\delta<\omega_n$ the functions $s^n_\delta\big|_{(\gamma+1)\times[\gamma]^{n-1}}$ and $s^n_\gamma$ differ by an $(n-1)$-trivial function. More precisely, the families
$$\{\varphi^\delta_{\vec{\alpha}}-\varphi^\gamma_{\vec{\alpha}}:\alpha_0\otimes [\gamma+1]^{n+1}\to\mathbb{Z}\mid\vec{\alpha}\in [\gamma]^{n-1}\}$$
defined above are $(n-1)$-trivial for all $\gamma\leq\delta<\omega_n$.
\item[(3')] There exists no $t^n$ such that for all $\gamma<\omega_n$ the functions $t^n\big|_{(\gamma+1)\times[\gamma]^{n-1}}$ and $s^1_\gamma$ differ by an $(n-1)$-trivial function. Put differently, there exists no $t^n=s^n_{\omega_n}$ such that for all $\gamma<\omega_n$ the family
$$\{\varphi^{\omega_n}_{\vec{\alpha}}-\varphi^\gamma_{\vec{\alpha}}:\alpha_0\otimes [\gamma+1]^{n+1}\to\mathbb{Z}\mid\vec{\alpha}\in [\gamma]^{n-1}\}$$
is nontrivial, where $s^n_{\omega_n}$ induces the function $\varphi^{\omega_n}_{\vec{\alpha}}$ in the manner described in the paragraph preceding (1').
\end{enumerate}
Cumbersome as the notations do grow, the idea of these families is in fact very simple, as Figure \ref{fig6} is meant to convey. There and below, we focus on the case of $n=2$.

\begin{figure}
\begin{tikzpicture}[MyPersp,font=\large]
	\coordinate (A) at (4.7,0,0);
	\coordinate (B) at (3.8,0,2.05);
	\coordinate (C) at (5.6,0,2.05);
	\coordinate (D) at (4.7,0,2.05);
	\coordinate (F) at (7,0,0);	
	\coordinate (G) at (5.8,0,3.1);	
	\coordinate (H) at (8.2,0,3.1);
	\coordinate (E) at (9.3,0,0);	
	\coordinate (I) at (9.3,0,4.07);	
	\coordinate (J) at (10.3,0,0);
	\coordinate (K) at (10.3,0,4.5);
	\coordinate (L) at (2.1,0,1.5);
	\coordinate (M) at (11.45,0,1.5);
	\coordinate (N) at (2.75,0,0);
	\coordinate (O) at (5.85,0,.65);
	\coordinate (P) at (5.85,0,0);
	\coordinate (Q) at (4.42,0,.9);	
	\coordinate (R) at (3.33,0,.32);
	\coordinate (S) at (5.26,0,.96);
	\coordinate (T) at (4.92,0,1.2);	
	\coordinate (U) at (6.46,0,.32);
	\coordinate (V) at (6.9,0,.5);	
	\coordinate (W) at (8.3,0,.4);
	\coordinate (X) at (7.67,0,.7);	
	\coordinate (Y) at (7.2,0,.82);
	
	\coordinate (AA) at (0,0,0);
	\coordinate (BB) at (11.45,0,5);
	\coordinate (CC) at (11.45,0,0);

	\coordinate (A-) at (4.7,0,-.12);
	\coordinate (F-) at (7,0,-.12);	
	\coordinate (E-) at (9.3,0,-.18);	
	\coordinate (J-) at (10.3,0,-.12);	
	\coordinate (CC-) at (11.45,0,-.18);		

	\coordinate (W-) at (3.3,0,-.1);
	\coordinate (W+) at (3.3,0,.1);
	
	\coordinate (BB-) at (5.9,0,2);
	\coordinate (BB+) at (6.1,0,2);
	
	\coordinate (EE-) at (5.9,0,3);
	\coordinate (EE+) at (6.1,0,3);

	\coordinate (F0-) at (6,2,-.1);
	\coordinate (F0+) at (6,2,.1);
	
	\coordinate (F1-) at (8,0,.1);
	\coordinate (F1+) at (8,0,-.1);
	
	\coordinate (F2-) at (6,3,-.1);
	\coordinate (F2+) at (6,3,.1);
	
	\coordinate (F3-) at (9,0,.1);
	\coordinate (F3+) at (9,0,-.1);


	
\fill[fill=gray, opacity=.17] (A)--(B)--(C);
\fill[fill=gray, opacity=.17] (F)--(G)--(H);
\draw[dashed] (M)--(L) node[left] {$\alpha$};
\draw[line width=1.2 mm, white, line cap=round] plot [smooth, tension=.6] coordinates { (4.7,0,0) (4.75,0,.3) (4.5,0,.8) (4.45,0,1.2) (4.6,0,1.5)};
\draw[line width=.5 mm, white, line cap=round, opacity=.6] plot [smooth, tension=.6] coordinates { (4.7,0,0) (4.75,0,.3) (4.45,0,1) (4.35,0,1.5) (4.5,0,1.85)};
\draw[line width=.5 mm, white, line cap=round, opacity=.6] plot [smooth, tension=.6] coordinates { (4.7,0,0) (4.75,0,.3) (4.55,0,.7) (4.52,0,.9) (4.65,0,1.2)};
\draw[line width=.5 mm, black, opacity=.3, line cap=round] plot [smooth, tension=.6] coordinates { (4.7,0,0) (4.75,0,.3) (4.7,0,.5) (4.85,0,.8) (4.95,0,1)};
\draw[line width=.5 mm, black, opacity=.3, line cap=round] plot [smooth, tension=.6] coordinates { (4.7,0,0) (4.75,0,.3) (4.65,0,.7) (4.85,0,1.2) (4.78,0,1.4)};
\draw[line width=1.2 mm, black, opacity=.9, line cap=round] plot [smooth, tension=.6] coordinates { (4.7,0,0) (4.75,0,.3) (4.65,0,.7) (4.8,0,1.1) (5,0,1.5)};
\draw[line width=1.2 mm, white, line cap=round] plot [smooth, tension=.6] coordinates { (7,0,0) (7,0,.3) (6.8,0,.8) (6.8,0,1.2) (6.6,0,1.5)};
\draw[line width=.5 mm, white, line cap=round, opacity=.6] plot [smooth, tension=.6] coordinates { (7,0,0) (7,0,.3) (6.8,0,.8) (6.8,0,1.2) (6.8,0,2) (6.65,0,2.3)};
\draw[line width=.5 mm, white, line cap=round, opacity=.6] plot [smooth, tension=.6] coordinates { (7,0,0) (7,0,.3) (6.8,0,.8) (6.8,0,1.2) (6.8,0,1.7) (7,0,2.3) (6.8,0,2.6)};
\draw[line width=.5 mm, black, opacity=.3, line cap=round] plot [smooth, tension=.8] coordinates { (7,0,0) (7.05,0,.3) (7.15,0,.6) (7.15,0,1) (7,0,1.5) (7.3,0,2.1) (7.4,0,2.5) (7.2,0,2.8)};
\draw[line width=.5 mm, black, opacity=.3, line cap=round] plot [smooth, tension=.7] coordinates { (7,0,0) (7.05,0,.3) (7.15,0,.6) (7.15,0,1) (7.1,0,1.3) (7.5,0,2.1)};
\draw[line width=1.2 mm, black, opacity=.9, line cap=round] plot [smooth, tension=.6] coordinates { (7,0,0) (7.05,0,.3) (7.15,0,.6) (7.15,0,1) (7.4,0,1.5)};
\draw[dotted, thick] (BB)--(AA)--(CC)--cycle;
\draw[densely dotted, thick] (A)--(B)--(C)--cycle;
\draw[densely dotted, thick] (F)--(G)--(H)--cycle;

\draw (E)--(I);
\draw (J)--(K);
\draw (N) node[above] {$s^2_\gamma(\beta,\alpha)$};
\draw (O) node[above] {$s^2_\delta(\beta,\alpha)$};
\draw (P) node[above] {$s^2_\gamma(\beta',\alpha)$};
\draw (W) node[above] {$s^2_\delta(\beta',\alpha)$};
\draw[->] (R)--(Q);
\draw[->] (S)--(T);
\draw[->] (U)--(V);
\draw[->] (X)--(Y);
\draw[dotted, thick] (A)--(A-) node[below] {$\beta$};
\draw[dotted, thick] (F)--(F-) node[below] {$\beta'$};
\draw[dotted, thick] (E)--(E-) node[below] {$\gamma$};
\draw[dotted, thick] (J)--(J-) node[below] {$\delta$};
\draw[dotted, thick] (BB)--(CC-) node[below] {$\omega_2$};
\end{tikzpicture}
\caption{The $n=2$ generalization of Figure \ref{fig5}, a ``tree of trees.'' This tree's branches $s_\gamma^2$  ($\gamma<\omega_2$) are height-$\gamma$ functions which encode the data of $\mathtt{f}_2(0,\,\cdot\,,\gamma)$. In analogy with Figure \ref{fig5}, $s_\gamma^2$ is loosely identified with the wedge to the left of the vertical line at $\gamma$. At each $\beta\leq\gamma$ the family $\{s_\gamma^2(\beta,\alpha)\mid\alpha<\beta\}$ forms a $1$-coherent family of functions which is nontrivial if $\text{cf}(\beta)=\aleph_1$; its elements are naturally viewed as branches of a tree, as above. Moreover, for any $\gamma<\delta<\omega_2$ the family of functions $\{s^2_\delta(\,\cdot\,,\alpha)-s^2_\gamma(\,\cdot\,,\alpha)\mid\alpha<\gamma\}$ is $1$-trivial, i.e., is equal (mod finite) to the restrictions of a single function; as in Figure \ref{fig5}, this is a relationship with the functions $s^2_\gamma$ which no length-$\omega_2$ function $t^2$ can globally achieve.}
\label{fig6}
\end{figure}
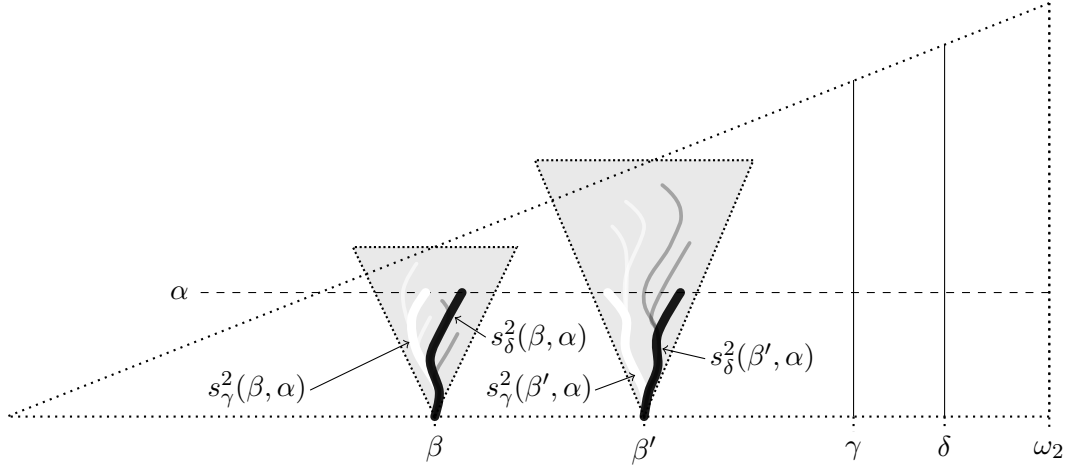
In this case, item (1') translates to the assertion that, viewed as a family of functions, \begin{align}\label{eq1}\{\mathtt{f}_2(0,\gamma,\delta)\big|_{\gamma\otimes [\beta]^2\otimes\{\beta\}}\mid\gamma<\beta\}\end{align} is $1$-coherent, and is nontrivial if $\text{cf}(\beta)=\aleph_1$. This, though, follows immediately from the observation in Section \ref{efffn} that \begin{align}\label{eq2} \{\mathtt{f}_2(0,\gamma,\beta)\big|_{\gamma\otimes [\beta]^2\otimes\{\beta\}}\mid\gamma<\beta\}\end{align} is a $1$-coherent family which is nontrivial if $\text{cf}(\beta)=\aleph_1$, together with equation \ref{efftwo} therein ---
\begin{align*}(\mathtt{f}_2(0,\beta,\delta)-\mathtt{f}_2(0,\gamma,\delta)+\mathtt{f}_2(0,\gamma,\beta))\big|_{\gamma\otimes[\omega_2]^3}=^*0\textnormal{ for all }\gamma<\beta\end{align*}
--- so that the family (\ref{eq1}) is simply a ``shift'' by $\mathtt{f}_2(0,\beta,\delta)\big|_{z=\beta}$ of the family (\ref{eq2}).

Item (2') amounts to the assertion that each $\{s^2_\delta(\,\cdot\,,\alpha)-s^2_\gamma(\,\cdot\,,\alpha)\mid\alpha<\gamma\}$, which may be identified with $\{(\mathtt{f}_2(0,\alpha,\delta)-\mathtt{f}_2(0,\alpha,\gamma))\big|_{\alpha\otimes [\gamma+1]^3}\mid\alpha<\gamma\}$, is a $1$-trivial family, or in other words is approximated (mod finite) by the restrictions of a single function. This too is immediate from a version of equation \ref{efftwo}:
$$(\mathtt{f}_2(0,\alpha,\delta)-\mathtt{f}_2(0,\alpha,\gamma))\big|_{\alpha\otimes [\gamma+1]^3}=^{*}\mathtt{f}_2(0,\gamma,\delta)\big|_{\alpha\otimes [\gamma+1]^3}\text{ for all }\alpha<\gamma.$$

Item (3') generally follows from the nontriviality of the functions $\mathtt{f}_n$; unlike items (1') or (2'), however, the work of arguing this substantially rises with the parameter $n$. In the case of $n=2$, we begin by assuming for contradiction that there exists a $t^2$ as described, i.e., such that for each $\gamma$ some $\psi_\gamma:[\gamma+1]^4\to\mathbb{Z}$ $1$-trivializes $\{\varphi^{\omega_2}_\alpha-\varphi^\gamma_\alpha\mid\alpha<\gamma\}$. Then the functions $\mathtt{e}_1(0,\gamma)$ defined for all successor $\gamma$ by
\[
  \mathtt{e}_1(0,\gamma)\big|_{\min\{\beta,\gamma\}\otimes[\beta]^2\otimes\{\beta\}} =
  \begin{cases}
                                   -\psi_\gamma\big|_{[\beta]^3\otimes\{\beta\}} & \text{if $\beta\leq\gamma$} \\
                                   -t^2(\beta,\gamma)=-\varphi^{\omega_2}_\gamma\big|_{\gamma\otimes [\beta]^2\otimes\{\beta\}} & \text{if $\beta\in(\gamma,\omega_2)$} 
  \end{cases}
\]
together satisfy equation \ref{e1} of Section \ref{efffn}:
\begin{align}\label{e1x2}
(\mathtt{e}_1(0,\gamma)-\mathtt{e}_1(0,\alpha))\big|_{\alpha\otimes[\omega_2]^3}=^{*}\mathtt{f}_2(0,\alpha,\gamma)\big|_{\alpha\otimes[\omega_2]^3}\hspace{.3 cm}\textnormal{ for all successor }\alpha<\gamma<\omega_2\,.
\end{align}
It is then straightforward to extend these assignments to full trivialization $\mathtt{e}_1$ of $\mathtt{f}_2$, giving the contradiction we had desired. Equation \ref{e1x2} is verified in three steps, which may be thought of as a left-to-right movement below the horizontal line at $\alpha$ in Figure \ref{fig6}. First fix successor ordinals $\alpha<\gamma<\omega_2$ and observe that for all $\xi<\alpha$,
\begin{align*} (-\psi_\gamma + \psi_\alpha)\big|_{\xi\otimes[\alpha+1]^3} & =\;\, (\varphi^{\omega_2}_\xi-\psi_\gamma -\varphi^{\omega_2}_\xi+ \psi_\alpha)\big|_{\xi\otimes [\alpha+1]^3} \\
& =^{*} (\varphi^\gamma_\xi-\varphi^\alpha_\xi)\big|_{\xi\otimes [\alpha+1]^3} \\ & =\;\, (\mathtt{f}_2(0,\xi,\gamma)-\mathtt{f}_2(0,\xi,\alpha))\big|_{\xi\otimes [\alpha+1]^3}\\ & =^{*} (\mathtt{f}_2(0,\alpha,\gamma))\big|_{\xi\otimes [\alpha+1]^3}
\end{align*}
This shows that (\ref{e1x2}) holds when restricted to $[\alpha+1]^4$. We show that it holds on the rest of $\alpha\otimes [\omega_2]^3$ by the following sequence of steps. On $\alpha\otimes[\gamma]^2\otimes (\alpha,\gamma]$, the difference (\ref{e1x2}) amounts to $-\psi_\gamma+\varphi^{\omega_2}_\alpha$, which equals $\varphi^\gamma_\alpha$ (mod finite) by definition, which via its identification with $\mathtt{f}_2(0,\alpha,\gamma)$ then extends the restriction on which we have seen (\ref{e1x2}) to hold to $\alpha\otimes[\gamma+1]^3$. For any $\delta\in(\gamma,\omega_2)$, the restriction of the left-hand side of (\ref{e1x2}) to $\alpha\otimes[\delta]^2\otimes (\gamma,\delta]$ is $\varphi^{\omega_2}_\alpha-\varphi^{\omega_2}_\gamma$, which equals (mod finite) $\varphi^{\delta}_\alpha-\varphi^{\delta}_\gamma$, which is finitely supported, by these functions' identifications with $\mathtt{f}_2$ together with equations \ref{419} and \ref{449}. As $\delta$ is arbitrary, this shows that equation \ref{e1x2} does in fact hold on $\alpha\otimes[\omega_2]^3$, concluding the argument.

As mentioned, the argument of the higher-$n$ cases involves no conceptual novelty, but does entail rising notational costs (again functions $\mathtt{e}_{n-1}$ trivializing $\mathtt{f}_n$ are defined from the trivializations $t_n$, but in more ``pieces'' together involving a longer series of derived trivializations $u$). In its place we describe a cleaner mild variation of the ``trees of trees'' structures above and show that nontrivial instances of this variation coincide with instances of the higher-order nontrivial coherence of Definition \ref{highernontriv}. Our argument of this fact is close enough in spirit to the one we omit that little content is ultimately lost.

We turn now to the definition of this variation. Though our terminology will overlap with that of Definition \ref{highernontriv} above, context will generally indicate which of the two notions of \emph{$n$-coherence} we have in mind; in cases of potential ambiguity, we denote the notions in Definitions \ref{highernontriv} and \ref{highernontrivII} as  \emph{$n$-coherence}$^{\mathrm{I}}$ and \emph{$n$-coherence}$^{\mathrm{II}}$, respectively (this overlap is also somewhat justified by Theorem \ref{twoiffone}). It will streamline discussion, also, to tacitly identify families of  functions $\{s^n_\gamma:[\gamma]^{n}\to A\mid\gamma<\kappa\}$ with single functions $s^n:[\kappa]^{n+1}\to A$ via the equation $s^n(\vec{\alpha},\gamma)=s^n_\gamma(\vec{\alpha})$. Such an identification is operative at each step of the following inductive definition:
\begin{defin}\label{highernontrivII} A function is \emph{0-trivial} if its support is finite.

For any $n>0$ a family of functions $\{s^n_\gamma:[\gamma]^{n}\to A\mid\gamma<\varepsilon\}$ is \emph{$n$-coherent} if
\begin{align*}s^n_\delta\big|_{[\gamma]^{n}}-s^n_\gamma\end{align*}
is $(n-1)$-trivial for all $\gamma\leq\delta<\varepsilon$.

Such a family is \emph{$n$-trivial} if there exists a $t^n:[\varepsilon]^{n}\to A$ such that
\begin{align*}t^n\big|_{[\gamma]^{n}}-s^n_\gamma\end{align*}
is $(n-1)$-trivial for all $\gamma<\varepsilon$.
\end{defin}

Observe that any $n$-coherent family $\mathcal{S}_n=\{s_\gamma^n\mid\gamma<\varepsilon\}$ induces a tree \begin{align}\label{treeform}T(\mathcal{S}_n):=\big(\{s_\gamma^n\big|_{[\beta]^{n}}\mid\beta\leq\gamma<\varepsilon\},\,\subseteq\big).\end{align} Observe also that if $\mathcal{S}_n$ is non-$n$-trivial then $T(\mathcal{S}_n)$ has no cofinal branch. In fact, $\mathcal{S}_n$ is non-$n$-trivial if and only if the \emph{uniform closure} $T^*(\mathcal{S}_n)$ of $T(\mathcal{S}_n)$ has no cofinal branch; this notion generalizes that of \cite[Definition 4.1.2]{todwalks}:
\begin{defin} An \emph{$n$-coherent tree} $T(\mathcal{S}_n)$ is one of the form (\ref{treeform}) for some $n$-coherent family of functions $\mathcal{S}_n=\{s^n_\gamma:[\gamma]^{n}\to A\mid\gamma<\varepsilon\}$. The \emph{uniform closure} of $T(\mathcal{S}_n)$ is the tree
$$T^*(\mathcal{S}_n)=(\bigcup_{\gamma<\varepsilon}\{u_\gamma:[\gamma]^n\to A\mid u_\gamma-s^n_\gamma\text{ is $(n-1)$-trivial}\},\subseteq).$$
An $n$-coherent tree $T(\mathcal{S}_n)$ is \emph{nontrivial} if its uniform closure $T^*(\mathcal{S}_n)$ has no cofinal branch.
\end{defin}
The ways or degrees to which $n$-coherence materializes at the level of the tree-structures of the trees $T(\mathcal{S}_n)$ (as it very consequentially does in the case of $n=1$; see \cite[Chapter 4]{todwalks}) remains an interesting question.

As hinted above, the terminological overlap of Definitions 7.8 and 7.1 reflects a degree of equivalence between the notions they describe. This is a relationship made precise by Theorem \ref{twoiffone}. Although at first glance, particularly for low $n$, this relationship may appear to consist in little more than notational rearrangements (as is indeed the case in one direction; see Lemma \ref{rerangement}), the theorem in its full generality is quite subtle, and its proof is sufficiently tedious that we defer it to our appendix (Section \ref{appendixB}). 
\begin{thm} \label{twoiffone} There exists a height-$\varepsilon$ $A$-valued nontrivial $n$-coherent$^{\mathrm{I}}$ family of functions if and only if there exists a height-$\varepsilon$ $A$-valued nontrivial $n$-coherent$^{\mathrm{II}}$ family of functions.
\end{thm}
In Appendix \ref{appendixB}, we show something slightly stronger than this, namely that for every $n>0$ and ordinal $\varepsilon$ and abelian group $A$,
$$\frac{\mathsf{coh}^{\mathrm{I}}(n,A,\varepsilon)}{\mathsf{triv}^{\mathrm{I}}(n,A,\varepsilon)}\cong\frac{\mathsf{coh}^{\mathrm{II}}(n,A,\varepsilon)}{\mathsf{triv}^{\mathrm{II}}(n,A,\varepsilon)}.$$
These quotients are the obvious modifications of that of Theorem \ref{cooicohtriv} in light of Definition \ref{highernontrivII}. From Theorems \ref{twoiffone} and \ref{cechtheorem} the following is immediate:
\begin{cor}\label{maybemylastlabel} For all $n>0$ the least ordinal $\varepsilon$ for which there exists a nontrivial $n$-coherent tree of height $\varepsilon$ is $\omega_n$.
\end{cor}

\section{The trace functions at the heart of the functions $\mathtt{f}_n$}\label{highertraces}

\subsection{The case of $n=1$}\label{thecasen=1again}
An organizing question in Sections \ref{thecasen1} and \ref{treessubsection} was that of the behavior of the function $\mathtt{f}_1$ on the planes $z=\beta$ for limit ordinals $\beta$. For any such $\beta$ and $\alpha<\beta<\gamma$ this question amounts essentially to that of $\mathtt{f}_1(\alpha,\gamma)$'s ``point of arrival'' to $z=\beta$: by equation \ref{43}, once $\beta$ appears as second coordinate in an output $-\langle\xi,\beta,\delta\rangle$ in the iterative expansion of $\mathtt{f}_1(\alpha,\gamma)$, some tail of the sequence (\ref{4110}) will follow by way of the term $\mathtt{f}_1(\xi,\beta)$. As $\beta$ is a limit, under our standing $C$-sequence assumptions (see Section \ref{bases}), the aforementioned $\delta$ can only have been $\beta+1$; if $\gamma>\beta+1$ then this output must have been preceded in the $\mathtt{f}_1(\alpha,\gamma)$-expansion by some $-\langle\xi',\beta+1,\delta'\rangle$, with $\xi'\leq\xi$. Reasoning along these lines focuses our attention on the region $\mathtt{f}_1(\alpha,\gamma)\big|_{[\alpha,\beta)\otimes[(\beta,\gamma]]^2}$, the prism (minus the $y=\beta$ plane) depicted in the $\alpha=0$ case of Figure \ref{walkdowntoz=beta} below. Note that
\begin{align}\label{3coordseq}-\langle\xi_i^{L},\xi_i^{T},\delta_i\rangle \end{align}
falls in this region only if $$\xi_i^{L}<\beta<\xi_i^{T}=C^{\delta_i}(\xi_i^{L})<\delta_i$$
and that in this case the indexing makes sense; by this we mean that there is a \textit{next} value in $\mathtt{f}_1(\alpha,\gamma)\big|_{[\alpha,\beta)\otimes[[\beta,\gamma]]^2}$, namely
$$-\langle\text{max}(\xi_i^{L},\text{max}(C_{\xi_i^{T}}\cap\beta)),\min C_{\xi^T_i}\backslash\beta,\xi_i^{T}\rangle$$
which we denote $-\langle\xi_{i+1}^{L},\xi_{i+1}^{T},\delta_{i+1}\rangle$. As above, these considerations have maximum scope when $\alpha=0$; therefore let it. Beginning with $\mathtt{f}_1(\alpha,\gamma)$, then, the right, left, and middle coordinates of the above-described sequence (ending when the middle coordinate is $\beta$) are more than a little familiar: writing $-\langle\xi_0^{L},\xi_0^{T},\delta_0\rangle$ for the first term in the expansion of $\mathtt{f}_1(\alpha,\gamma)$ to fall in the $[\alpha,\beta)\otimes[[\beta,\gamma]]^2$ region, we have $\delta_0=\gamma$ and $\xi^T_0=\min C_\gamma\backslash\beta$ and $\xi^L_0=\max C_\gamma\cap\beta$ (equaling $0$ if this intersection is empty) and, more generally,
\begin{align} \{\delta_0>\delta_1>\dots>\beta\}\, = \,
\{\gamma>\xi_0^{T}>\dots>\xi_{k-1}^{T}\} & = \,\text{Tr}(\beta,\gamma),\textnormal{ and}\label{412}\\ \label{4121} \{\xi_0^{L}\leq\xi_1^{L}\leq\dots\leq\xi_{k-1}^{L}\} & =\,\text{L}(\beta,\gamma), \end{align}
where \begin{align}\label{rhotwo}k=|\,\text{supp}(\mathtt{f}_1(0,\gamma)\big|_{\beta\otimes[[\beta,\omega_1)]^2})\,|=\rho_2(\beta,\gamma).\end{align}
In other words, the restrictions of $\mathtt{f}_0(\alpha,\gamma)$ to the regions $\beta\otimes[[\beta,\omega_1)]^2$ isolate sequences (\ref{3coordseq}) which coordinatewise are precisely \emph{the upper and lower traces of the walk from $\gamma$ down to $\beta$} (see again Section \ref{walkssection} for definitions). As $\bigcup_{0<\beta<\omega_1}\beta\otimes[[\beta,\omega_1)]^2=[\omega_1]^3$, each element of the support of any $\mathtt{f}_1(\alpha,\gamma)$ will fall in some such prism and, hence, in some sequence of the form (\ref{3coordseq}); in consequence, $\mathtt{f}_1$ may reasonably be regarded as little other than a knitting together, in a strikingly comprehensive fashion, of the upper and lower traces of pairs of countable ordinals. See Figure \ref{walkdowntoz=beta}.

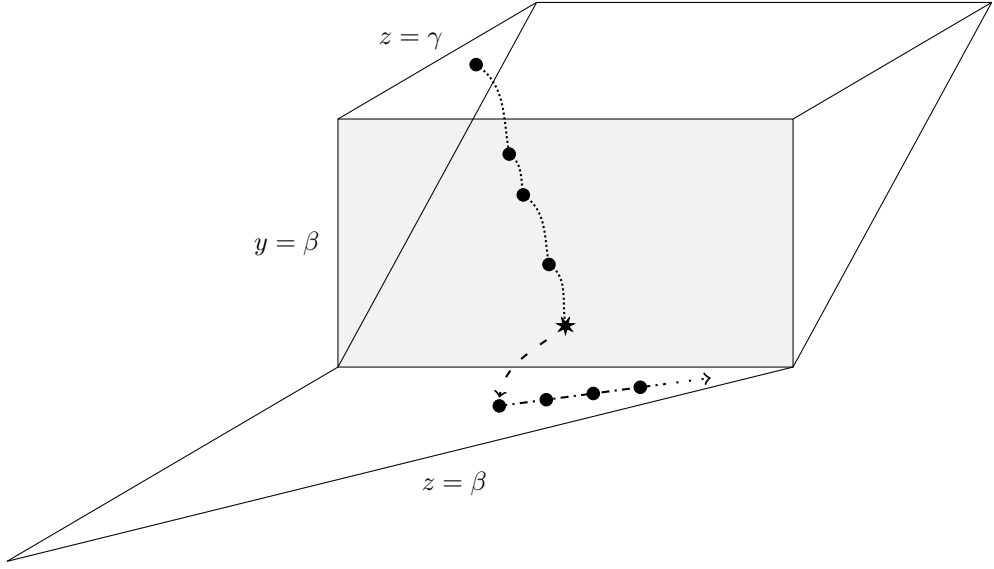
\begin{figure}
\centering
\begin{tikzpicture}[MyPersp]
\draw (0,0,0) -- (0,5,0);
\draw (0,5,0) -- (5,5,0) -- (5,5,3) -- (0,5,3) -- (0,5,0);
\fill[fill=gray, opacity=.1] (0,5,0) -- (5,5,0) -- (5,5,3) -- (0,5,3) -- (0,5,0);
\draw (5,5,0) -- (0,0,0);
\draw (5,8,3) -- (0,8,3);
\draw (0,5,0) -- (0,8,3);
\draw (5,5,0) -- (5,8,3);
\draw (0,5,3) -- (0,8,3);
\draw (5,5,3) -- (5,8,3);
\draw (.5,6.4,3) node[circle,draw,fill=black,scale=.5]{};
\draw (1.3,5.8,2.2) node[circle,draw,fill=black,scale=.5]{};
\draw[thick, densely dotted] (.5,6.4,3) to[out=-34,in=100] (1.3,5.8,2.2);
\draw (1.6,5.6,1.8) node[circle,draw,fill=black,scale=.5]{};
\draw (2.1,5.3,1.1) node[circle,draw,fill=black,scale=.5]{};
\draw (2.5,5,.5) node[star,star points=7,star point ratio=0.4,draw,fill=black,scale=.7]{};
\draw[thick, densely dotted] (1.3,5.8,2.2) to[out=-34,in=100] (1.6,5.6,1.8);
\draw[thick, densely dotted] (1.6,5.6,1.8) to[out=-34,in=100] (2.1,5.3,1.1);
\draw[thick, densely dotted] (2.1,5.3,1.1) to[out=-34,in=100] (2.5,5,.5);
\draw (2.5,4,0) node[circle,draw,fill=black,scale=.5]{};
\draw (2.9,4.16,0) node[circle,draw,fill=black,scale=.5]{};
\draw (3.3,4.32,0) node[circle,draw,fill=black,scale=.5]{};
\draw (3.7,4.48,0) node[circle,draw,fill=black,scale=.5]{};
\draw[->][thick, loosely dotted] (3.9,4.56,0) -- (4.3,4.72,0);
\draw[thick, dashdotted] (2.5,4,0) -- (3.9,4.56,0);
\draw[->][thick, loosely dashed] (2.5,5,.5) to[out=-140,in=80] (2.5,4,.1);
\draw (3,2,0) node[right]{$z=\beta$};
\draw (-.1,5,1.5) node[left]{$y=\beta$};
\draw (-.2,7,3) node[left]{$z=\gamma$};
\end{tikzpicture}
\caption{The path of the supports of $\mathtt{f}_1(0,\gamma)$ through $\beta\otimes[[\beta,\gamma]]^2$ to the plane $z=\beta$. This plane and prism sit within the larger $\mathtt{f}_1$ system just as the $z=\alpha$ plane and $\alpha\otimes[[\alpha,\beta]]^2$ prism do in Figure \ref{walksclubs}. The elements of the upper and lower traces of the walk from $\gamma$ down to $\beta$ form the coordinates of the first phase of this path; see (\ref{412}) and (\ref{4121}) below. More precisely, this phase begins at the point $(\max(C_\gamma\cap\beta),\min(C_\gamma\backslash\beta),\gamma)$, moving down and forward until $\beta$ appears as $y$-coordinate, in the starred node $(\max(\mathrm{L}(\beta,\gamma)),\beta,\beta+1)$ (for concreteness, here we take both $\beta$ and $\gamma$ to be limit ordinals). It then outputs an image of the tail of the club $C_\beta$ (as in (\ref{4110})) above $\max(\mathrm{L}(\beta,\gamma))$.}
\label{walkdowntoz=beta}
\end{figure}

These recognitions can illuminate the classical: by (\ref{rhotwo}), for example,
\begin{align}\label{rho2reasoning}|\rho_2(\alpha,\gamma)-\rho_2(\alpha,\beta)|\leq |\,\text{supp}(\mathtt{f}_1(0,\gamma)\big|_{\beta\otimes [\omega_1]^2}-\,\mathtt{f}_1(0,\beta))\,|\end{align}
for all $\alpha<\beta<\gamma<\omega_1$. In other words, for any such $\alpha,\beta,\gamma$ the difference between $\rho_2(\alpha,\gamma)$ and $\rho_2(\alpha,\beta)$ manifests as difference between $\mathtt{f}_1(0,\gamma)\!\!\upharpoonright_{\beta\otimes [\omega_1]^2}$ and $\mathtt{f}_1(0,\beta)$, which is finitely supported by equation \ref{45}. This imposes a uniform bound, immediately implying the \emph{mod bounded} coherence relations of $\rho_2$ recorded in equation \ref{cc}. Observe in contrast, however, that the \emph{mod bounded} \emph{nontriviality} relations of $\rho_2$ are considerably less spatially obvious within the $\mathtt{f}_1$ system. These latter relations appear only really to be accessible via a deeper Ramsey-theoretic analysis of the functions $\rho_2$, a point of more general significance below.

Observe lastly that even the functions $\rho_1(\beta,\gamma)$ and $\rho_0(\beta,\gamma)$ (see \cite{todwalks}) are legible (by (\ref{412})) in the second coordinate of the image of $\mathtt{f}_1(0,\gamma)\big|_{\beta\otimes[\,[\beta,\omega_1)]^2}$ in the uniform collapse of its $y$-axis to $\omega$, i.e., in the following function:
\begin{align*}
  \widetilde{\mathtt{f}}_1(\alpha,\beta) = \left\{\def\arraystretch{1}%
  \begin{array}{@{}c@{\quad}l@{}}
    0 & \hspace{.3 cm}\textnormal{if }\beta=\alpha+1\\
    -\langle\alpha,|\alpha\cap C_\beta|,\beta\rangle+\widetilde{\mathtt{f}}_1(\alpha,C^\beta(\alpha))+\widetilde{\mathtt{f}}_1(C^\beta(\alpha),\beta) & \hspace{.3 cm}\mathrm{otherwise}.\\
  \end{array}\right.
\end{align*}
The materialization of so much of the classical walks apparatus --- $\rho_0$, $\rho_1$, $\rho_2$, $\text{Tr}$, $\text{L}$, for example, even the clubs $C_\beta$ themselves --- as elementary spatial features of the $\mathtt{f}_1$ system may not itself be altogether surprising, in light of equation \ref{43}. The \textit{interesting} point is that $\mathtt{f}_1$ is only the first in an infinite series of nontrivial $n$-coherent systems $\mathtt{f}_n$ of broadly similar spatial organization.

\subsection{The basic form $\mathrm{tr}_n$}\label{thecasehighernagain}
As the reader may recall from Section \ref{thecasen1}, classical walks also appear in the $x=\beta$ planes of the functions $\mathtt{f}_1(\beta,\gamma)$ (see the planes $x=0$ and $x=\alpha$ of Figure \ref{walksclubs}). These in fact are the projections $\{-\langle \beta,\xi_i^{T},\delta_i\rangle\mid i<\rho_2(\beta,\gamma)-1\}$ to the $x=\beta$ plane of $\mathtt{f}_1(0,\gamma)\big|_{\beta\otimes[\,[\beta,\omega_1)]^2}$ or, more precisely, of all but the last element of the associated sequence (\ref{3coordseq}).

From the perspective of the $\mathtt{f}_1$ system, then, the classical upper trace $\mathrm{Tr}(\beta,\gamma)$ manifests at once in or as \begin{enumerate}
\item $\mathtt{f}_1(\beta,\gamma)\big|_{x=\beta}$,
\item the $y$ and $z$ coordinates of $\mathrm{supp}(\mathtt{f}_1(0,\gamma)\big|_{\beta\otimes [[\beta,\omega_1)]^2})$,
\end{enumerate}
and, via this second item, since $\bigcup_{0<\beta<\omega_1}\beta\otimes[[\beta,\omega_1)]^2=[\omega_1]^3$, as
\begin{enumerate}
\setcounter{enumi}{2}
\item the key \emph{finitary constituents} (together with $\mathrm{L}(\beta,\gamma)$) of the nontrivial coherent family of \emph{countably} supported functions $\mathtt{f}_1(0,\gamma)$ $(\gamma<\omega_1$). A related view is of the functions $\mathrm{Tr}$ and $\mathrm{L}$ as distilling away the redundancies of the $\mathtt{f}_1$ system, as in lines \ref{412} and \ref{4121} above.
\end{enumerate}
These interrelated conditions generalize: in this and the following section we describe functions $\mathrm{Tr}_n$, each recursively defined on the pattern of $\mathrm{Tr}_1:=\mathrm{Tr}$, satisfying the higher-$n$ analogues of items (1) through (3) above. Conditions (1) and/or (2) tell us computationally what these higher-order traces should be. Several minor choices arise as to how precisely to render this data, a point we return to below. For concreteness, though, we first record what we propose as the basic form $\mathrm{tr}_2$ of the degree-$2$ upper trace, one which is broadly ``read off'' from either $\mathtt{f}_2(\alpha,\beta,\gamma)\big|_{w=\alpha}$ or $\mathtt{f}_2(0,\beta,\gamma)\big|_{\alpha\otimes[ [\alpha,\omega_2)]^3}$ in the manner of (1) or (2) above, as the reader may verify.\footnote{Readers are encouraged to compare this definition with equation \ref{439}, from which it is derived. Observe that either of the determining restrictions just cited will entail that the $\mathtt{f}_2(C^{\beta\gamma}(\alpha),\beta,\gamma)$ (if $\beta\in C_\gamma$) and $\mathtt{f}_2(\beta,C^\gamma(\beta),\gamma)$ (if $\beta\not\in C_\gamma$) terms disappear in the translation to the $\mathrm{tr}_2$ form. This is because each of those terms exits the aforementioned restrictions (in either approach), and does not return to them in any subsequent step of its expansion. This in itself does not represent data loss, however, as the supports of those terms are, in aggregate, recovered by evaluations at other values of $\mathrm{tr}_2$. This phenomenon is visible already in the relations between the ``$2$-branching'' expansion-trees of $\mathtt{f}_1$ and the ``$1$-branching'' form of the classical $\mathrm{Tr}$, which readers also may find it valuable to compare.} The form is recursively defined as follows: for all $\alpha\leq\beta\leq\gamma<\omega_2$,
\begin{align}\label{thetr2form}
\mathrm{tr}_2(\alpha,\beta,\gamma)=
\begin{cases} 
\{\min(C_{\beta\gamma}\backslash\alpha)\} \,\sqcup\,\mathrm{tr}_2(\alpha,\text{min}(C_{\beta\gamma}\backslash\alpha),\gamma) & \\ \hspace{.3 cm} \sqcup\: \mathrm{tr}_2(\alpha,\text{min}(C_{\beta\gamma}\backslash\alpha),\beta) & \text{if }\beta\in C_\gamma \\
& \\
      \{\min(C_\gamma\backslash\beta)\}\, \sqcup\,\mathrm{tr}_2(\alpha,\text{min}(C_\gamma\backslash\beta),\gamma) & \\ \hspace{.3 cm} \sqcup\: \mathrm{tr}_2(\alpha,\beta,\text{min}(C_\gamma\backslash\beta)) & \text{if }\beta\not\in C_\gamma \\
   \end{cases}
\end{align}
Grounding this recursion are the following boundary conditions:
\begin{itemize}
\item if $\beta\in C_\gamma$ and $C_{\beta\gamma}\backslash\alpha=\varnothing$ then $\mathrm{tr}_2(\alpha,\beta,\gamma)=\varnothing$;
\item if $\beta=\gamma$ then $\mathrm{tr}_2(\alpha,\beta,\gamma)=\varnothing$.
\end{itemize}
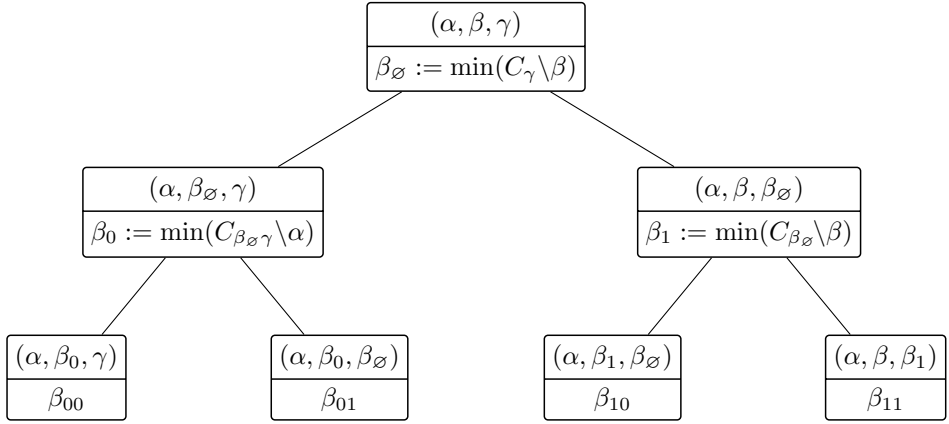
\begin{figure}
\centering
\begin{tikzpicture}[MyPersp]

\node[rectangle split,rectangle split parts=2,draw=black,line width=.2mm,inner sep=3pt, rounded corners=.5mm,text centered](A) at (4,0,7) {$(\alpha,\beta,\gamma)$ \nodepart{second} $\beta_{\varnothing}:=\min(C_\gamma\backslash\beta)$};

\node[rectangle split,rectangle split parts=2,draw=black,line width=.2mm,inner sep=3pt, rounded corners=.5mm,text centered](B) at (7,0,5) {$(\alpha,\beta,\beta_{\varnothing})$ \nodepart{second} $\beta_1:=\min(C_{\beta_{\varnothing}}\!\backslash\beta)$};

\node[rectangle split,rectangle split parts=2,draw=black,line width=.2mm,inner sep=3pt, rounded corners=.5mm,text centered](C) at (1,0,5) {$(\alpha,\beta_{\varnothing},\gamma)$ \nodepart{second} $\beta_0:=\min(C_{\beta_{\varnothing}\gamma}\backslash\alpha)$};

\node[rectangle split,rectangle split parts=2,draw=black,line width=.2mm,inner sep=3pt, rounded corners=.5mm,text centered](D) at (8.5,0,3) {$(\alpha,\beta,\beta_{1})$ \nodepart{second} $\beta_{11}$};

\node[rectangle split,rectangle split parts=2,draw=black,line width=.2mm,inner sep=3pt, rounded corners=.5mm,text centered](E) at (5.5,0,3) {$(\alpha,\beta_1,\beta_{\varnothing})$ \nodepart{second} $\beta_{10}$};

\node[rectangle split,rectangle split parts=2,draw=black,line width=.2mm,inner sep=3pt, rounded corners=.5mm,text centered](F) at (2.5,0,3) {$(\alpha,\beta_0,\beta_{\varnothing})$ \nodepart{second} $\beta_{01}$};

\node[rectangle split,rectangle split parts=2,draw=black,line width=.2mm,inner sep=3pt, rounded corners=.5mm,text centered](G) at (-.5,0,3) {$(\alpha,\beta_0,\gamma)$ \nodepart{second} $\beta_{00}$};

\draw (A) -- (B);
\draw (A) -- (C);
\draw (B) -- (D);
\draw (B) -- (E);
\draw (C) -- (F);
\draw (C) -- (G);
\end{tikzpicture}
\caption{First steps of $\mathrm{tr}_2(\alpha,\beta,\gamma)$. Associated to any $\mathrm{tr}_2$-input such as $(\alpha,\beta,\gamma)$ are an ordinal output, like $\beta_{\varnothing}$, and two further $\mathrm{tr}_2$-inputs, like $(\alpha,\beta_{\varnothing},\gamma)$ and $(\alpha,\beta,\beta_{\varnothing})$. Shaping the above diagram are the assumptions (made somewhat at random, simply for concreteness) that $\beta<\gamma$ and the outputs $\beta_\varnothing$, $\beta_0$, and $\beta_1$ defined as above are each meaningful and not equal to $\beta$. Lower outputs $\beta_\sigma$ may well correspond to undefined expressions, in which case they should be regarded as the empty set, marking the end of a branch.}
\label{asampletr2}
\end{figure}
Several remarks are immediately in order:
\begin{enumerate}
\item[(i)] Any $\mathrm{tr}_2(\alpha,\beta,\gamma)$ is naturally viewed as a binary tree, one in fact generalizing the $1$-branching tree (i.e., the walk) associated to $\mathrm{Tr}(\alpha,\beta)$. Depicted in Figure \ref{asampletr2} are the first two levels of the tree associated to $\mathrm{tr}_2(\alpha,\beta,\gamma)$ under generic assumptions on $\alpha$, $\beta$, and $\gamma$. The nodes of this tree are labeled with two sorts of data: as with the classical $\mathrm{Tr}$, the recursively defined function $\mathrm{tr}_2$ records an ordinal (appearing in the lower half of a node) then proceeds to new inputs (appearing in the top halves of successor nodes); each of these displayed in Figure \ref{asampletr2} above. This is because unlike in the classical case, the collection of ordinals output and the collection of tuples input in the course of a higher-dimensional walk are no longer informationally equivalent; while the former may be better suited to combinatorial applications, deductions \emph{per se} will often require the fuller data of the latter. This also is the reason for the disjoint unions appearing in equation \ref{thetr2form}: a novelty of the higher-dimensional traces is the possibility of output-ordinals arising therein repeatedly.
\item[(ii)] As mentioned, the form of $\mathrm{tr}_2$ derives from that of $\mathtt{f}_2(0,\beta,\gamma)\big|_{\alpha\otimes[ [\alpha,\omega_2)]^3}$ together with multiple presentational choices.\footnote{Note that this zone of attention is further dictated by the fact that unlike those of $\mathtt{f}_2(0,\beta,\gamma)\big|_{\alpha\otimes[ [\alpha,\omega_2)]^3}$, the supports of restrictions like $\mathtt{f}_2(0,\beta,\gamma)\big|_{[\alpha]^2\otimes[ [\alpha,\omega_2)]^2}$, for example, are not guaranteed to be finite.} One example of such a choice is the following: coordinates both of the form $\min (C_{\beta\gamma}\backslash\alpha)$ and $\min (C_\gamma\backslash\beta+1)$ may arise among the supports of $\mathtt{f}_2(0,\beta,\gamma)\big|_{\alpha\otimes[ [\alpha,\omega_2)]^3}$. Outside the contexts of $\mathbf{R}_n(\varepsilon)$ or $\mathbf{P}_n(\varepsilon)$, however, steps to $\min (C_\gamma\backslash\beta)$ are more natural than steps to $\min(C_\gamma\backslash\beta+1)$, and the combinatorial effects of uniformly adjusting definitions in this direction are superficial. Similar comments apply to the expansion of domain to inputs $\alpha\leq\beta\leq\gamma<\omega_2$,  to the choice of boundary conditions, and to the choice of which ordinals to record as outputs: the fact is, a number of variations on the basic idea hold fairly equal title to the name of ``higher walk''; what counts in any of them is that they record the data of the systems $\mathtt{f}_n$ with some of the concision (i.e., finitude) and versatility of classical walks.
\end{enumerate}
We now describe the fundamental features of the function $\mathrm{tr}_2$, framing our discussion in terms of Figure \ref{asampletr2}. Observe that in the passage from any node to one directly below, exactly one element of the associated coordinate-triple is replaced. Observe also that this element is never the least one, $\alpha$. Descending along the leftmost branch, for example, it is the second coordinate that is always changing; observe that its pattern $\beta$, $\beta_\varnothing$, $\beta_0,\dots$ is that of the $C_\gamma$-internal walk from $\beta$ down to $\alpha$ (see Section \ref{internalsection}). Along the rightmost branch, on the other hand, it is the third coordinate that is in motion; its pattern is visibly that of the classical walk from $\gamma$ down to $\beta$. Here it is natural to term the walks associated to rightwards paths through the tree \emph{external}. Any $\mathrm{tr}_2(\alpha,\beta,\gamma)$ may then be regarded as a structured family either of internal walks or of external walks, insofar as any binary tree is, as a set, simply the union either of its leftwards or rightwards branches.

Here a word of clarification is in order. While leftwards paths through $\mathrm{tr}_2(\alpha,\beta,\gamma)$ correspond precisely to internal walks, rightwards paths may properly contain classical walks in the following way: let $\beta'=\min (\mathrm{Tr}(\beta,\gamma)\backslash\{\beta\})$. If $\alpha'=\min(C_{\beta\beta'}\backslash\alpha)$ is defined then the node below and to the right of $(\alpha,\beta,\beta')$ is $(\alpha,\alpha',\beta)$. The rightwards path out of this node will then describe a classical walk from $\beta$ down to $\alpha'$, possibly again initiating a further walk out of $\alpha'$ upon arrival, and so on. External walks correspond in this way to iterated descending chains of classical walks (``walks of walks'').
\begin{lem}\label{lemma81} The set $\mathrm{tr}_2(\alpha,\beta,\gamma)$ is finite for any $\alpha\leq\beta\leq\gamma$.
\end{lem}
\begin{proof}
Observe that if node $(\delta',\varepsilon',\eta')$ is two steps below $(\delta,\varepsilon,\eta)$ in $\mathrm{tr}_2(\alpha,\beta,\gamma)$ then either $\varepsilon>\varepsilon'$ or $\eta>\eta'$. If $\mathrm{tr}_2(\alpha,\beta,\gamma)$ were infinite then it would have an infinite branch; by our observation, this branch would induce a strictly decreasing infinite sequence of ordinals, a contradiction.
\end{proof}
The derivation, via the generalized principles of (1) and (2) above, of a $\mathrm{tr}_2$ as in (\ref{thetr2form})  readily extends to define functions $\mathrm{tr}_n$ sending $(n+1)$-tuples of ordinals to $n$-branching trees; details are left to the reader. By essentially the same argument as that for Lemma \ref{lemma81}, these trees are all finite. Also as above, ``hyperplanes'' through these trees (i.e., subtrees of smaller branching number determined by some fixed rule of descent) correspond to lower-order walks-structures and their relativizations. To give something of the flavor of these generalizations, the first step of such a $\mathrm{tr}_3(\alpha,\beta,\gamma,\delta)$ would depend on whether $\gamma\in C_\delta$ and, if so, on whether $\beta\in C_{\gamma\delta}$ and, if so, on the value of $\min(C_{\beta\gamma\delta}\backslash\alpha)$.
\subsection{The functions $\mathrm{Tr}_n$ and $\rho^n_2$}\label{lastsection}
Mirroring the $n$-branching descending trees $\mathrm{tr}_n$ from below are higher-order \emph{lower trace} functions $L_n$ tracking the movement of the first coordinate of the associated outputs of $\mathtt{f}_n(0,\,\cdot\,)\big|_{\alpha\otimes [[\alpha,\omega_n)]^{n+1}}$; these $n$-branching trees increase, like $L_1$, along each of their branches. Together, the functions $\mathrm{tr}_n$ and $L_n$ essentially encompass the structured \emph{coordinate} data of the systems $\mathtt{f}_n$, just as in the case of $n=1$. For higher $n$, however, these $\mathtt{f}_n$ systems carry the additional data of \emph{sign}; the full-fledged upper trace functions $\mathrm{Tr}_n$ consist simply in the addition of this data to the functions $\mathrm{tr}_n$. To record this data, these functions take as inputs signed $n$-tuples of ordinals, and output signed ordinals, as in the example of $\mathrm{Tr}_2$:
\[ 
\mathrm{Tr}_2(\pm,\alpha,\beta,\gamma)=
\begin{cases} 
\{\mp\,\text{min}(C_{\beta\gamma}\backslash\alpha)\} \sqcup\, \mathrm{Tr}_2(\pm,\alpha,\text{min}(C_{\beta\gamma}\backslash\alpha),\gamma) & \\ \hspace{.1 cm} \sqcup\, \mathrm{Tr}_2(\mp,\alpha,\text{min}(C_{\beta\gamma}\backslash\alpha),\beta) & \text{if }\beta\in C_\gamma \\
      & \\
          \{\pm\,\text{min}(C_\gamma\backslash\beta)\} \sqcup\, \mathrm{Tr}_2(\pm,\alpha,\text{min}(C_\gamma\backslash\beta),\gamma) & \\ \hspace{.1 cm} \sqcup\, \mathrm{Tr}_2(\pm,\alpha,\beta,\text{min}(C_\gamma\backslash\beta)) & \text{if }\beta\not\in C_\gamma \\  
   \end{cases}
\]
The boundary conditions are just as before. Observe that, its outputs' signs being unvarying (and its outputs' ordinals all distinct),
$$\mathrm{Tr}_1(\pm,\alpha,\beta)=\{\mp\,\beta\}\,\sqcup\,\mathrm{Tr}_1(\pm,\alpha,\text{min}(C_\beta\backslash\alpha))$$
is informationally equivalent to the classical $\mathrm{Tr}$, and naturally identifies therewith (strict identification would require revision of the ordinal outputs to begin with $\beta$ and revision of the boundary values $\mathrm{Tr}_1(\pm,\alpha,\alpha)$ to $\{\pm\alpha\}$, but these are minor points). Just as subdivision of $n$-simplices is a reasonable heuristic for the iterative processes of $\mathrm{tr}_n$ (i.e., new $n$-tuples are formed from a new ordinal conjoined with elements of the boundary of the old one, viewed as a simplex), geometric notions of orientation form a natural heuristic for the signs arising in $\mathrm{Tr}_n$; as in multivariable calculus or geometry, for example, these signs or orientations only assume their proper significance in settings of more than two coordinates.

These signs exhibit useful and interesting organizing effects within $\mathrm{Tr}_n$. For the duration of the down-and-rightwards movement in $\mathrm{tr}_2$ that we identified above with a classical walk, for example, inputs' and outputs' signs are both constant, until the last step. If here a further walk is initiated, signs flip, remaining constant for its duration, and so on. Similarly, inputs' signs along any leftwards, internal walk of $\mathrm{Tr}_2$ are constant on the branch's full length; outputs' signs are constant after a possible first step ``up into'' the internal walk (see again Figure \ref{thepifigure}, caption). Hence one may speak not only of the signs of nodes, but of the signs of eventually-rightwards branches as well.

These observations should begin to suggest the number of interesting \emph{characteristics}, ``statistics,'' or $\rho$-functions which higher walks admit. Clearly, their fuller analysis falls beyond the scope of the present paper. We close, therefore, with a description of one of the simplest of these, a generalization of the function $\rho_2$, to better indicate where we believe the main next work to lie. For increasing $(n+1)$-tuples of ordinals $\vec{\alpha}$, define the function $\rho_2^n(\vec{\alpha})$ as the ``negative charge'' of $\mathrm{Tr}_n(\vec{\alpha})$, i.e., as the ``minuses minus the pluses'':
\begin{align*}
\rho_2^n(\vec{\alpha})= & \textit{ the number of negative signed ordinals in }\mathrm{Tr}_n(+,\vec{\alpha})\\ & - \textit{the number of positive signed ordinals in }\mathrm{Tr}_n(+,\vec{\alpha}).\end{align*}
In particular, under the boundary condition parenthesized above, $\rho_2^1(\alpha,\beta)=|\mathrm{Tr}_1(+,\alpha,\beta)|=|\mathrm{Tr}(\alpha,\beta)|-1=\rho_2(\alpha,\beta)$ for all $\alpha\leq\beta$. The reasoning at (\ref{rho2reasoning}) also readily generalizes; in other words, the following may be deduced either from the coherence of $\mathtt{f}_n$ or by arguments generalizing the classical:
\begin{prop} For each $n>0$ let $\Phi_n=\{\varphi_{\vec{\beta}}:\beta_0\to\mathbb{Z}\mid\vec{\beta}\in[\omega_n]^n\}$ denote the family of fiber maps $\varphi_{\vec{\beta}}(\alpha)=\rho_2^n(\alpha,\vec{\beta})$. Then $\Phi_n$ is an $n$-coherent family of functions, with respect to either of the moduli of \emph{bounded functions} or of \emph{locally constant functions}.
\end{prop}
Here \emph{$n$-coherence} refers to the relations described in Definition \ref{highernontriv}, with the modulus adjusted accordingly. To see the proposition's conclusion with respect to bounded functions, for example, observe that for all $\vec{\beta}\in [\omega_n]^{n+1}$ and $\alpha<\beta_0$,
$$\big|\sum_{i=0}^n(-1)^i\rho_2^n(\alpha,\vec{\beta}^i)\big|\leq\big|\mathrm{supp}\big(\sum_{i=0}^n(-1)^i\mathtt{f}_n(0,\vec{\beta}^i)\big|_{\beta_0\otimes [\omega_n]^{n+1}}\big)\big|$$
by the same reasoning as at (\ref{rho2reasoning}). Observe then that the right-hand side is finite by equation \ref{append3} of Appendix \ref{appendixAA}.

As in the case of $n=1$, establishing \emph{non-$n$-trivial} $n$-coherence is a taller order, typically calling on Ramsey-theoretic facts like the following:
\begin{thm}[\cite{todwalks}] Let $\mathcal{A}=\{(\beta_i,\gamma_i)\,|\,i\in\omega_1\}\subseteq \omega_1\times\omega_1$ satisfy $\max\{\beta_i,\gamma_i\}<\min\{\beta_j,\gamma_j\}$ for every $i<j$ in $\omega_1$. Then for any $\ell\in\mathbb{N}$ there exists a cofinal $\Gamma\subseteq\omega_1$ such that $\rho_2(\beta_i,\gamma_j)>\ell$ for any $i<j$ in $\Gamma$.
\end{thm}\label{lasttyping}
We term this property $\rho_2$-unboundedness below; see \cite[Lemma 2.4.3]{todwalks} for a stronger version of this theorem. This property is not far in spirit from one of the first and most celebrated applications of walks, namely the construction in \cite{todpairs} of witnesses to the negative partition relation
$$\aleph_1\not\to[\aleph_1]_{\aleph_0}^2.$$
Here the subscript could even be taken to be $\aleph_1$; our choice is mainly for parallelism with the following $\mathsf{ZFC}$ result from \cite{cubes},
$$\aleph_2\not\to[\aleph_1]_{\aleph_0}^3,$$
together with this relation's higher-order variants $\aleph_n\not\to[\aleph_1]_{\aleph_0}^{n+1}$, implicit therein.

This in turn suggests that one guiding question for future work should be whether functions deriving from higher walks witness higher-dimensional versions of the $\rho_2$-unboundedness of Theorem \ref{lasttyping}. This question does not appear altogether easy. Simply to motivate this work and define its central objects has been among this paper's main aims.
\section{Conclusion}\label{conclusion}

In the above, we described for each $n>0$ a number of interrelated $n$-dimensional (or $(n+1)$-dimensional) combinatorial phenomena correlating closely with the ordinal $\omega_n$. Much of the interest of these principles lies in their being $\mathsf{ZFC}$ phenomena, and indeed, our account has very deliberately avoided any appeal to additional set-theoretic assumptions. Having named these phenomena, however, the effects of such assumptions form some of the most immediate next questions.

\begin{quest}\label{question91}
How do the combinatorial phenomena of this work interact with $\square(\kappa)$-type principles? \textnormal{Here we ask also the complementary question:} Under what circumstances do these combinatorial phenomena fail at higher cardinals $\kappa$?
\end{quest}
Answers to these questions are well-understood in the $n=1$ case: $\square(\kappa)$, for example, ensures the existence of nontrivial coherent families of functions (and hence of nontrivial $1$-coherent trees) of height $\kappa$, while the P-Ideal Dichotomy implies that there are no such families on any ordinal of cofinality other than $\aleph_1$. One framing of the second part of Question \ref{question91} is the question of whether it is consistent that $\mathrm{\check{H}}^2(\omega_3;\mathcal{A})=0$ for all abelian groups $A$; any positive result seems certain to require large cardinal assumptions (see \cite{CoOI}).

Question \ref{question91} in part addresses the conspicuous question of ``beyond $\aleph_\omega$'', but there are other senses in which we might:
\begin{quest} To what degree may we regard the techniques of this paper as more general stepping-up principles, translating combinatorial phenomena on any $\aleph_\alpha$ to higher-dimensional phenomena on $\aleph_{\alpha+k}$?
\end{quest}
The question of whether families of combinatorial phenomena indexed by the natural numbers might extend into the transfinite is too nebulous to record formally, but cannot be altogether dismissed; Hjorth in \cite{hjorth} has described characterizations of all the alephs of \emph{countable} subscript, for example.

In broadest senses, the present work pursues a study of set-theoretic incompactness principles of higher dimension; it thereby raises questions of how this shaping notion of \emph{dimension} might in these contexts be made precise. The \emph{cohomological dimension} of Mitchell's Theorem is one way it might be, but a hallmark of classical dimension theory is the provable equivalence in ``nice'' settings of several otherwise distinct notions (see, e.g., \cite[Introduction]{hurewicz}). The structures we have highlighted are far from isolated; $n$-dimensional combinatorial phenomena on $\omega_n$ manifest with increasing frequency in a variety of mathematical fields, often to considerable effect. The following is a somewhat haphazard survey:
\begin{itemize}
\item Kuratowski's Free Set Theorem (\cite{sierpinski, kuratowski, sikorski}; see \cite{ehmr}) may be the best-known of results relating the cardinals $\aleph_n$ to their subscripts. The theorem in recent decades has found application to longstanding problems in both model theory \cite{malliarisshelah} (see also \cite{baldwinetal, laskowskishelah}) and lattice theory \cite{wehrung98, wehrung2007, ruzicka}.
\item Results in which the combinatorics of $\aleph_n$ make an appearance via assumptions like $2^{\aleph_0}=\aleph_n$ or even $\beth_n=\aleph_n$ are too numerous to even begin to list here. We do note, though, that the decisive precedent \cite{osofch} for Mitchell's Theorem was of this form (see the discussion in the appendix below); we note as well the importance of such assumptions in infinite combinatorics of an additive and/or spatial character \cite{komjath, schmerl, fox}, much as the present work's are.
\item In an application to Banach space theory, Lopez-Abad and Todorcevic construct length-$\omega_n$ normalized weakly-null sequences without unconditional subsequences via compoundings (of $\rho$ functions) not far in spirit from those pervading our work above \cite{todpos}.
\item In the note \cite{larson}, Larson records combinatorial features of finite subsets of the ordinals strongly evocative of Lebesgue covering dimension: associated to each $\omega_n$ are collections whose subsets ``reduce'' to sets of size $n+1$, but not in general to sets which are smaller.
\end{itemize}

Finer invariants than cohomological dimension are particular \v{C}ech cohomology computations, as we have described in Section \ref{71cohomology} above. The outstanding question in this area is the following:
\begin{quest} Is it a $\mathsf{ZFC}$ theorem that $\mathrm{\check{H}}^n(\omega_n;\mathbb{Z})\neq 0$ for all $n\geq 0$?
\end{quest}
The question is in some sense perverse: throughout our account above, combinatorial structures on $\omega_n$ (particularly for $n>1$) have appeared ``most at home'' in the wide berths of settings like $[\omega_n]^{n+2}$ or $\bigoplus_{\omega_n}\mathbb{Z}$. Hence one of the most natural and persuasive ways of certifying the $n$-dimensionality of $\omega_n$, namely, a $\mathsf{ZFC}$ argument that $\check{\mathrm{H}}^n(\omega_n;\mathbb{Z})\neq 0$, would entail overcoming the very affinity through which it was first perceived, i.e., it would entail realizing these ``wide-angle'' combinatorics on the much more restricted setting of $\mathbb{Z}$. It would be equally interesting if this is not possible in the $\mathsf{ZFC}$ framework, but it seems likelier that this task is simply a test of the depth of our understanding of the higher-dimensional combinatorics of the ordinals $\omega_n$. Higher walks seem to us both a potential resource and motivation for developing just this sort of understanding.
\begin{quest}\label{94}
Do functions deriving from higher walks exhibit higher-dimensional negative partition properties? In particular, do any exhibit higher-dimensional versions of $\rho_2$-unboundedness?
\end{quest}
As indicated, the main result of \cite{cubes} does suggest the existence of underexplored combinatorics in this direction.

We close with a simplified version of a question lingering from Section \ref{goodsection} (see footnote \ref{footnote}):
\begin{quest} Call a simplicial complex $X$ \emph{acyclic} if $\tilde{\mathrm{H}}_k^\Delta(X) = 0$ for all $k$. Is it the case that, for any $n>0$, any $n$-dimensional acyclic simplicial complex $X$ on a set $S$ may be extended to an $n$-dimensional acyclic simplicial complex $Y$ on $S$ with a complete ($n-1$)-skeleton, i.e., satisfying $X^{n-1}=[S]^n$?
\end{quest}
We do not expect this question to involve any deep set-theoretic considerations; its answer does, however, a bit surprisingly, appear to be unknown.

The question also cues the following reflection: we labored above to access some of the concrete content of Mitchell's abstract category-theoretic result; its nature in turn suggests that abstract simplicial or homotopical techniques could conceivably play a role in its further development. Particularly attractive would be a homotopical framework in which cardinal succession figures as suspension, helping to account, for example, for the chart concluding \cite[44]{CoOI}. Suggestive in this direction is how the ``space between'' one cardinal and the next accommodates a cone construction at the heart of Mitchell's original argument; this is one further reason we record that argument in our appendix, just below.

\appendix

\section{Mitchell's original argument and its background}

For simplicity, we restate Mitchell's theorem just in terms of ordinals:

\begin{thm}[\cite{rings}] \label{notprj} If $\varepsilon>0$ is an ordinal of cofinality $\aleph_\xi$ and $\xi$ is finite, then the projective dimension of $\mathbf{\Delta}_{\varepsilon}(\mathbb{Z})$ is $\xi+1$. If $\xi$ is infinite, then the projective dimension of $\mathbf{\Delta}_{\varepsilon}(\mathbb{Z})$ is $\infty$.
\end{thm}

That pd($\mathbf{\Delta}_{\varepsilon}(\mathbb{Z}))\leq \xi+1$ for such an $\varepsilon$ was known at the time, due to Goblot \cite{Goblot}. Hence the novelty of Mitchell's result was its computation of lower bounds for pd($\mathbf{\Delta}_{\varepsilon}(\mathbb{Z}))$; this part of the theorem may be rephrased as follows.

\begin{prop} \label{notproject}
Let $\varepsilon$ be of cofinality $\aleph_\xi$. Then $\mathbf{d}_n\mathbf{P}_n(\varepsilon)$ is not projective for any finite ordinal $n\leq\xi$.
\end{prop}
Below we sketch the original argument of Proposition \ref{notproject}, referring the reader to \cite{rings} or \cite{SSH} for details. Some words of context, though, seem to be in order before beginning.

Fundamental to all our arguments and constructions above were \emph{functor} or \emph{presheaf} categories $\mathsf{Ab}^{\varepsilon^{\text{op}}}$; the name ``presheaf'' derives from the centrality of $\mathsf{Ab}^{\tau(X)^{\text{op}}}$ to \v{C}ech or sheaf cohomology computations, relevant in our context as well (here $\tau(X)^{\text{op}}$ denotes the collection of open subsets of a topological space $X$, reverse-ordered by inclusion). More generally, let $P$ denote any partial order; Mitchell's point of departure was the resemblance of $\mathsf{Ab}^P$ to $R$-module categories $_R\mathsf{Mod}\cong\mathsf{Ab}^R$, where $R$ is a ring, construed on the right as a one-object additive category. Under this view, just as module theory is the representation theory of rings $R$, the study of the category $\mathsf{Ab}^P$ might be thought of as ``the representation theory of orders,'' and all of the foregoing may be viewed as a study of several of the most fundamental objects of these categories, $\mathbf{\Delta}_{\varepsilon}(\mathbb{Z})$ and $\mathbf{P}_n(\varepsilon)$ $(n\in\omega)$.

Put differently, the theorem that has formed our focus first emerged within a larger project of translating ``noncommutative homological ring theory [...] to (pre)additive category theory'' and as such incorporates multiple prior recognitions of the homological significance of the cardinals $\aleph_n$ \cite[2]{rings}. We would heartily recommended Osofsky's 1974 survey \emph{The subscript of $\aleph_n$, projective dimension, and the vanishing of lim$^n$} to any reader interested in that background, as well as Husainov's wider-ranging 2002 survey of Mitchell's work and its wake \cite{osofsubscript, husainov}. We record here just a few of the more noteworthy points:
\begin{enumerate}
\item ``The first irrefutable indications that cardinality was intimately tied up with projective dimension came in 1967 in two separate papers where lower bounds as well as upper bounds on dimensions were calculated in terms of subscripts of cardinalities'' \cite[14]{osofsubscript}. These were \cite{pierce} and \cite{osofvaluation}; rough outlines of the argument we will sketch below are legible in each.
\item Very shortly thereafter, Barbara Osofsky published \cite{osofch}; this is the acknowledged template for Mitchell's theorem \cite[6]{rings}. It seems telling that these first ``indications'' all emerged in work on rings which articulate \emph{orders}: Boolean rings, valuation rings, and directed rings, respectively. In \cite{osofch}, cardinal arithmetic assumptions transfer features of the cardinals $\aleph_n$ to rings of size continuum. Perhaps the best-known result in this line (cited in \cite[98]{weibel}, for instance), for example, is that the global dimension of $\prod_\omega\mathbb{C}$ is $k+1$ if and only if $2^{\aleph_0}=\aleph_k$. Most striking from our perspective, though, is the appendix of \cite{osofch}: therein, following the lead of \cite{bass}, Osofsky constructs bases for projective modules $d_1 P_1$ and $d_2 P_2$ quite close in spirit to the $n=1$ and $n=2$ cases of our more general constructions above. 
\end{enumerate}
Osofsky summarizes her survey as follows: ``What began as a study of dimension via derived functors branched off into a study of dimension via cardinality and came back to a study of derived functors via cardinality'' \cite[8]{osofsubscript}. As should be clear, the present work pursues a fourth combination in this sequence: the study of cardinality via derived functors and dimension. 
\begin{proof}[Sketch of proof of Proposition \ref{notproject}]

The argument is by induction on $\xi$. The base case $\xi=0$ consists in verifying that $\mathbf{d}_0\mathbf{P}_0(\varepsilon)\cong\mathbf{\Delta}_\varepsilon(\mathbb{Z})$ is not projective if $\varepsilon$ is a limit ordinal. The mechanism of the induction is an argument that if $\mathbf{d}_{n-1}\mathbf{P}_{n-1}(\delta)$ is not projective for any $\delta<\varepsilon$ with $\text{cf}(\delta)<\text{cf}(\varepsilon)$ then $\mathbf{d}_n\mathbf{P}_n(\varepsilon)$ is not projective either.\\

\noindent \underline{The base case}: By the following claim, if $\varepsilon$ is a limit ordinal then the epimorphism $\mathbf{d}_0:\mathbf{P}_0(\varepsilon)\rightarrow\mathbf{\Delta}_\varepsilon(\mathbb{Z})$ has no right-inverse. Hence $\mathbf{\Delta}_\varepsilon(\mathbb{Z})$ is not projective.
\begin{clm}\label{clma3}
Let $\varepsilon$ be a limit ordinal. Then the only morphism $\mathbf{f}:\mathbf{\Delta}_\varepsilon(\mathbb{Z})\rightarrow \mathbf{P}_0(\varepsilon)$ is the zero morphism.
\end{clm}
\begin{proof} Such an $\mathbf{f}$ is a collection of morphisms $\{f_\alpha:\mathbb{Z}\rightarrow\bigoplus_{[\alpha,\varepsilon)}\mathbb{Z}\,|\,\alpha<\varepsilon\}$ commuting with the bonding maps in $\mathbf{\Delta}_{\varepsilon}(\mathbb{Z})$ and $\mathbf{P}_0(\varepsilon)$; in consequence, $f_\alpha(1)$ must equal $f_\beta(1)$ for all $\alpha<\beta<\varepsilon$. This is not possible if $\beta>\min(\text{supp}(f_\alpha(1)))$. Hence $\min(\text{supp}(f_\alpha(1)))$ must not be defined for any $\alpha<\varepsilon$. Thus $\mathbf{f}$ is the zero morphism.
\end{proof}

\noindent \underline{The induction step}: This consists in showing that if $\mathbf{d}_n\mathbf{P}_n(\varepsilon)$ is projective and $n>0$ then for any regular $\kappa<\text{cf}(\varepsilon)$ there exists a $\delta\in\text{Cof}(\kappa)\cap\varepsilon$ with $\mathbf{d}_{n-1}\mathbf{P}_{n-1}(\delta)$ projective as well. This is argued via the following diagram:
\begin{align}\label{chartone}
\xymatrix{\mathbf{P}_n(\varepsilon) \ar@/_1pc/[d]_{\mathbf{p}} \ar[r] & \mathbf{d}_n\mathbf{P}_n(\varepsilon) \ar@/_1pc/@{.>}[l]_{\mathbf{s}} \ar@/_1pc/@{.>}[d]_{\mathbf{q}} \ar[r] & \mathbf{P}_{n-1}(\varepsilon) \ar@/_1pc/[d]_{\mathbf{p}} \\ \mathbf{P}_n(\delta) \ar@/_1pc/[u]_{\mathbf{i}} \ar[r] & \mathbf{d}_n\mathbf{P}_n(\delta) \ar@/_1pc/[u]_{\mathbf{j}} \ar[r] & \mathbf{P}_{n-1}(\delta) \ar@/_1pc/[u]_{\mathbf{i}} \ar[r] & \mathbf{d}_{n-1}\mathbf{P}_{n-1}(\delta)}
\end{align}

\noindent Rows are telescopings of the projective resolutions of $\mathbf{\Delta}_\varepsilon(\mathbb{Z})$ and $\mathbf{\Delta}_\delta(\mathbb{Z})$, respectively; they are, in other words, the natural decompositions of the differentials $\mathbf{d}_n:\mathbf{P}_n\rightarrow\mathbf{P}_{n-1}$ into $\mathbf{P}_n\twoheadrightarrow\mathbf{d}_n\mathbf{P}_n$ followed by $\mathbf{d}_n\mathbf{P}_n\hookrightarrow\mathbf{P}_{n-1}$. Natural projections $\mathbf{p}:\mathbf{P}_n(\varepsilon)\rightarrow\mathbf{P}_n(\delta)$ and inclusions $\mathbf{i}:\mathbf{P}_n(\delta)\rightarrow\mathbf{P}_n(\varepsilon)$ connect pairs $\mathbf{P}_n(\varepsilon)$ and $\mathbf{P}_n(\delta)$. Similarly, $\mathbf{d}_n\mathbf{P}_n(\delta)$ naturally includes into $\mathbf{d}_n\mathbf{P}_n(\varepsilon)$; what is perhaps surprising is that this inclusion $\mathbf{j}$ may have no left-inverse.\footnote{$\mathbf{j}:\mathbf{d}_1\mathbf{P}_1(5)\rightarrow\mathbf{d}_1\mathbf{P}_1(\omega)$, for example, does have a left-inverse, while $\mathbf{j}:\mathbf{d}_1\mathbf{P}_1(\omega)\rightarrow\mathbf{d}_1\mathbf{P}_1(\omega_1)$ does not, as the reader is encouraged to verify.} This is the first key observation in the induction: if $\mathbf{d}_n\mathbf{P}_n(\varepsilon)$ is projective and, hence, admits some section $\mathbf{s}$ of the map $\mathbf{d}_n$, then at closure points $\delta$ of $\mathbf{s}\,\mathbf{d}_n$, a left-inverse to $\mathbf{j}$ does exist --- namely, $\mathbf{q}=\mathbf{d}_n\,\mathbf{p}\,\mathbf{s}$.

The second key observation is that $\mathbf{q}$, together with the space in $\varepsilon$ above $\delta$, may be used to define a retract $\mathbf{r}$ of the inclusion $\mathbf{d}_n\mathbf{P}_n(\delta)\hookrightarrow\mathbf{P}_{n-1}(\delta)$. The existence of such an $\mathbf{r}$ will imply that $\mathbf{P}_{n-1}(\delta)\cong\mathbf{d}_n\mathbf{P}_n(\delta)\oplus\mathbf{d}_{n-1}\mathbf{P}_{n-1}(\delta)$, by the exactness of the sequence
\begin{align*}
\xymatrix{\mathbf{0} \ar[r] & \mathbf{d}_n\mathbf{P}_n(\delta) \ar[r] & \mathbf{P}_{n-1}(\delta) \ar[r] \ar@/_1pc/@{.>}[l]_{\mathbf{r}} & \mathbf{d}_{n-1}\mathbf{P}_{n-1}(\delta) \ar[r] & \mathbf{0}}
\end{align*}

\noindent Hence $\mathbf{d}_{n-1}\mathbf{P}_{n-1}(\delta)$ is a summand of the free system $\mathbf{P}_{n-1}(\delta)$. By Lemma \ref{summandlemma}, $\mathbf{d}_{n-1}\mathbf{P}_{n-1}(\delta)$ is therefore projective. If we have shown that $\text{pd}(\mathbf{\Delta}_\delta(\mathbb{Z}))\geq n$, then this is a contradiction, hence our assumption that $\mathbf{d}_n\mathbf{P}_n(\varepsilon)$ is projective was false. In consequence, $\text{pd}(\mathbf{\Delta}_\varepsilon(\mathbb{Z}))\geq n+1$.

In the preceding paragraph, we referenced ``the space in $\varepsilon$ above $\delta$'': fix $\xi\in\varepsilon\backslash\delta$. The key device in this second part of the argument --- i.e., in the derivation of a retract $\mathbf{r}$ from $\mathbf{q}$ --- is the \emph{formation of a cone}
over $\mathbf{P}_{n-1}(\delta)$ in $\mathbf{P}_n(\varepsilon)$. By this we mean the following: let $\mathbf{Q}_n(\delta,\xi)$ be the subsystem of $\mathbf{P}_n(\varepsilon)$ generated by $\{\langle\vec{\alpha},\xi\rangle\,|\,\vec{\alpha}\in [\delta]^n\}$. As the reader may verify, there are natural inclusion-relations between $\mathbf{d}_n\mathbf{Q}_n(\delta,\xi)$ and several of the main terms in the diagram \ref{chartone}. These we denote $\mathbf{t}$, $\mathbf{u}$, and $\mathbf{v}$ in the diagram below; to see that $\mathbf{t}$, for example, is meaningful, observe that $d_n\langle\vec{\beta}\rangle=\sum_{i=0}^n(-1)^i\langle\vec{\beta}^i\rangle=\sum_{i=0}^n(-1)^{i+n}d_n\langle\vec{\beta}^i,\xi\rangle$ for any $\vec{\beta}\in [\delta]^{n+1}$.
\begin{align}\label{charttwo}
\xymatrix{\mathbf{P}_n(\varepsilon) \ar@/_1pc/[dd]_{\mathbf{p}} \ar[r] & \mathbf{d}_n\mathbf{P}_n(\varepsilon) \ar@/_1pc/@{.>}[l]_{\mathbf{s}} \ar@/_1pc/@{.>}[dd]_{\mathbf{q}} \ar[rr]^{\mathbf{b}} & & \mathbf{P}_{n-1}(\varepsilon) \ar@/_1pc/[dd]_{\mathbf{p}} \ar@/_1pc/@{.>}[dl]^{\mathbf{w}} \\ & & \mathbf{d}_n\mathbf{Q}_n(\delta,\xi) \ar@/_1pc/[ur]_{\mathbf{v}} \ar[ul]_{\mathbf{u}} & & \\ \mathbf{P}_n(\delta) \ar@/_1pc/[uu]_{\mathbf{i}} \ar[r] & \mathbf{d}_n\mathbf{P}_n(\delta) \ar@/_1pc/[uu]_{\mathbf{j}} \ar[ur]^{\mathbf{t}} \ar[rr]_{\mathbf{a}} & & \mathbf{P}_{n-1}(\delta) \ar@/_1pc/@{.>}[ll]_{\mathbf{r}} \ar@/_1pc/[uu]_{\mathbf{i}} \ar[r] & \mathbf{d}_{n-1}\mathbf{P}_{n-1}(\delta)}
\end{align}

\noindent What the cone construction critically affords us is a retract, $\mathbf{w}$, of $\mathbf{v}$. This is defined as follows: for $\vec{\beta}\in [\varepsilon]^n$, let $$\mathbf{w}(\langle\vec{\beta}\rangle) =
\begin{cases}
(-1)^n d_n\langle\vec{\beta},\xi\rangle & \text{if }\vec{\beta}\in [\delta]^n \\
0 & \text{otherwise}
\end{cases}
$$
For a generator $d_n\langle\vec{\alpha},\xi\rangle$ of $\mathbf{d}_n\mathbf{Q}_n(\delta,\xi)$,
$$\mathbf{w}\mathbf{v}(\mathbf{d}_n\langle\vec{\alpha},\xi\rangle)=\mathbf{w}\bigg(\sum_{i=0}^{n-1} (-1)^i\langle\vec{\alpha}^i,\xi\rangle + (-1)^n \langle\vec{\alpha}\rangle\bigg)=d_n\langle\vec{\alpha},\xi\rangle$$
Hence $\mathbf{w}$ is a retract of $\mathbf{v}$, as desired. The point is the following: 

Write $\mathbf{a}$ for $\mathbf{d}_n\mathbf{P}_n(\delta)\hookrightarrow\mathbf{P}_{n-1}(\delta)$, as in diagram \ref{charttwo} above. Then given a $\mathbf{q}$ left-inverse to $\mathbf{j}$, the map $\mathbf{r}=\mathbf{q}\mathbf{u}\mathbf{w}\mathbf{i}$ is left-inverse to $\mathbf{a}$:
$$\mathbf{r}\mathbf{a}=\mathbf{q}\mathbf{u}\mathbf{w}\mathbf{i}\mathbf{a}=\mathbf{q}\mathbf{u}\mathbf{w}\mathbf{b}\mathbf{j}=\mathbf{q}\mathbf{u}\mathbf{w}\mathbf{b}\mathbf{u}\mathbf{t}=\mathbf{q}\mathbf{u}\mathbf{w}\mathbf{v}\mathbf{t}=\mathbf{q}\mathbf{u}\mathbf{t}=\mathbf{q}\mathbf{j}$$
The equation records a diagram-chase on (\ref{charttwo}) above, together with the fact that $\mathbf{w}\mathbf{v}=\mathbf{id}$. It shows that $\mathbf{r}$ is indeed a retract of $\mathbf{a}$, and thereby concludes the induction step.
\end{proof}
\section{Details of the proof that $\mathtt{f}_n$ is nontrivial}
\label{appendixAA}
Here we record in greater notational detail the argument outlined at the conclusion of Section \ref{efffn} that the functions $\mathtt{f}_n$ are nontrivial. We follow that outline closely. In particular, our argument is by induction on $n$; therefore assume the claim true for all $\mathtt{f}_m$ with $m<n$, recalling that we handled the base cases of $n=1$ and $n=2$ in Sections \ref{thecasen1} and Sections \ref{efffn}, respectively.

Consider now the function $\mathtt{f}_n\in K^n(\mathbf{R}_{n+1}(\omega_n))$ derived, in the manner described in Section \ref{section52}, from equation \ref{effff}. Equation \ref{fcoherence} takes the form of the following coherence condition:
\begin{align}\label{append1}\sum_{i=0}^{n+1}(-1)^i\mathtt{f}_n(\vec{\alpha}^i)=^* 0\hspace{.3 cm}\text{ for all }\vec{\alpha}\in [\omega_n]^{n+2}
\end{align}
Our aim is to show that there exists no $\mathtt{e}_{n-1}\in K^{n-1}(\mathbf{R}_{n+1}(\omega_n))$ trivializing $\mathtt{f}_n$ in the manner of equation \ref{showingthis}, i.e., satisfying
\begin{align}\label{append2}
\sum_{i=0}^n(-1)^i\mathtt{e}_{n-1}(\vec{\alpha}^i)=^*\mathtt{f}_n(\vec{\alpha})\hspace{.3 cm}\text{ for all }\vec{\alpha}\in [\omega_n]^{n+1}
\end{align}
To that end, we begin by deducing from equation \ref{effff} that
\begin{align}\label{fcontainment}\text{supp}(\mathtt{f}_n(\vec{\beta}))\subseteq [\,[\beta_0,\beta_n]\,]^{n+2}\hspace{.3 cm}\text{ for all }\vec{\beta}\in [\omega_n]^{n+1}\end{align}
and hence, by (\ref{append1}), that
\begin{align}\label{append3}\sum_{i=0}^{n}(-1)^i\mathtt{f}_n(0,\vec{\beta}^i)\big|_{\beta_0\otimes [\omega_n]^{n+1}}=^* 0\hspace{.3 cm}\text{ for all }\vec{\beta}\in [\omega_n]^{n+1}
\end{align}
This is the ``degree reduction'' in coherence relations alluded to in the argument's outline in Section \ref{section52}. Similarly, by definition,
$$\text{supp}(\mathtt{e}_{n-1}(\vec{\beta}))\subseteq [\,[\beta_0,\omega_n]\,]^{n+2}\hspace{.3 cm}\text{ for all }\vec{\beta}\in [\omega_n]^{n}$$
for any $\mathtt{e}_{n-1}\in K^{n-1}(\mathbf{R}_{n+1}(\omega_n))$; this implies the relation
\begin{align}\label{append4}
\sum_{i=0}^{n-1}(-1)^i\mathtt{e}_{n-1}(0,\vec{\beta}^i)\big|_{\beta_0\otimes [\omega_n]^{n+1}}=^*\mathtt{f}_n(0,\vec{\beta})\big|_{\beta_0\otimes [\omega_n]^{n+1}}\hspace{.3 cm}\text{ for all }\vec{\beta}\in [\omega_n\backslash\{0\}]^{n}
\end{align}
for any $\mathtt{e}_{n-1}$ as in equation \ref{append2}. Hence to show that $\mathtt{f}_n$ is nontrivial in the sense of equation \ref{append2}, it will suffice to show that there exists no $\mathtt{e}_{n-1}$ as in equation \ref{append4}.

For this purpose the key point is the existence of nontrivial $(n-1)$-coherent families of functions within the restriction, for each $\gamma\in S^n_{n-1}$, of $\mathtt{f}_n(\,\cdot\,,\gamma)$ to the hyperplane $z=\gamma$; it is here that we make use of our inductive hypothesis. Before proceeding, though, we will require some additional notation.

Fix $\gamma\in S^n_{n-1}$ and let $\pi$ denote the order-isomorphism $\omega_{n-1}\to C_\gamma$. It is then immediate from definitions that for any $\vec{\beta}\in [C_\gamma]^n$, the maximal proper internal tail of $(\vec{\beta},\gamma)$ is $(\vec{\alpha},\gamma)$, where $\vec{\xi}$ is the maximal proper initial tail of $\pi^{-1}[\vec{\beta}]$ and $\vec{\alpha}=\pi[\vec{\xi}]$. It follows that $\mathtt{b}(\vec{\beta},\gamma)=\langle\pi[\mathtt{b}(\pi^{-1}[\vec{\beta}])],\gamma\rangle$ for all such $\vec{\beta}$. Here we've committed the mild abuse of applying $\pi$ to a group element, but our meaning should be clear; more generally, for any $a\in\prod_{[\omega_{n-1}]^{n+1}}\mathbb{Z}$, let $\pi[a]$ denote the element of $\prod_{[\gamma]^{n+1}}\mathbb{Z}$ induced by the order-isomorphism $\pi$ together with the inclusion of $C_\gamma$ into $\gamma$. Also, for any element $a$ of $\prod_{[\gamma]^{n+1}}\mathbb{Z}$, define the element $a * \gamma$ of $\prod_{[\omega_n]^{n+1}\otimes\{\gamma\}}\mathbb{Z}$ by $a*\gamma\,(\vec{\alpha},\gamma)=a(\vec{\alpha})$. In particular, $\pi[0]=0$ and $0*\gamma=0$. Our central assertion may now be stated as follows:
\begin{align}\label{alonggamma}
\mathtt{f}_n(\vec{\beta},\gamma)\big|_{[\omega_n]^{n+1}\otimes\{\gamma\}}=\pi[\mathtt{f}_{n-1}(\pi^{-1}[\vec{\beta}])]*\gamma\hspace{.3 cm}\text{ for all }\vec{\beta}\in[C_\gamma]^n
\end{align}
Put differently, the order-isomorphism $\pi$ translates the nontrivial coherence of $\mathtt{f}_{n-1}$ to the restriction of the $\mathtt{f}_n(\,\cdot\,,\gamma)$-images of $[C_\gamma]^n$ to the hyperplane $z=\gamma$. To see (\ref{alonggamma}), observe first that $\mathtt{f}_n(\vec{\beta},\gamma)=0$ if and only if $\mathtt{f}_{n-1}(\pi^{-1}(\vec{\beta}))=0$, for any $\vec{\beta}\in[C_\gamma]^n$. If for some such $\vec{\beta}$ this is not the case then
\begin{align}\label{24revisited}
\mathtt{f}_n(\vec{\beta},\gamma)\big|_{[\omega_n]^{n+1}\otimes\{\gamma\}}=(\text{-}1)^{j+1}\Big[\mathtt{b}(\vec{\beta},\gamma)-\displaystyle\sum_{i=0}^{j}\,(\text{-}1)^i \mathtt{f}_n(\mathtt{b}(\vec{\beta,\gamma})^i)-\displaystyle\sum_{i=j+2}^{n}\,(\text{-}1)^i \mathtt{f}_n(\mathtt{b}(\vec{\beta},\gamma)^i)\Big]\big|_{[\omega_n]^{n+1}\otimes\{\gamma\}}
\end{align}
for some $j\leq n-2$, by equation \ref{effff}. Note that the term $\mathtt{f}_n(\mathtt{b}(\vec{\beta},\gamma)^{n+1})$ of the latter has disappeared from the above sum, on the grounds that it contributes nothing to the restricted range $[\omega_n]^{n+1}\otimes\{\gamma\}$. Observe that the branches of the expansion (in the sense of Lemma \ref{nocircularity}) of (\ref{24revisited}) are of the form
\begin{align*}(\vec{\beta},\gamma)= & \pi[\pi^{-1}[\vec{\beta}]]*\gamma\;\longrightarrow\;\mathtt{b}(\vec{\beta},\gamma)=\pi[\mathtt{b}(\pi^{-1}[\vec{\beta}])]*\gamma\;\longrightarrow\;\pi[\mathtt{b}(\pi^{-1}[\vec{\beta}])^{k_1}]*\gamma \\ &\longrightarrow\;\mathtt{b}(\pi[\mathtt{b}(\pi^{-1}[\vec{\beta}])^{k_1}]*\gamma)=\pi[\mathtt{b}(\mathtt{b}(\pi^{-1}[\vec{\beta}])^{k_1})]*\gamma\;\longrightarrow\;\pi[\mathtt{b}(\mathtt{b}(\pi^{-1}[\vec{\beta}])^{k_1})^{k_2}]*\gamma\;\longrightarrow\cdots \end{align*}
where each $k_i\leq n$ is other than the index of the coordinate added by the $i^{\mathrm{th}}$ application of $\mathtt{b}$. It is then easy to see that these branches are precisely those which may be written as $\pi[B]*\gamma$, where $B$ is a branch of the expansion of $\mathtt{f}_{n-1}(\pi^{-1}[\vec{\beta}])$. More precisely, the $\pi$ and $*\gamma$ operations induce a correspondence between the expansion-trees, and thus between the expansions, and, hence, between the values of $\mathtt{f}_{n-1}(\pi^{-1}[\vec{\beta}])$ and $\mathtt{f}_n(\vec{\beta},\gamma)\big|_{[\omega_n]^{n+1}\otimes\{\gamma\}}$, in exactly the sense recorded in equation \ref{alonggamma}.

This correspondence underlies the following lemma, from which the remainder of step $n$ of our inductive argument rapidly follows.
\begin{lem}\label{appBlemma} Let $\mathtt{e}_{n-1}\in K^{n-1}(\mathbf{R}_{n+1}(\omega_n))$ be as in equation \ref{append4}. For any $\gamma\in S^n_{n-1}$ at which (\ref{alonggamma}) holds, there exists some $\vec{\alpha}_\gamma\in [C_\gamma]^n$ such that
\begin{align}\label{11July}\Bigg(\sum_{i=0}^{n-1}(-1)^i \mathtt{e}_{n-1}(0,\vec{\alpha}_\gamma^i)\Bigg)\Big|_{\vec{\alpha}_{\gamma}(0)\otimes [\omega_n]^n\otimes\{\gamma\}}\neq0.\end{align}
As previously noted, $\vec{\alpha}_{\gamma}(0)$ denotes the minimum element of $\vec{\alpha}_{\gamma}$.
\end{lem}
Since we have shown that (\ref{alonggamma}) holds at all $\gamma\in S^n_{n-1}$, Lemma \ref{11July} implies that there exists a stationary $S\subseteq S_{n-1}^n$ and $\vec{\alpha}\in [\omega_n]^n$ such that $\vec{\alpha}_\gamma=\vec{\alpha}$ for all $\gamma\in S$. However, (\ref{fcontainment}) and (\ref{append4}) together imply that
\begin{align}
err(\vec{\alpha}):=\Big\{\gamma>\vec{\alpha}\;|\:\Big(\sum_{i=0}^{n-1}(-1)^i \mathtt{e}_{n-1}(0,\vec{\alpha}_\gamma^i)\Big)\Big|_{\vec{\alpha}_{\gamma}(0)\otimes [\omega_n]^n\otimes\{\gamma\}}\neq0\Big\} \text{ is finite for all }\vec{\alpha}\in[\omega_n]^n.
\end{align}
This is the desired contradiction, showing that no $\mathtt{e}_{n-1}\in K^{n-1}(\mathbf{R}_{n+1}(\omega_n))$ can trivialize $\mathtt{f}_n$ in the sense of equation \ref{append4}, and thereby concluding our argument.
\begin{proof}[Proof of Lemma \ref{appBlemma}] It is our inductive assumption that there exists no $\mathtt{e}_{n-2}\in K^{n-2}(\mathbf{R}_n(\omega_{n-1}))$ trivializing $\mathtt{f}_{n-1}$ in the (lower degree) sense of (\ref{append4}). We will show that if the conclusion of Lemma \ref{appBlemma} failed, then the $\mathtt{e}_{n-1}$ under discussion would induce an $\mathtt{e}_{n-2}$ which contradicts this assumption. Our argument consists in two claims. It is convenient at the outset to assume that $0\in C_\gamma$; there is no loss of generality in doing so, as in any case $\mathtt{f}_n(0,\vec{\beta},\gamma)$ will ``expand into'' $\mathtt{f}_n\big|_{[C_\gamma]^n\otimes\{\gamma\}}$, within which setting our arguments would, without this assumption, apply with only cosmetic changes.

\begin{clm}\label{endlessness} If the negation of (\ref{11July}) holds for all $\vec{\alpha}_\gamma\in [C_\gamma\backslash\{0\}]^n$ then there exists a $\mathtt{g}:[C_\gamma]^{n-1}\to\prod_{[\omega_n]^{n+1}\otimes\{\gamma\}}\mathbb{Z}$ such that 
\begin{align}\label{12July}\bigg(\sum_{i=0}^{n-2}(-1)^i \mathtt{g}(0,\vec{\alpha}^i)\bigg)\Big|_{\alpha_0\otimes [\omega_n]^n\otimes\{\gamma\}} = \mathtt{e}_{n-1}(0,\vec{\alpha})\Big|_{\alpha_0\otimes [\omega_n]^n\otimes\{\gamma\}}\end{align}
for all $\vec{\alpha}\in [C_\gamma\backslash\{0\}]^{n-1}$.
\end{clm}
\begin{proof}
One approach is to observe that functions $\mathtt{e}_{n-1}$ satisfying the negation of (\ref{11July}) for all $\vec{\alpha}_\gamma\in [C_\gamma\backslash\{0\}]^n$ define cocycles in cohomology groups corresponding to $\mathrm{lim}^{n-2}$ of a \emph{flasque} inverse system, in the sense of Jensen's \cite[page 5]{jensen}; Theorem 1.8 therein then converts in our context to the existence of a trivializing $\mathtt{g}$ in precisely the sense of (\ref{12July}). A more computational argument would consist in verifying that letting
$$\mathtt{g}(0,\vec{\alpha})\big|_{\{\xi\}\otimes [\omega_n]^n\otimes\{\gamma\}}=-\mathtt{e}_{n-1}(0,\xi+1,\vec{\alpha})\big|_{\{\xi\}\otimes [\omega_n]^n\otimes\{\gamma\}}$$
for all $\vec{\alpha}\in [C_\gamma\backslash\{0\}]^{n-2}$ with $\alpha_0\in\mathrm{Lim}$ and $\xi<\alpha_0$ partially defines a function $\mathtt{g}$ which then canonically extends to satisfy the conclusion of the claim on the domain described. 
\end{proof}
\begin{clm} If there exists a $\mathtt{g}$ as in Claim \ref{endlessness} for an $\mathtt{e}_{n-1}$ as in (\ref{append4}), then there exists an $\mathtt{e}_{n-2}$ trivializing $\mathtt{f}_{n-1}$ in the (lower degree) sense of (\ref{append4}). 
\end{clm}
\begin{proof}
The $\mathtt{e}_{n-2}$ in question is the $\pi^{-1}$-translation of the function $\mathtt{e}_{n-1}(0,\,\cdot\,,\gamma)+(-1)^{n-1}\mathtt{g}(0,\,\cdot\,)$. This works as asserted by equation \ref{alonggamma}, together with the fact that
\begin{align*}
& \sum_{i=0}^{n-2}(-1)^i\Big(\mathtt{e}_{n-1}(0,\vec{\alpha}^i,\gamma)+(-1)^{n-1}\mathtt{g}(0,\vec{\alpha}^i)\Big)\Big|_{[\omega_n]^{n+1}\otimes\{\gamma\}} \\
=^{\hspace{.17 cm}} & \sum_{i=0}^{n-1}(-1)^i\mathtt{e}_{n-1}(0,(\vec{\alpha},\gamma)^i)\big|_{[\omega_n]^{n+1}\otimes\{\gamma\}}\\
=^* & \mathtt{f}_n(0,\vec{\alpha},\gamma)\big|_{[\omega_n]^{n+1}\otimes\{\gamma\}}
\end{align*}
for all $\vec{\alpha}\in [C_\gamma\backslash\{0\}]^{n-1}$. The first and second equalities follow from our claim's first and second premises, respectively, and the equality of the first and third lines should be read as a relativization to the $z=\gamma$ hyperplane of the $\mathtt{e}_{n-2}$ variant of equation \ref{append4}.
\end{proof}
\end{proof}
\section{A proof of Theorem \ref{twoiffone}}\label{appendixB}
We first recall the statement of the theorem:
\begin{unthm} There exists a height-$\varepsilon$ $A$-valued nontrivial $n$-coherent$^{\mathrm{I}}$ family of functions if and only if there exists a height-$\varepsilon$ $A$-valued nontrivial $n$-coherent$^{\mathrm{II}}$ family of functions.
\end{unthm}
As remarked above, we will in fact show that for every $n>0$ and ordinal $\varepsilon$ and abelian group $A$,
\begin{align}\label{lefttoright}\frac{\mathsf{coh}^{\mathrm{I}}(n,A,\varepsilon)}{\mathsf{triv}^{\mathrm{I}}(n,A,\varepsilon)}\cong\frac{\mathsf{coh}^{\mathrm{II}}(n,A,\varepsilon)}{\mathsf{triv}^{\mathrm{II}}(n,A,\varepsilon)}.\end{align}
In this context we write $[\Phi_n]_{\mathrm{I}}$ and $[\mathcal{S}_n]_{\mathrm{II}}$ for the cosets associated to $n$-coherent$^{\mathrm{I}}$ families $\Phi_n$ and $n$-coherent$^{\mathrm{II}}$ families $\mathcal{S}_n$, respectively, omitting the subscripts when they are clear from context.
\begin{proof} We will define functions $a_n$ and $b_n$ from the left side to the right and from the right side to the left, respectively, of equation \ref{lefttoright} such that $a_n b_n$ and $b_n a_n$ are each the identity map. A few preliminary observations will be useful:
\begin{itemize}
\item  For $n>0$, any $n$-coherent$^{\mathrm{I}}$ family of successor height $\varepsilon$ is $n$-trivial$^{\mathrm{I}}$; similarly for any $n$-coherent$^{\mathrm{II}}$ family of successor height $\varepsilon$. Hence we may, and will, restrict our attention below to limit ordinals $\varepsilon$.
\item If $X$ is a cofinal subset of $\varepsilon$ then the class $[\Phi_n]_{\mathrm{I}}$ of an $n$-coherent$^{\mathrm{I}}$ family $\Phi_n$ is determined by $\Phi_n\big|_X:=\{\varphi_{\vec{\beta}}\mid\vec{\beta}\in [X]^n\}$. Put differently, if $\Phi_n\big|_X$ is $n$-trivial$^{\mathrm{I}}$, then so too is $\Phi_n$. Similarly for $n$-coherence$^{\mathrm{II}}$. Hence we may, and will, define $b_n$-images as classes $[\Psi_n\big|_X]_{\mathrm{I}}$, where $X\subseteq\varepsilon$ is the cofinal set $\varepsilon\cap\mathrm{Succ}$.
\item A family's initial segments are all $n$-trivial$^{\mathrm{I}}$ or $n$-trivial$^{\mathrm{II}}$ if and only if that family is $n$-coherent$^{\mathrm{I}}$ or
$n$-coherent$^{\mathrm{II}}$, respectively.\end{itemize}
We now define the maps $a_n$ and $b_n$. It is both convenient and illustrative to first handle the cases of $n=1$ and $n=2$. Throughout this discussion, assume the abelian group $A$ in question to be fixed.
\begin{enumerate}
\item The map $a_1$ is that induced by the map $\tilde{a}_1$ sending any $1$-coherent$^{\mathrm{I}}$ $\Phi_1=\{\varphi_\beta\mid\beta\in \varepsilon\}$ to the family $\mathcal{S}_1=\{s^1_\beta\mid\beta<\varepsilon\}$ via the assignments $s_\beta^1(\alpha)=\varphi_\beta(\alpha)$ for each $\alpha<\beta<\varepsilon$. The map $b_1$ is that induced by the reverse of this operation. Clearly these maps are as desired.
\item The map $a_2$ is that induced by the map $\tilde{a}_2$ sending any $2$-coherent$^{\mathrm{I}}$ $\Phi_2=\{\varphi_{\vec{\beta}}:\vec{\beta}\in [\varepsilon]^2\}$ to the family $\mathcal{S}_2=\{s_\gamma^2\mid\gamma<\varepsilon\}$ via the assignments $s^2_\gamma(\alpha,\beta)=\varphi_{\beta\gamma}(\alpha)$ for each $\alpha<\beta<\gamma<\varepsilon$. The family $\mathcal{S}_2$ is $2$-coherent$^{\mathrm{II}}$ if for all $\gamma<\delta<\varepsilon$ there exists a $t^1_{\gamma\delta}:\gamma\to A$ such that for all $\beta<\gamma$, the expression
\begin{align}\label{coherenttwocondition}t^1_{\gamma\delta}\big|_\beta-(s_\delta^2(\,\cdot\,,\beta)-s_\gamma^2(\,\cdot\,,\beta))\end{align}
is $0$-trivial$^{\mathrm{II}}$, or in other words is finitely supported. Of course the function $\alpha\mapsto\varphi_{\gamma\delta}(\alpha)$ is just such a $t^1_{\gamma\delta}$, by the $2$-coherence$^{\mathrm{I}}$ of $\Phi_2$; this shows that $\tilde{a}$ is indeed a map $\mathsf{coh}^{\mathrm{I}}(n,A,\varepsilon)\to\mathsf{coh}^{\mathrm{II}}(n,A,\varepsilon)$. The argument that it is a map $\mathsf{triv}^{\mathrm{I}}(n,A,\varepsilon)\to\mathsf{triv}^{\mathrm{II}}(n,A,\varepsilon)$ as well is similar.

It is tempting now to define the map $b_2$ as that induced by the reverse of $\tilde{a}_2$, but this would be misguided: above $n=1$, this operation can fail to send $2$-coherent$^{\mathrm{II}}$ families to $2$-coherent$^{\mathrm{I}}$ ones. The correct approach is rather to let $b_2([\mathcal{S}_2])=[\Psi_2=\{\psi_{\vec{\beta}}\mid\vec{\beta}\in[\varepsilon\cap\mathrm{Succ}]^2\}]$, where each $\psi_{\gamma\delta}$ is a trivializing function $t^1_{\gamma\delta}$ as in (\ref{coherenttwocondition}) above. Cancellations of the $s$-terms imply that $(t^1_{\delta\eta}-t^1_{\gamma\eta}+t^1_{\gamma\delta})\big|_{\beta}=^*0$ for all $\beta<\gamma$ with $\gamma<\delta<\eta$ in $\varepsilon\cap\mathrm{Succ}$, implying in turn that $\Psi_2$ is $2$-coherent$^{\mathrm{I}}$, as desired (note the importance of the fact that $\mathrm{cf}(\gamma)\neq\aleph_0$ to this reasoning). Below, in the paragraph beginning with (\ref{andanother}), we argue more generally that this map does not depend on the choices of $t^1_{\gamma\delta}$; since in the case of $\mathcal{S}_2=\tilde{a}_2(\Phi_2)$ each $t^1_{\gamma\delta}$ \emph{may} be taken to be $\varphi_{\gamma\delta}$, the composition $b_2 a_2$ is therefore indeed the identity, as claimed.

The composition $a_2 b_2$ is also the identity if and only if for every $2$-coherent$^{\mathrm{I}}$ family of witnesses $\{t^1_{\gamma\delta}\mid\gamma<\delta<\varepsilon\}$ to the $2$-coherence$^{\mathrm{II}}$ of a family $\mathcal{S}_2$, the family $\mathcal{S}_2-\mathcal{R}_2$ is $2$-trivial$^{\mathrm{II}}$, where $\mathcal{R}_2=\{r_\gamma^2\mid\gamma<\varepsilon\}$ is given by the assignments $r_\gamma^2(\alpha,\beta)=t^1_{\beta\gamma}(\alpha)$. This is so if and only if there exists a $u:[\varepsilon]^2\to A$ such that for all $\gamma<\varepsilon$,
\begin{align}\label{indeedonetrivial} u\big|_{[\gamma]^2}-(s_\gamma^2-r_\gamma^2)\end{align}
is $1$-trivial$^{\mathrm{II}}$. To that end, fix a function $f:\varepsilon\to\varepsilon$ satisfying $\beta<f(\beta)$ for all $\beta<\varepsilon$ and let $u(\,\cdot\,,\beta)=s_{f(\beta)}^2(\,\cdot\,,\beta)-t^1_{\beta,f(\beta)}(\,\cdot\,)$ for each $\beta<\varepsilon$. Then for all $\gamma<\varepsilon$,\begin{align*}
u(\,\cdot\,,\beta)-(s_\gamma^2(\,\cdot\,,\beta)-r_\gamma^2(\,\cdot\,,\beta)) & =\; s_{f(\beta)}^2(\,\cdot\,,\beta)-t^1_{\beta,f(\beta)}(\,\cdot\,)-s_\gamma^2(\,\cdot\,,\beta)+t_{\beta\gamma}^1(\,\cdot\,) \\ & =^* t^1_{\gamma, f(\beta)}(\,\cdot\,)-t^1_{\beta,f(\beta)}(\,\cdot\,)+t_{\beta\gamma}^1(\,\cdot\,) \\ & =^* 0
\end{align*}
for every $\beta<\gamma$ (in the above computation we assumed that $f(\beta)<\gamma$; clearly the negation of this assumption has the same effect). Hence the expression (\ref{indeedonetrivial}) is $1$-trivial$^{\mathrm{II}}$ and thus the composition $a_2 b_2$ is the identity, as claimed.
\end{enumerate}
We turn now to the maps $a_n$ and $b_n$ for $n>2$. The definitions and verifications for these maps broadly follow the pattern of the $n=2$ case above.

The simplest steps of the argument concern the functions $a_n$. As above, $a_n$ is defined to be the function induced by a more concrete $\tilde{a}_n$ amounting to a rearrangement operation on the arguments and indices of any family of functions $\Phi_n$. By the following lemma, this operation converts $n$-coherence$^{\mathrm{I}}$ and $n$-triviality$^{\mathrm{I}}$ to $n$-coherence$^{\mathrm{II}}$ and $n$-triviality$^{\mathrm{II}}$, respectively, ensuring, in particular, that $a_n$ is well-defined.
\begin{lem}\label{rerangement} Let $\tilde{a}_n$ denote the function sending any $\Phi_n=\{\varphi_{\vec{\beta}}:\beta_0\to A\mid\vec{\beta}\in [\varepsilon]^n\}$ to the family $\mathcal{S}_n=\{s^n_\beta:[\beta]^n\to A\mid\beta<\varepsilon\}$ given by $s^n_{\beta_{n-1}}(\alpha,\beta_0,\dots,\beta_{n-2})=\varphi_{\vec{\beta}}(\alpha)$ for each $\vec{\beta}\in[\varepsilon]^n$ and $\alpha<\vec{\beta}$. Then
\begin{itemize}
\item if $\Phi_n$ is $n$-coherent$^{\mathrm{I}}$ then $\tilde{a}_n(\Phi_n)$ is $n$-coherent$^{\mathrm{II}}$, and
\item if $\Phi_n$ is $n$-trivial$^{\mathrm{I}}$ then $\tilde{a}_n(\Phi_n)$ is $n$-trivial$^{\mathrm{II}}$.
\end{itemize}
\end{lem}
\begin{proof}
We verify that $\tilde{a}_n$ is indeed a map $\mathsf{coh}^{\mathrm{I}}(n,A,\varepsilon)\to\mathsf{coh}^{\mathrm{II}}(n,A,\varepsilon)$. Essentially by definition, $\mathcal{S}_n$ is $n$-coherent$^{\mathrm{II}}$ if for every $\beta_{n-1}<\beta_n$ there exists a $t^{n-1}_{\beta_{n-1}\beta_n}:[\beta_{n-1}]^{n-1}\to A$ such that for all $\beta_{n-2}<\beta_{n-1}$,
\begin{align}\label{anotherlabel}t^{n-1}_{\beta_{n-1}\beta_n}\big|_{[\beta_{n-2}]^{n-1}}-(s_{\beta_n}^n(\,\cdot\,,\beta_{n-2})-s_{\beta_{n-1}}^n(\,\cdot\,,\beta_{n-2}))
\end{align}
is $(n-2)$-trivial$^{\mathrm{II}}$. We claim the assignments $t^{n-1}_{\beta_{n-1}\beta_n}(\alpha,\beta_0,\dots,\beta_{n-3})=\varphi_{\vec{\beta}^{n-2}}(\alpha)$ for each $\vec{\beta}\in [\varepsilon]^{n+1}$ and $\alpha<\vec{\beta}$ define such $t^{n-1}_{\beta_{n-1}\beta_n}$. To that end we define $(n-2)$-trivializations $t^{n-2}_{\beta_{n-2}\beta_{n-1}\beta_n}$ for the expressions (\ref{anotherlabel}) via the assignment $t^{n-2}_{\beta_{n-2}\beta_{n-1}\beta_n}(\alpha,\beta_0,\dots,\beta_{n-4})=\varphi_{\vec{\beta}^{n-3}}(\alpha)$ for each $\vec{\beta}\in [\varepsilon]^{n+1}$ and $\alpha<\vec{\beta}$ --- and so on, down to $t^1_{\vec{\beta}^0}(\alpha)=\varphi_{\vec{\beta}^0}(\alpha)$. The verification that these successive trivializations interact as desired amounts to showing that an alternating sum, together with $s_{\beta_n}^n-s_{\beta_{n-1}}^n$, is $0$-trivial$^{\mathrm{II}}$; by our assignments, this translates in turn to the equation $\sum_{i=0}^n(-1)^i \varphi_{\vec{\beta}^i}=^{*}0$, which holds for all $\vec{\beta}\in [\varepsilon]^{n+1}$, of course, exactly when $\Phi_n$ is $n$-coherent$^{\mathrm{I}}$. The verification that $\tilde{a}_n$ maps $\mathsf{triv}^{\mathrm{I}}(n,A,\varepsilon)$ to $\mathsf{triv}^{\mathrm{II}}(n,A,\varepsilon)$ is similar and left to the reader.
\end{proof}

Cascading families of trivializing functions like the $t^k_{\vec{\beta}}$ appearing above form a main motif in analyses of $n$-coherence$^\mathrm{II}$ and $n$-triviality$^{\mathrm{II}}$. They arise, for example, within a slightly different schema in the definition of $b_n$ appearing below. In light of such variety, we adopt a somewhat open definition:
\begin{defin}
A \emph{cascading family of trivializations} is a collection of the form $$\{t^k_{\vec{\gamma}}:[\gamma_0]^k\to A\mid\vec{\gamma}\in [X]^{n-k+1}\text{ and }1\leq k< n\}$$
for some set of ordinals $X$ and $n>2$, in which, for each $k>1$, each $t^{k-1}_{\vec{\gamma}}$ trivializes$^{\mathrm{II}}$ an expression containing $t^k_{\vec{\beta}}$ for some $\vec{\beta}\subseteq\vec{\gamma}$. As above, such collections are often defined in the presence of a distinguished family of functions $\Phi_n$; in such cases, definitions of $t^k_{\vec{\gamma}}$ via $t^k_{\vec{\gamma}}(\alpha,\vec{\beta})=\varphi_{\vec{\beta}\vec{\gamma}}(\alpha)$ for all $(\alpha,\vec{\beta})\in [\gamma_0]^k$ and $\vec{\gamma}\in [\varepsilon]^{n-k+1}$ will be called the \emph{natural assignments} with respect to $\Phi_n$.
\end{defin}
We now define the maps $b_n$. As in the case of $n=2$, the $b_n$-image of any $[\mathcal{S}_n]$ will be a $[\Psi_n=\{\psi_{\vec{\beta}}\mid\vec{\beta}\in[\varepsilon\cap\mathrm{Succ}]^n\}]$ given by $\psi_{\vec{\beta}}(\alpha)=t^1_{\vec{\beta}}(\alpha)$, but for higher $n$ the derivation of $t^1_{\vec{\beta}}$ is more elaborate. First observe that as $\mathcal{S}_n$ is $n$-coherent$^{\mathrm{II}}$, for each $(\gamma,\delta)\in [\varepsilon\cap\mathrm{Succ}]^2$ there exist $t_{\gamma\delta}^{n-1}$ such that equation \ref{anotherlabel} holds with $\gamma=\beta_{n-1}$ and $\delta=\beta_n$. It follows that $(t_{\delta\eta}^{n-1}-t_{\gamma\eta}^{n-1}+t_{\gamma\delta}^{n-1})\big|_{[\beta]^{n-1}}$ is $(n-2)$-trivial$^{\mathrm{II}}$ for all $\gamma<\delta<\eta$ in $\varepsilon\cap\mathrm{Succ}$ and $\beta<\gamma$; as $\gamma$ is a successor, $t_{\delta\eta}^{n-1}-t_{\gamma\eta}^{n-1}+t_{\gamma\delta}^{n-1}$ itself admits an $(n-2)$-trivialization$^{\mathrm{II}}$ $t_{\gamma\delta\eta}^{n-2}$. Reasoning just as above, families $$t_{\delta\eta\xi}^{n-2}-t_{\gamma\eta\xi}^{n-2}+t_{\gamma\delta\xi}^{n-2}-t_{\gamma\delta\eta}^{n-2}$$ will then admit $(n-3)$-trivializations$^{\mathrm{II}}$ $t_{\gamma\delta\eta\xi}^{n-3}$, and so on: this sequence of $t_{\vec{\beta}}^{n-i}$ $(\vec{\beta}\in [\varepsilon\cap\mathrm{Succ}]^{i+1})$ ends with the $t_{\vec{\beta}}^1$ $(\vec{\beta}\in [\varepsilon\cap\mathrm{Succ}]^{n})$ we desire.

It is easy to see that $\Psi_n$ so defined is $n$-coherent$^{\mathrm{I}}$. This follows from the fact that each $$t^1_{\vec{\beta}}\big|_\alpha=^*\Big(\sum_{i=0}^{n-1}(-1)^i t^2_{\vec{\beta}^i}\Big)(\,\cdot\,,\alpha)$$
for all $\alpha<\beta_0$; as $\beta_0$ is a successor, the parameter $\alpha$ may be dropped. We then have
$$\sum_{j=0}^{n}(-1)^j t^1_{\vec{\beta}^j}=^*\sum_{j=0}^n(-1)^j \sum_{i=0}^{n-1}(-1)^i t^2_{(\vec{\beta}^j)^i}=^*0$$ for all $\vec{\beta}\in[\varepsilon\cap\mathrm{Succ}]^{n+1}$, as desired.

Suppose now that $\mathcal{S}_n$ is the $\tilde{a}_n$-image of some $n$-coherent$^{\mathrm{I}}$ family of functions $\Phi_n$. Then as the reader may verify, natural assignments for the cascading family of trivializations $t^k_{\vec{\beta}}$ defining $b_n$ would have been valid, ending with $t^1_{\vec{\beta}}=\varphi_{\vec{\beta}}$. In consequence, once we have shown that the above procedure indeed determines a well-defined map $$b_n:\frac{\mathsf{coh}^{\mathrm{II}}(n,A,\varepsilon)}{\mathsf{triv}^{\mathrm{II}}(n,A,\varepsilon)}\to\frac{\mathsf{coh}^{\mathrm{I}}(n,A,\varepsilon)}{\mathsf{triv}^{\mathrm{I}}(n,A,\varepsilon)}$$it will follow almost immediately that $b_n a_n$ is indeed the identity, as claimed. We therefore focus on showing the former.

To this end, fix $\mathcal{R}_n,\mathcal{S}_n\in \mathsf{coh}^{\mathrm{II}}(n,A,\varepsilon)$ with $\mathcal{S}_n-\mathcal{R}_n$ $n$-trivial$^{\mathrm{II}}$. In particular, there exists a $v^n:[\varepsilon]^n\to A$ and $v_\delta^{n-1}:[\delta]^n\to A$ for each $\delta\in \varepsilon\cap\mathrm{Succ}$ such that \begin{align}\label{andanother}\tag{$\star(\delta)$}v^{n-1}_\delta\big|_{[\beta]^{n-1}}-(v^n-(s^n_\delta-r^n_\delta))(\,\cdot\,,\beta)\end{align} is $(n-2)$-trivial$^{\mathrm{II}}$ for all $\beta<\delta$. Fix also cascading families of trivializations $u_{\vec{\beta}}^{n-i}$ and $t_{\vec{\beta}}^{n-i}$ $(\vec{\beta}\in [\varepsilon\cap\mathrm{Succ}]^{i+1})$ deriving the functions $\theta_{\vec{\beta}}:=u^1_{\vec{\beta}}$ and $\psi_{\vec{\beta}}:=t^1_{\vec{\beta}}$, respectively, from $\mathcal{S}_n$ and $\mathcal{R}_n$ in the manner described above. We will show the family $\Theta_n-\Psi_n=\{\theta_{\vec{\beta}}-\psi_{\vec{\beta}}\mid \vec{\beta}\in[\varepsilon\cap\mathrm{Succ}]^n\}$ to be $n$-trivial$^{\mathrm{I}}$. The crux of the matter is this: in combination, the $s$ and $r$ terms within the $u$ and $t$ variants of equation \ref{anotherlabel} (again with $\gamma=\beta_{n-1}$ and $\delta=\beta_n$) and within the $\delta$ and $\gamma$ variants of equation (\ref{andanother}), cancel, so that \begin{align}\label{lastpromise}v_\delta^{n-1}-v_\gamma^{n-1}+u_{\gamma\delta}^{n-1}-t_{\gamma\delta}^{n-1}\end{align} admits an $(n-2)$-trivialization$^{\mathrm{II}}$,which we denote by $w_{\gamma\delta}^{n-2}$. More precisely, the expressions ($\star(\delta)$) and ($\star(\gamma)$) are each $(n-2)$-trivial$^{\mathrm{II}}$, as are the expressions
\begin{align*}
& u^{n-1}_{\gamma\delta}\big|_{[\beta]^{n-1}}-(s^n_\delta(\,\cdot\,,\beta)-s^n_\gamma(\,\cdot\,,\beta))\\
\text{ and }& t^{n-1}_{\gamma\delta}\big|_{[\beta]^{n-1}}-(r^n_\delta(\,\cdot\,,\beta)-r^n_\gamma(\,\cdot\,,\beta)),
\end{align*}
and the first minus the second plus the third minus the fourth yields (\ref{lastpromise}). Two sums then each $(n-2)$-trivialize$^{\mathrm{II}}$ $$(u^{n-1}_{\delta\eta}-t^{n-1}_{\delta\eta})-(u^{n-1}_{\gamma\eta}-t^{n-1}_{\gamma\eta})+(u^{n-1}_{\gamma\delta}-t^{n-1}_{\gamma\delta}),$$ namely $\sigma_{\gamma\delta\eta}:=u_{\gamma\delta\eta}^{n-2}-t_{\gamma\delta\eta}^{n-2}$ and $\tau_{\gamma\delta\eta}:=w_{\delta\eta}^{n-2}-w_{\gamma\eta}^{n-2}+w_{\gamma\delta}^{n-2}$. If $n=3$, we are done: the $\sigma$ and $\tau$ functions both $1$-trivialize$^{\mathrm{II}}$ the same families; as $\gamma$ in all such cases is a successor, this tells us that $\{w_{\gamma\delta}^{n-2}\mid(\gamma,\delta)\in[\varepsilon\cap\mathrm{Succ}]^2\}$ $3$-trivializes$^{\mathrm{I}}$ $\Theta_3-\Psi_3$, as desired. For $n>3$, observe that each $\sigma_{\gamma\delta\eta}-\tau_{\gamma\delta\eta}$ admits an $(n-3)$-trivialization$^{\mathrm{II}}$, which we denote $w_{\gamma\delta\eta}^{n-3}$. As the associated $\tau$ terms will cancel, $w_{\delta\eta\xi}^{n-3}-w_{\gamma\eta\xi}^{n-3}+w_{\gamma\delta\xi}^{n-3}-w_{\gamma\delta\eta}^{n-3}$ will then $(n-3)$-trivialize$^{\mathrm{II}}$ $$(u_{\delta\eta\xi}^{n-2}-t_{\delta\eta\xi}^{n-2})-(u_{\gamma\eta\xi}^{n-2}-t_{\gamma\eta\xi}^{n-2})+(u_{\gamma\delta\xi}^{n-2}-t_{\gamma\delta\xi}^{n-2})-(u_{\gamma\delta\eta}^{n-2}-t_{\gamma\delta\eta}^{n-2}),$$ and may thus be compared with $\sigma_{\gamma\delta\eta\xi}:=u_{\gamma\delta\eta\xi}^{n-3}-t_{\gamma\delta\eta\xi}^{n-3}$, just as before. This process may be repeated until reaching the stage at which the functions $w_{\vec{\gamma}}^1$ $n$-trivialize$^{\mathrm{I}}$ $\Theta_n-\Psi_n$ as in the $n=3$ case, as desired.

Only that $a_n b_n$ is the identity now remains to be seen. A representative of $a_n b_n([\mathcal{S}_n])$ is a family $\mathcal{R}_n$ given by $r_{\gamma}^n(\alpha,\vec{\beta})=t^1_{\vec{\beta}\gamma}(\alpha)$, where the family of functions $t^1_{\vec{\beta}\gamma}$ is one deriving from $\mathcal{S}_n$ in the manner defining $b_n$, as described above. What we must show is that any such $\mathcal{S}_n-\mathcal{R}_n$ is $n$-trivial$^{\mathrm{II}}$. This is perhaps the intuitively clearest of our claims --- in essence the assertion is that $\mathcal{S}_n$ and a family given by trivializations of combinations of elements of $\mathcal{S}_n$ do not differ by a \emph{non}-trivial family --- but may also be the most computationally tedious to verify. The cleanest approach seems to be to break the difference $\mathcal{S}_n-\mathcal{R}_n$ into steps; the general method is clear from the case of $n=3$. Terms $t_{\vec{\beta}}^k$ will be those defining the $b_3$-image of $\mathcal{S}_3$. As in the case of $n=2$, we fix a function $f:\varepsilon\to\varepsilon$ satisfying $\beta<f(\beta)$ for all $\beta<\varepsilon$; we then show that there exists a $u^3:[\varepsilon]^3\to A$ such that for all $\delta<\varepsilon$ there exists a $u_\delta^2:[\delta]^2\to A$ such that for all $\gamma<\delta$ there exists a $u_{\gamma\delta}^1:\gamma\to A$ such that\begin{align}\label{onelastlongone} 0 =^* u_{\gamma\delta}^1(\,\cdot\,)-u_\delta^2(\,\cdot\,,\beta)+u^3(\,\cdot\,,\beta,\gamma)-s_\delta^3(\,\cdot\,,\beta,\gamma)+t_{\gamma\delta}^2(\,\cdot\,,\beta)
\end{align}
Namely (assuming for simplicity that $f(\gamma)<\delta$), let $u^3(\,\cdot\,,\beta,\gamma)=s_{f(\gamma)}^3(\,\cdot\,,\beta,\gamma)-t_{\gamma f(\gamma)}^2(\,\cdot\,,\beta)$. Letting $v_{\gamma f(\gamma)\delta}$ be the $1$-trivialization$^{\mathrm{II}}$ of $t_{f(\gamma)\delta}^2(\,\cdot\,,\beta)-(s_\delta^3(\,\cdot\,,\beta,\gamma)-s_{f(\gamma)}^3(\,\cdot\,,\beta,\gamma))$ $(\beta<\gamma)$ given by the definition of $t_{f(\gamma)\delta}^2$, we then have
\begin{align*}
u^3(\,\cdot\,,\beta,\gamma)-s_\delta^3(\,\cdot\,,\beta,\gamma)+t^2_{\gamma\delta}(\,\cdot\,,\beta) & = \; s_{f(\gamma)}^3(\,\cdot\,,\beta,\gamma)-s_\delta^3(\,\cdot\,,\beta,\gamma)-t_{\gamma f(\gamma)}^2(\,\cdot\,,\beta)+t_{\gamma\delta}^2(\,\cdot\,,\beta) \\ & =^* v_{\gamma f(\gamma)\delta}(\,\cdot\,)-t_{f(\gamma)\delta}^2(\,\cdot\,,\beta)-t_{\gamma f(\gamma)}^2(\,\cdot\,,\beta)+t_{\gamma\delta}^2(\,\cdot\,,\beta) \\ & =^*(v_{\gamma f(\gamma)\delta}-t_{\gamma f(\gamma)\delta}^1)\big|_\beta
\end{align*}
Now observe that the assignments $u_{\gamma\delta}^1=t_{\gamma f(\gamma)\delta}^1-v_{\gamma f(\gamma)\delta}$, $u_2=0$, and $u_3$ as defined above do indeed solve equation \ref{onelastlongone}, as desired. A similar trick, involving $f(\beta)$ instead of $f(\gamma)$, will satisfy equation \ref{onelastlongone} with $t_{\gamma\delta}^2(\,\cdot\,,\beta)$ and $t_{\beta\gamma\delta}^1(\,\cdot\,)$ taking the place of $s_\delta^3(\,\cdot\,,\beta,\gamma)$ and $t_{\gamma\delta}^2(\,\cdot\,,\beta)$, respectively; together these equations show that $\mathcal{S}_n-\mathcal{R}_n$ is $3$-trivial$^{\mathrm{II}}$, as desired. A little thought will persuade the reader more effectively than any further computations that this method of trivialized ``steps'' from $s_{\gamma}^n$ to $t_{\vec{\gamma}}^1$ is entirely general, and this will conclude the proof.
\end{proof}

What we have technically argued is an isomorphism of two groups as sets; the slightly more work to see that the maps $a_n$ and $b_n$ are homomorphisms is left to the reader. We then have from Theorem \ref{cooicohtriv} the following corollary:
\begin{cor} $\check{\mathrm{H}}^n(\varepsilon;\mathcal{A})\cong \mathsf{coh}^{\mathrm{II}}(n,A,\varepsilon)/\mathsf{triv}^{\mathrm{II}}(n,A,\varepsilon)$ for all ordinals $\varepsilon$, abelian groups $A$, and integers $n>0$.
\end{cor}

\section*{Acknowledgements}

This work expands on material first appearing in the author's PhD thesis. He would like to thank his advisor, Justin Tatch Moore, and the rest of his thesis committee, Jim West and Slawek Solecki, for their support, instruction, and patience throughout the course of its composition. The periodic stimulus of conversations with Stevo Todorcevic has also richly informed this work. The author would like finally to thank the referee for so exceptionally alert and thoughtful a reading of this text.

\printbibliography

\end{document}